\documentclass[10pt,reqno]{amsart}

\usepackage[a4paper,left=35mm,right=35mm,top=30mm,bottom=30mm,marginpar=25mm]{geometry}

\usepackage{amsmath}
\usepackage{amssymb}
\usepackage{amsthm}
\usepackage[utf8]{inputenc}
\usepackage{eurosym}
\usepackage{graphicx}
\usepackage{hyperref}
\usepackage{dsfont}
\usepackage{dsfont}
\usepackage{xcolor,colortbl}
\usepackage{comment}
\usepackage{bm}

\allowdisplaybreaks

\usepackage{hyperref}

\usepackage{ifthen}

\makeindex

\newcommand{\argmin}{\operatorname{argmin}}

\newcommand{\Rr}{{\mathbb{R}}}

\newcommand{\Ee}{{\mathds{E}}}

\newcommand{\Pp}{{\mathcal{P}}}

\newcommand{\bx}{{\bf x}}
\newcommand{\bepsilon}{{\bm{ \epsilon}}}

\newcommand{\bX}{{\bf X}}

\newcommand{\bP}{{\bf P}}

\newcommand{\bv}{{\bf v}}

\def\dx{{\rm d}x}

\def\leq{\leqslant}
\def\geq{\geqslant}

\numberwithin{equation}{section}

\newtheoremstyle{thmlemcorr}{10pt}{10pt}{\itshape}{}{\bfseries}{.}{10pt}{{\thmname{#1}\thmnumber{
#2}\thmnote{ (#3)}}}
\newtheoremstyle{thmlemcorr*}{10pt}{10pt}{\itshape}{}{\bfseries}{.}\newline{{\thmname{#1}\thmnumber{
#2}\thmnote{ (#3)}}}
\newtheoremstyle{defi}{10pt}{10pt}{\itshape}{}{\bfseries}{.}{10pt}{{\thmname{#1}\thmnumber{
#2}\thmnote{ (#3)}}}
\newtheoremstyle{remexample}{10pt}{10pt}{}{}{\bfseries}{.}{10pt}{{\thmname{#1}\thmnumber{
#2}\thmnote{ (#3)}}}
\newtheoremstyle{ass}{10pt}{10pt}{}{}{\bfseries}{.}{10pt}{{\thmname{#1}\thmnumber{
A#2}\thmnote{ (#3)}}}

\theoremstyle{thmlemcorr}
\newtheorem{theorem}{Theorem}
\numberwithin{theorem}{section}
\newtheorem{lemma}[theorem]{Lemma}

\newtheorem{proposition}[theorem]{Proposition}

\theoremstyle{thmlemcorr*}
\newtheorem{theorem*}{Theorem}
\newtheorem{lemma*}[theorem]{Lemma}
\newtheorem{corollary*}[theorem]{Corollary}
\newtheorem{proposition*}[theorem]{Proposition}
\newtheorem{problem*}[theorem]{Problem}
\newtheorem{conjecture*}[theorem]{Conjecture}

\theoremstyle{defi}

\newtheorem{hyp}{Assumption}

\newtheorem{problem}{Problem}

\theoremstyle{remexample}
\newtheorem{remark}[theorem]{Remark}

\theoremstyle{ass}

\usepackage{tikz}
\usetikzlibrary{arrows}
\usetikzlibrary{positioning}
\usepackage{bm}
\usepackage{caption}
\usepackage{subcaption}
\usepackage[linesnumbered,ruled]{algorithm2e}

\begin{document}

\title{Machine Learning architectures for price formation models}

\author{Diogo Gomes}\thanks{King Abdullah University of Science and Technology (KAUST), CEMSE Division, Thuwal 23955-6900.
Saudi Arabia.
e-mail: diogo.gomes@kaust.edu.sa.}
\author{Julian Gutierrez}\thanks{King Abdullah University of Science and Technology (KAUST), CEMSE Division, Thuwal 23955-6900.
Saudi Arabia.
e-mail: julian.gutierrezpineda@kaust.edu.sa.}
\author{Mathieu Lauri\`{e}re}\thanks{New York University Shanghai (NYU Shanghai), Shanghai 200122.
China.
e-mail: ml5197@nyu.edu.}




\thanks{
      The authors were partially supported by King Abdullah University of Science and Technology (KAUST) baseline funds and KAUST OSR-CRG2021-4674.
}
\date{\today}

\begin{abstract}
Here, we study machine learning (ML) architectures to solve a mean-field games (MFGs) system arising in price formation models. We formulate a training process that relies on a min-max characterization of the optimal control and price variables. Our main theoretical contribution is the development of a posteriori estimates as a tool to evaluate the convergence of the training process. We illustrate our results with numerical experiments for linear dynamics and both quadratic and non-quadratic models.
\end{abstract}

\maketitle

\section{introduction}

Here, we consider machine learning (ML) methods to numerically solve the mean-field games (MFGs) price formation model introduced in \cite{gomes2018mean}. This model describes the price $\varpi$ of a commodity with deterministic supply $Q$. This commodity is traded in a population of rational agents under a market-clearing condition on a finite time horizon $T>0$. The price formation problem is reduced to the following MFGs system.

\begin{problem} \label{problem:MFGs price problem GS} 
Given $m_0\in \mathcal{P}(\Rr)$, $H:\Rr^2\to \Rr$ differentiable in the second argument, $u_T:\Rr\to\Rr$, and $Q\in C^\infty([0,T])$, find $u,m: [0,T]\times \Rr \to \Rr$ and $\varpi:[0,T]\to \Rr$ satisfying  $m\geq 0$ and 
\begin{equation}\label{eq:MFG system}
\begin{cases}
-u_t +H(x,\varpi + u_x )=0 & [0,T]\times \Rr,
\\
u(T,x)=u_T(x) & x \in \Rr,
\\
m_t - \left(H_p(x,\varpi + u_x)m\right)_x =0 & [0,T]\times \Rr,
\\
m(0,x)=m_0(x) & x\in \Rr,
\\
-\int_{\Rr} H_p(x,\varpi + u_x)m \dx = Q(t) & t \in [0,T].
\end{cases}
\end{equation}
\end{problem}  
In the previous problem, the Hamiltonian $H$ is the Legendre transform of a Lagrangian $L$; that is,
\begin{equation}\label{eq:Legendre transform}
H(x,p)=\sup_{v\in\Rr} \left\{-p v - L(x,v)\right\}, \quad (x,p) \in \Rr^2,
\end{equation}
where $v\mapsto L(x,v)$ is convex for all $x\in \Rr$. The existence and uniqueness of solutions for \eqref{eq:MFG system} were obtained in \cite{gomes2018mean} under convexity assumptions for $u_T$ and $L$, and some further technical assumptions on $H$. In their framework, authors required $C^\infty([0,T])$ regularity of $Q$, however $C^1([0,T])$ regularity is enough for their fixed-point approach (see equation (20) in \cite{gomes2018mean}). The first equation in \eqref{eq:MFG system} is solved in the viscosity sense, and the second equation is solved in the distributional sense. Moreover, $u$ is Lipschitz continuous in $x$, and $\varpi$ is continuous. Furthermore, the linear-quadratic model admits semi-explicit solutions, as presented in \cite{gomes2021randomsupply} for the game with a finite population and random supply. 

Here, we present a method to approximate solutions to this price formation problem using ML tools. For that, we formulate \eqref{eq:MFG system} as a constrained minimization problem. Then, we follow an update rule analogous to the dual-ascent method in constrained optimization to optimize the ML parameters. We can verify the convergence of our method using a posteriori estimates, which are based on the Euler-Lagrange equation that characterizes our minimization problem. The optimal control problem described by \eqref{eq:MFG system} is briefly recalled in Section \ref{sec:The mfg price problem}, where we present a variational problem that motivates our numerical method. In Section \ref{sec:Assumptions}, we state the main assumptions for the MFGs price formation model. In addition to developing ML frameworks for price formation models, our main theoretical contribution is an a posteriori estimate that controls the difference between the solution of Problem \ref{problem:MFGs price problem GS} and an approximate solution to an Euler-Lagrange equation. These a posteriori estimates do not require the exact solution to be known, whose existence is guaranteed by the results in \cite{gomes2018mean}. Our result reads as follows:

\begin{theorem}\label{Prop:Main result}
Suppose that Assumptions \ref{hyp:L uniformly convex}, \ref{hyp:uT convex DS}, \ref{hyp:L separability DS}, \ref{hyp:Lipschitz}, and \ref{hyp:Lipschitz 2} hold. 
Let $(\bX,\bP)$ and $\varpi^N$ solve the $N$-player price formation problem (see \eqref{eq:HSystem N agent}), and let $(\tilde{\bX},\tilde{\bP})$ and $\tilde{\varpi}^N$ be an approximate solution to the $N$-player problem (see \eqref{eq:HSystem N agents penalized}) up to the error terms $\bepsilon:[0,T]\to \Rr^N$, $\bepsilon_T \in \Rr^N$, and $\epsilon_q:[0,T]\to\Rr$. Then, there exists $C(T,H,u_T,N)>0$ such that 
\begin{align}
\label{eq:EL discrete a posteriori}
	\| \varpi^N - \tilde{\varpi}^N \|_{L^2([0,T])}^2  &  \leq C(T,H,u_T,N) \bigg( \|\bepsilon\|_{L^2([0,T])}^2 + |\bepsilon_T|^2 + \|\epsilon_q \|_{L^2([0,T])}^2 \bigg).
\end{align}
\end{theorem}
In the above result, to be specific, $C(T,H,u_T,N)$ depends on the Lipschitz constants of $u'_T$ and of the partial derivatives of $H$. This result is proved in Section \ref{sec:A posteriori estimates} and provides a convergence criterion for any numerical approximation of the solution to Problem \ref{problem:MFGs price problem GS}. In particular, we show that for a sequence $\tilde{\varpi}^k$ and $\tilde{\bX}^k$, $k=1,\ldots$, obtained by ML techniques, such that $\bepsilon^k$, $\bepsilon_T^k$, and $\epsilon_q$ converge to zero, our approximations converge to the solution of the price formation problem.

The result in Theorem \ref{Prop:Main result} is in line with the study of stability of solutions w.r.t. perturbations in optimality conditions. This is a central property studied in the optimal control theory under the concept of metric regularity (see \cite{MetricRegBook}, Chapter 3) of solution maps. For instance, \cite{Malanowski2} studied constrained optimization problems depending on a parameter. They considered the associated Lagrange multiplier and Kuhn--Tucker condition and used a perturbed linear-quadratic problem to approximate the original optimization problem. They proved that neighborhoods of the parameter and the optimal pair (the minimizer and its associated multiplier) exist such that any parameter perturbation has an associated Kuhn-Tucker point whose distance from the optimal pair can be estimated using the parameter perturbation. Moreover, when sufficient optimality conditions hold, the associated Kuhn--Tucker point is the solution to the perturbed optimization problem. In \cite{Malanowski1}, using the results from \cite{Malanowski2}, authors studied control and state-constrained non-linear optimization problems with parameter dependence. They obtained Lipschitz continuity and directionally differentiability of the solutions w.r.t. the parameters. A linear-quadratic optimal control problem characterizes the directional derivatives. \cite{Quin2} introduced the concept of bi-metric regularity to study the effect of perturbations on solutions of linear optimal control problems. They obtained a H\"{o}lder-type stability result of an inclusion problem equivalent to the Pontryagin principle of the original optimization problem. In their study, a perturbation of the Pontryagin principle corresponded to the optimality condition of a perturbed optimization problem. \cite{Quin1} relied on the concept of bi-metric regularity introduced in \cite{Quin2} to study the stability of optimal control problems affine in the control. Using the Pontryagin principle as a generalized inclusion, they derived a partly linearized inclusion corresponding to the Pontryagin system of an auxiliary optimal control problem.  

Price formation has been studied in several contexts and with different techniques. In \cite{BS10} and \cite{BS02}, Stackelberg games to maximize the producer's revenue were used. A Cournot model was introduced in \cite{TBD20}, which specifies the price dynamics with noise from a Brownian motion and a jump process. A model with finitely many agents was presented in \cite{aid2020equilibrium}, where the demand includes common noise. The MFG approach for price formation has been adopted in several works. \cite{ATM19} obtains the spot price of an energy market as a function of the demand and trading rates. This model was extended in \cite{alasseur2021mfg} to include penalties at random times and random jumps. A major player in the market was studied in \cite{FTT20}. Using a linear-quadratic structure, the authors obtained a market price depending on the average position of the agents and the position of the major agent. The successive work \cite{feron2021price} proposed a MFG  model with common noise, with numerical results for the EPEX intraday electricity market. A MFG of optimal stopping was introduced in \cite{AidDumitrescuTankov2021} to model the transition from traditional to renewable means of energy production. \cite{JSF20} studied price formation for Solar Renewable Energy Certificates using a balance condition. McKean-Vlasov type equations were used to characterize price equilibrium of homogeneous markets in \cite{FT20}, and with a major player in \cite{fujii2021equilibrium}. The imbalance function approach we consider borrows ideas from  \cite{Gomes20164693}, where price formation was studied in a market of optimal investment and consumption. They obtained a MFG model coupled with the price variable, determined by a balance condition guaranteeing that all investments remain in the market. They provided numerical results for a linear model using a fixed-point iteration method. Common noise in the supply was considered in \cite{GoGuRi2021} for a MFG with a continuum of players, and in \cite{gomes2021randomsupply} for a market with finitely many agents.

Numerical methods for mean-field games were first introduced in \cite{DY} using finite-differences schemes and Newton-based methods (see \cite{achdou2021mean}, Chapter 4, and \cite{MFGNumerical} and \cite{LauriereEtAlSurvey} for a recent survey). Fourier series approximations were proposed in \cite{Yang2021}. Other methods include semi-Lagrangian schemes \cite{CarliniSilva2014}, fictitious play \cite{HADIKHANLOO2019369}, pseudo-spectral elements \cite{spectral}, and variational methods \cite{BenamouBrenierVarMFG}, \cite{Proximal2018}. However, the price formation MFG system \eqref{eq:MFG system} does not fit the above mentioned schemes because of the balance constraint. A novel approach to the numerical solution of \eqref{eq:MFG system} was proposed in \cite{PotentialPrice} using a reduction technique to obtain an equivalent variational problem. A semi-Lagrangian scheme is under study by our collaborators.

ML techniques have been used in several contexts. Early attempts to approximate differential equations solutions using neural networks (NN) can be traced to \cite{LagarisNNODE}, where the loss function is the differential equation. For optimal control problems, \cite{EffatiNNOC} trained a NN to fit the corresponding Pontryagin maximum principle equations. Since then, several variations have been proposed. For instance, regarding the architecture, the Deep Ritz method \cite{DeepRitz} introduced an architecture that adds residual connections between layers of the NN. This method outperforms the standard finite difference method for the Poisson equation in two dimensions. \cite{han2016deep} combines NN outputs with Monte-Carlo sampling to solve an energy allocation problem, for which the dynamic programming principle is expensive since the problem is high-dimensional. A supervised learning approach was adopted by \cite{SanchezSupervised} to train a NN model for landing problems against data of optimal trajectories. More recently, automatic differentiation of NM was used in \cite{bensoussan2021valuegradient} to solve optimal control problems. They minimized a loss functional that depends on the gradient of the value function, which was obtained by automatic differentiation of the NN representing the value function. \cite{HanHu2021} presented a numerical study using NN to solve stochastic optimal control problems with delay. The authors parametrized the controls using feed-forward and recurrent neural networks and a loss function corresponding to the discrete cost criteria.

For MFGs and ML, \cite{ruthotto2020machine} and \cite{Line2024713118} provided a ML approach to solve potential MFGs and mean-field control problems. They used a loss function that penalizes deviations from the Hamilton-Jacobi equation, which, for potential games, completely characterizes the model. A min-max problem was used in \cite{campbell2021deep}  to motivate a training algorithm for NN to approximate the solution of a principal-agent MFGs arising in renewable energy certificates models. \cite{cao2021connecting} studied the connection between generative adversarial networks, mean-field games, and optimal transport. The case of common-noise effects in mean-field equilibrium was investigated by \cite{min2021signatured} using rough path theory and deep learning techniques. The analysis of convergence of machine learning algorithms to solve mean-field games and control problems was presented in \cite{carmona2021convergenceergo} for ergodic problems and in \cite{carmona2021convergence} for finite horizon problems. Mean-field control with delay has been solved using recurrent neural networks in \cite{fouque2020deep}. The interested reader is referred to e.g. \cite{carmonalauriere2021handbookdeep}, for a survey of deep learning methods applied to MFGs and mean-field control problems. However, the problem in \eqref{eq:MFG system} is quite distinct from the previous ones. One PDE has a terminal condition, whereas the other is subject to an initial condition. Moreover, the coupling between them is given by an integral constraint.

The success of ML techniques applied to differential equations and optimal control problems depends on two aspects. First, the selection of a ML architecture implies the choice of a class of functions. This class relies on hyper-parameters, such as the activation function, the number of layers, and the number of neurons per layer. The convergence properties of this class of functions, when implemented to approximate solutions of differential equations and optimal control problems, remain a challenging issue. In this direction, the works \cite{arora2018understanding} and \cite{XuFiniteNN} present convergence results for NN with the rectified linear unit activation function when approximating piece-wise linear functions and functions in $L^q(\Rr)$, and  they provide bounds for the selection of the NN hyper-parameters. 
We introduce the ML framework for the price model in Section \ref{sec:NN for price}, where we consider two different architectures to approximate the solution of problem \eqref{eq:MFG system} by a NN. First, we consider the feed-forward NN, which is a standard tool in ML due to its simple structure. This architecture passes information among time steps using the output variable. The second is a recurrent architecture, which allows passing information among time steps using the output and a hidden state variable. This architecture is a standard tool in natural language processing. Our selection is guided by further extensions of our method to price formation models with common noise, which requires progressive measurability.

The second aspect when implementing ML techniques is the training algorithm used to optimize the NN parameters. In contrast to traditional optimization techniques, the ML approach constraints a best approximation to the class of functions defined by the selected architecture. Thus, finding the best approximation in this class depends on the convergence properties of the training algorithm. For instance, as presented in \cite{chen2018training}, the standard generative adversarial networks (GANs) training procedure is related to the sub-gradient method applied to a convex optimization problem. Thus, the observed convergence of GANs training is explained by analytical results in convex analysis. In the same spirit, the training algorithm we present in Section \ref{sec:training algorithm} can be regarded as an adversarial-like training between minimizing a cost functional and maximizing the penalization for deviations from a balance condition. Following this approach, we update the NN parameters the same way as primal and dual variables are updated in the dual ascent method for convex optimization with constraints. In particular, instead of updating the primal variable by the exact solution of a minimization problem, we adopt a variation borrowed from the Arrow-Hurwicz-Uzawa algorithm, which updates the primal variable by moving in the descent direction of the gradient. While global convergence results for training algorithms are difficult, we develop an alternative approach. We rely on our main result, the a posteriori estimate, to evaluate the convergence of the NN approximation to the solution of the price formation problem. 

We conclude our presentation with numerical results in Section \ref{sec: Numerical results}, where we implement our method in the linear-quadratic model. The results show that both architectures introduced in Section \ref{sec:NN for price} provide an accurate approximation of the solution to the price formation model.


\section{The mfg price problem}
\label{sec:The mfg price problem}

Here, we briefly recall the optimization problem that leads to the formulation of Problem \ref{problem:MFGs price problem GS}, and we relate this problem to a game with a finite number of players. We use the relation between the discrete and the continuous models to formulate the variational problem whose numerical solution, using ML techniques, approximates the price $\varpi$ in Problem \ref{problem:MFGs price problem GS}.

First, we recall the derivation of Problem \ref{problem:MFGs price problem GS}. At time $t=0$, a representative player owns a quantity $x_0\in \Rr$ of the commodity. By selecting his trading rate $v:[0,T]\to \Rr$, this player minimizes the cost functional
\begin{equation*}
v\mapsto \int_0^T \left(L(X(t),v(t)) + \varpi(t) v(t) \right) \; dt + u_T\left(X(T)\right),
\end{equation*}
where $X$ solves the ODE:
\begin{equation}
\label{eq:Agent dynamics}
\begin{cases}
\dot{X}(t)=v(t), & t\in[0,T],
\\
X(0)=x_0.
\end{cases}
\end{equation}
When the price $\varpi$ is known, the optimal trading rate $v^*$ can be obtained through the value function $u$, which solves the first equation in \eqref{eq:MFG system}. At points of differentiability, we have
\begin{equation}\label{eq:optimal feedback}
	v^*(t,x)=-H_p(x,\varpi(t)+u_x(t,x)).
\end{equation}
Because all players select their optimal strategy, an initial distribution $m_0 \in \Pp(\Rr)$ evolves in time according to the second equation in \eqref{eq:MFG system}. In turn, the balance constraint, the third equation in \eqref{eq:MFG system}, requires the aggregated demand and the supply $Q$ to match. Notice that $\varpi$ is the coupling term in \eqref{eq:MFG system}. Thus, it is enough to approximate $\varpi$ to decouple the system \eqref{eq:MFG system} and approximate the entire solution of Problem \ref{problem:MFGs price problem GS}.

Relying on \eqref{eq:optimal feedback}, in the following, we consider feedback controls $v=v(t,x)$. Given $\tilde{\varpi}:[0,T]\to \Rr$ and $\tilde{v}:[0,T]\times \Rr \to \Rr$, let $\tilde{m}:[0,T]\times \Rr \to \Rr$ solve
\begin{equation}\label{eq:Auxiliary transport eq}
\begin{cases}
\tilde{m}_t + \left( \tilde{v}(t,x)\tilde{m}\right)_x =0 & [0,T]\times \Rr,
\\
\tilde{m}(0,x)=m_0(x) & x\in \Rr.
\end{cases}
\end{equation}
Define the imbalance function, $\mathcal{I}:[0,T] \to \Rr$, by
\begin{equation}\label{def:imbalance function}
	\mathcal{I}(t)=\tilde{\varpi}(t)\left( \int_{\Rr} \tilde{v}(t,x) \tilde{m}(t,x) \dx - Q(t)\right),\quad t \in [0,T].
\end{equation}
The imbalance function measures deviations from the balance constraint proportionally to $\tilde{\varpi}$. Notice that $(u,m,\varpi)$ solving Problem \ref{problem:MFGs price problem GS} satisfies $\mathcal{I} \equiv 0$ for $\tilde{\varpi}=\varpi$ and $\tilde{v}=v^*$, for $v^*$ given by \eqref{eq:optimal feedback}. 

Next, we consider the relation between $v^*$ and the optimal trading rate for a finite-player game. Let $N$ be the number of players, and let $x_0^i \in \Rr$, $i=1,\ldots,N$, be a sample of initial positions drawn according to $m_0$. 
Let $v^i: [0,T] \times \mathbb{R} \to \mathbb{R}$, $1\leq i \leq  N$, be feedback controls, and let $X^i$ solve \eqref{eq:Agent dynamics} for $v^i$ and initial condition $x_0^i$. Given $\tilde{\varpi}^N:[0,T]\to \Rr$, consider the functional
\begin{equation}\label{eq:Functional per agent}
v\mapsto \int_0^T \big( L(X(t),v(t,X(t))) + \tilde{\varpi}^N(t) \left( v(t,X(t)) - Q(t) \right) \big) dt + u_T\left(X(T)\right).
\end{equation}
Replacing $\tilde{m}$ by the empirical measure 
\begin{equation*}
	\tilde{m}^N(t,x)=\frac{1}{N}\sum_{i=1}^N \delta_{X^i(t)}(x), \quad t\in[0,T], \; x\in \Rr,
\end{equation*}
the imbalance function \eqref{def:imbalance function} becomes
\begin{equation}\label{eq:Balance condition}
	\mathcal{I}^N(t) = \tilde{\varpi}^N(t)\left( \frac{1}{N} \sum_{i=1}^N v^i(t,X^i(t)) - Q(t)\right), \quad  t \in [0,T].
\end{equation}
Thus, the existence of ${v^*}^i$ minimizing \eqref{eq:Functional per agent} for $1\leq i \leq N$ is equivalent to the existence of 
\[
	\bv^*=({v^*}^1,\ldots,{v^*}^N)
\]
minimizing the functional
\begin{gather*}
\bv \mapsto \frac{1}{N} \sum_{i=1}^N \int_0^T \big(L(X^i(t),v^i(t)) + \mathcal{I}^N(t) \big) \; dt + u_T\left(X^i(T)\right), 
\end{gather*}
where $X^i$ is controlled by $v^i$. 
In the $N$-player price formation model (see \cite{SummerCamp2019}), there exists $\varpi^N:[0,T]\to\Rr$ such that the functional
\begin{gather}\label{eq:Functional N agent}
\bv \mapsto \frac{1}{N} \sum_{i=1}^N \int_0^T \big( L(X^i(t),v^i(t)) + \varpi^N(t) v^i(t) \big)\; dt + u_T\left(X^i(T)\right) 
\end{gather}
has a minimizer $\bv^*$ in the set of admissible controls
\[
	\mathcal{A}^N=\left\{ \bv:[0,T]\to\Rr^N:\; \frac{1}{N}\sum\limits_{i=1}^N v^i(t)-Q(t)=0, \; t\in[0,T]\right\}.
\]
Thus, $\mathcal{I}^N \equiv 0$ for $\tilde{\varpi}^N=\varpi^N$ and $\bv = \bv^*$. Moreover, adding the constant $-\int_0^T \varpi^N(t)Q(t) dt$ to the functional in \eqref{eq:Functional N agent}, we consider the functional
\begin{gather}\label{eq:Saddle points functional}
(\bv,\tilde{\varpi}^N) \mapsto \frac{1}{N} \sum_{i=1}^N \int_0^T \big(L(X^i(t),v^i(t)) + \tilde{\varpi}^N(t)\left(v^i(t) - Q(t) \right)\big) dt + u_T\left(X^i(T)\right),
\end{gather}
which shows that $\varpi^N$ is the Lagrange multiplier associated with the discrete balance constraint defining $\mathcal{A}^N$. Therefore, $(\bv^*,\varpi^N)$ is a saddle point of the functional in \eqref{eq:Saddle points functional}.
%
Notice that the $N$-player price formation problem solves Problem \ref{problem:MFGs price problem GS} with singular initial data 
\[
	m_0 = \frac{1}{N}\sum_{i=1}^N \delta_{x^i_0}
\]
when the players consider the price as given and do not anticipate their own influence on it. 
Moreover, if the initial data converges to a continuous distribution in $\Rr$ as $N\to \infty$, we expect the convergence of $\varpi^N$ to $\varpi$ and ${v^*}^i$ to $v^*$, where $v^*$ is given by \eqref{eq:optimal feedback}. The convergence was proved for the linear-quadratic model in \cite{gomes2021randomsupply} under the assumption that a linear stochastic differential equation describes $Q$. In particular, convergence follows for the deterministic supply with mean-reverting dynamics we consider in Section \ref{sec: Numerical results}. 

Based on the previous results, to approximate $\varpi$ solving Problem \ref{problem:MFGs price problem GS}, we consider the following unconstrained variational problem
\begin{gather}\label{eq:Nagent minmax problem}
\sup_{\tilde{\varpi}^N} \inf_{\bv} \frac{1}{N} \sum_{i=1}^N \int_0^T \big(L(X^i(t),v^i(t)) + \tilde{\varpi}^N(t)\left( v^i(t) - Q(t) \right) \big)\; dt + u_T\left(X^i(T)\right).
\end{gather}
Here, the infimum is over all controls and not only those in $\mathcal{A}^N$. Notice that \eqref{eq:Nagent minmax problem} penalizes deviations from $\mathcal{A}^N$ through $\tilde{\varpi}^N$ using the imbalance function \eqref{eq:Balance condition} in accordance with \eqref{def:imbalance function}.

Lastly, given $\tilde{m} \in C([0,T];\Pp(\Rr))$, we define the following auxiliary process for the mean quantity:
\begin{equation*}
	\overline{X}(t;\tilde{m}) = \int_{\Rr} x \tilde{m}(t,x)dx, \quad t \in [0,T].
\end{equation*}
When $m$ belongs to the triplet $(u,m,\varpi)$ solving \eqref{eq:MFG system}, the dynamics of $\overline{X}(\cdot;m)$ are given by the second and third equations in \eqref{eq:MFG system} because
\begin{equation}\label{eq:Xbar dynamics}
	\dot{\overline{X}}(t;m) = \int_{\Rr} x m_t dx=\int_{\Rr} x \left(H_p(x,\varpi + u_x)m\right)_x dx = Q(t), \quad t \in [0,T],
\end{equation}
provided $\int_{\Rr} x m_0(x) dx < \infty$. Analogously, for the empirical measure $\tilde{m}^N$, we consider
\begin{equation*}
	\overline{X}(t;\tilde{m}^N) = \frac{1}{N} \sum_{i=1}^N X^i(t), \quad t \in [0,T].
\end{equation*}
As shown in \cite{gomes2018mean} and \cite{gomes2021randomsupply}, the previous auxiliary processes simplify the computation of the price $\varpi$ in the linear-quadratic setting. This setting will be considered in Section \ref{sec: Numerical results} to validate our numerical results.

\section{Assumptions}
\label{sec:Assumptions}
In this section, we state the assumptions that we use to obtain the a posteriori estimates. These estimates are obtained using an Euler-Lagrange equation that characterizes trajectories that solve the optimal control problem \eqref{eq:MFG system} describes. 

The first two assumptions are used to prove estimates on the error of sub-optimal trajectories of agents. These estimates depend on both an Euler-Lagrange equation and the error of the price approximation. These assumptions imply the convexity conditions that were required in \cite{gomes2018mean}. We consider a Lagrangian $L:\Rr^2 \to \Rr$ that satisfies the following convexity assumption.

\begin{hyp}\label{hyp:L uniformly convex}
The Lagrangian $L:\Rr^2 \to \Rr$ is uniformly convex in $(x,v)$; that is, there exists $\beta>0$ such that $(x,v)\mapsto L(x,v)-\frac{\beta}{2}\|(x,v)\|^2$ is convex. Furthermore, $L\in C^2(\Rr^2)$.
\end{hyp}

Analogously, we require convexity for the terminal cost, which guarantees the uniqueness of solutions to the price formation problem (see \cite{gomes2018mean}).

\begin{hyp}\label{hyp:uT convex DS}
	The terminal cost $u_T:\Rr \to \Rr$ is uniformly convex; that is, there exists $\gamma_T>0$ such that $x\mapsto u_T(x) - \frac{\gamma_T}{2}x^2$ is convex.
\end{hyp}

%
%
%


The following assumption was considered in \cite{gomes2018mean} to obtain, under further technical conditions, the existence of solutions to the price formation problem. This assumption simplifies the proof of the a posteriori estimates.

\begin{hyp}\label{hyp:L separability DS}
	The Lagrangian $L$ is separable; that is, 
	\begin{equation*}
		L(x,v)=\ell(v)+V(x),
	\end{equation*} 
where $V\in C^2(\Rr)$ is convex and bounded from below.
\end{hyp}


Notice that, under the previous assumption, $H$, defined by \eqref{eq:Legendre transform}, is separable as well; that is
\[
	H(x,p)=\mathcal{H}(p) - V(x),
\]
where $\mathcal{H}$ is the Legendre transform of $\ell$. In this case, Assumption \ref{hyp:L uniformly convex} holds when, for instance, both $\ell$ and $V$ are uniformly convex. 

The following assumption is used to obtain bounds on the error of the price approximation in terms of an Euler-Lagrange equation. These assumptions are a standard tool in convex optimization (see \cite[Chapter 1]{Ryu2015APO}).

\begin{hyp}\label{hyp:Lipschitz}
The Hamiltonian $H$ and the terminal cost $u_T$ are differentiable, and their derivatives $H_x$ and $u'_T$ are Lipschitz continuous. 
\end{hyp}

\begin{hyp}\label{hyp:Lipschitz 2}
$H_p$ is Lipschitz continuous. 
\end{hyp}

A Lagrangian and terminal cost satisfying the previous assumptions is the quadratic model. For $\kappa,\zeta \in \Rr$, $\eta,\gamma\geq 0$, and $c>0$, let
\[
	L(x,v)=\frac{\eta}{2} \left( x - \kappa\right)^2 + \frac{c}{2} v^2, \quad \mbox{and} \quad u_T\left(x\right) = \frac{\gamma}{2}\left(x-\zeta\right)^2.
\]
With the previous selection, agents have a preferred state $\kappa$ during the game and finish close to $\zeta$ while they are charged proportionally to the trading rate. In Section \ref{sec: Numerical results}, we use the quadratic framework to illustrate our results numerically.



\section{A posteriori estimates}
\label{sec:A posteriori estimates}
In this section, we consider a posteriori estimates using the first-order characterization of solutions associated with the optimization problem that each agent solves, as introduced in Section \ref{sec:The mfg price problem}. This estimate will be used to assess the convergence of the approximate solutions obtained using NN.  

The Euler-Lagrange equation of \eqref{eq:Functional per agent} is
\begin{equation}\label{eq:EL single agent}
	\begin{cases}
	L_x(X(t),v(t))-\frac{d}{dt} \left( L_v(X(t),v(t)) + \varpi(t)\right) =0 &  t \in [0,T],
	\\
	L_v(X(T),v(T)) +\varpi(T) + u_T'\left(X(T)\right)=0.
	\end{cases}
\end{equation}
For $P(t) := - \left( L_v(X(t),v(t)) + \varpi(t) \right)$, the previous equation, together with equation \eqref{eq:Agent dynamics}, are equivalent to the following Hamiltonian system for $(X,P)$
\begin{equation*}
	\begin{cases}
	\dot{P}(t) = H_x(X(t),P(t)+\varpi(t)) & t \in [0,T],
	\\
	P(T) = u'_T(X(T)),
	\\
	\dot{X}(t) = -H_p(X(t),P(t)+\varpi(t)) &  t \in [0,T],
	\\
	X(0) = x_0,
	\end{cases}
\end{equation*}
where $H$ is given by \eqref{eq:Legendre transform}. The solution $(u,m,\varpi)$ of \eqref{eq:MFG system} defines $v^*:[0,T]\times \Rr\to \Rr$, according to \eqref{eq:optimal feedback}, and the pair $(X^*,v^*)$ given by \eqref{eq:Agent dynamics} solves \eqref{eq:EL single agent}. 
Let $\tilde{v}:[0,T]\times \Rr \to \Rr$ and $\tilde{\varpi}:[0,T]\to \Rr$ satisfy 
\begin{equation}\label{eq:EL single agent residual}
	\begin{cases}
	L_x(\tilde{X}(t),\tilde{v}(t,\tilde{X}(t)))-\frac{d}{dt} \left( L_v(\tilde{X}(t),\tilde{v}(t,\tilde{X}(t))) + \tilde{\varpi}(t)\right) = \epsilon(t) &  t \in [0,T],
	\\
	L_v(\tilde{X}(T),\tilde{v}(T,\tilde{X}(T))) +\tilde{\varpi}(T) + u_T'\left(\tilde{X}(T)\right)=\epsilon_T,
	\end{cases}
\end{equation}
where $\tilde{X}$ satisfies \eqref{eq:Agent dynamics} with $\tilde{v}$, $\epsilon:[0,T]\to \Rr$, and $\epsilon_T \in \Rr$. We regard $\tilde{v}$ and $\tilde{\varpi}$ as perturbations of $v^*$ and $\varpi$, respectively, and $\epsilon$ and $\epsilon_T$ as residuals incurred by the perturbations in \eqref{eq:EL single agent}.
For $\tilde{P}(t) := - \left( L_v(\tilde{X}(t),\tilde{v}(t,\tilde{X}(t))) + \tilde{\varpi}(t) \right)$, \eqref{eq:EL single agent residual} is equivalent to
\begin{equation*}
	\begin{cases}
	\dot{\tilde{P}}(t) = H_x(\tilde{X}(t),\tilde{P}(t)+\tilde{\varpi}(t))+\epsilon(t) & t \in [0,T],
	\\
	\tilde{P}(T) = u'_T(\tilde{X}(T))-\epsilon_T,
	\\
	\dot{\tilde{X}}(t) = -H_p(\tilde{X}(t),\tilde{P}(t)+\tilde{\varpi}(t)) &  t \in [0,T],
	\\
	X(0) = x_0.
	\end{cases}
\end{equation*}
Notice that no residual appears in the equation for $\dot{\tilde{X}}$ because $\tilde{X}$ is driven by $\tilde{v}$, according to \eqref{eq:Agent dynamics}. We study how far the solutions of \eqref{eq:EL single agent residual} are from the solutions of \eqref{eq:EL single agent} in terms of the residuals $\epsilon$ and $\epsilon_T$. We can estimate the difference between $v^*$ and $\tilde{v}$ in terms of the residuals $\epsilon$ and $\epsilon_T$, and the difference between $\varpi$ and $\tilde{\varpi}$, as we show next.

\begin{proposition}\label{Prop:FirstBound}
Suppose that Assumptions \ref{hyp:L uniformly convex} and \ref{hyp:uT convex DS} hold. Let $(u,m,\varpi)$ solve \eqref{eq:MFG system} and assume that $v^*$ is given by \eqref{eq:optimal feedback}. Let $\tilde{\varpi},\tilde{v}:[0,T]\to\Rr$ satisfy \eqref{eq:EL single agent residual} for some $\epsilon:[0,T]\to \Rr$ and $\epsilon_T \in \Rr$, where $\tilde{X}$ is given by $\tilde{v}$ through \eqref{eq:Agent dynamics}. Then, there exists $C(\beta)>0$ such that 
\begin{align*}
	& \|X^*-\tilde{X}\|^2_{L^2([0,T])} + \|v^*-\tilde{v}\|^2_{L^2([0,T])} 
	\\
	& \leq C \left( \|\varpi-\tilde{\varpi}\|^2_{L^2([0,T])} +  \|\epsilon\|^2_{L^2([0,T])} + T (\epsilon_T)^2 \right).
\end{align*}
\end{proposition}
\begin{proof}
Let
\begin{align*}
	L_x(t)=L_x\left(X^*(t),v^*(t)\right), \quad L_v(t)=L_v\left(X^*(t),v^*(t)\right),
	\\
	\tilde{L}_x(t)=L_x\left(\tilde{X}(t),\tilde{v}(t) \right), \quad \tilde{L}_v(t)=L_v\left(\tilde{X}(t),\tilde{v}(t) \right).
\end{align*}
By Assumption \ref{hyp:L uniformly convex}
, and using \eqref{eq:EL single agent} and \eqref{eq:EL single agent residual}, we have
\begin{align*}
	& \beta \left(\left( X^*(t)-\tilde{X}(t) \right)^2 +\left( v^*(t)-\tilde{v}(t) \right)^2\right) \nonumber
	\\
	& \leq \left(L_x(t) - \tilde{L}_x(t)\right)\left(X^*(t) - \tilde{X}(t) \right) + \left(L_v(t) - \tilde{L}_v(t)\right)\left(v^*(t) - \tilde{v}(t) \right) \nonumber
	\\
	& = \tfrac{d}{dt}\left[\left( L_v(t) + \varpi(t) - \left(\tilde{L}_v(t) + \tilde{\varpi}(t) - \epsilon_T \right) \right) \left(X^*(t) - \tilde{X}(t) \right) \right] \nonumber
	\\
	& \quad - \left(\varpi(t) - \tilde{\varpi}(t) + \epsilon_T \right) \left(v^*(t) - \tilde{v}(t) \right) - \epsilon(t) \left(X^*(t) - \tilde{X}(t) \right) .
\end{align*}
Integrating the previous on $[0,T]$, using the terminal condition in \eqref{eq:EL single agent residual}, the initial condition in \eqref{eq:Agent dynamics}, and Assumption \ref{hyp:uT convex DS}
, we obtain 
\begin{align*}
	& \beta \int_0^T \left( X^*(t)-\tilde{X}(t) \right)^2 +\left( v^*(t)-\tilde{v}(t) \right)^2 dt
	\\
	& \leq -\left(u'_T\left(X^*(T)\right) - u'_T\left(\tilde{X}(T)\right)  \right) \left(X^*(T) - \tilde{X}(T)\right)
	\\
	& \quad + \int_0^T \left( \varpi(t)-\tilde{\varpi}(t) +\epsilon_T \right) \left(\tilde{v}(t)-v^*(t) \right) + \epsilon(t) \left(\tilde{X}(t)-X^*(t) \right) dt
	\\
	&\leq \int_0^T \left( \varpi(t)-\tilde{\varpi}(t) + \epsilon_T \right) \left(\tilde{v}(t)-v^*(t) \right) + \epsilon(t) \left(\tilde{X}(t)-X^*(t) \right) dt
	\\
	& \leq \int_0^T \delta \left( \left( \varpi(t)-\tilde{\varpi}(t)\right)^2 + (\epsilon_T)^2 + (\epsilon(t))^2\right) + \tfrac{1}{2\delta} \left( \left(\tilde{X}(t)-X^*(t) \right)^2 + \left(\tilde{v}(t)-v^*(t) \right)^2   \right) dt,
\end{align*}
where $\delta >0$ is to be chosen. Taking $\delta>\tfrac{1}{2 \beta}$ in the previous, we get
\begin{align*}
	& \int_0^T \left( X^*(t)-\tilde{X}(t) \right)^2 +\left( v^*(t)-\tilde{v}(t) \right)^2 dt 
	\\
	& \leq \tfrac{2\delta^2}{2\beta \delta -1} \int_0^T \left( \left( \varpi(t)-\tilde{\varpi}(t)\right)^2 + (\epsilon_T)^2 + (\epsilon(t))^2\right) dt,
\end{align*} 
which concludes the proof.
\end{proof}

The first-order condition \eqref{eq:EL single agent} characterizes the optimal vector field $v^*$ given by \eqref{eq:optimal feedback} derived from the minimization problem that a representative player solves. However, \eqref{eq:EL single agent} does not characterize the balance condition that holds between $v^*$ and the supply. Therefore, we consider the residual in the balance condition introduced by the perturbations $\tilde{v}$ and $\tilde{\varpi}$. Let $\tilde{m}$ solve \eqref{eq:Auxiliary transport eq} with $\tilde{v}$. Then, 
\begin{equation}\label{eq:Balance residual}
	\int_{\Rr} \tilde{v}(t,x)\tilde{m}(t,x) \dx = Q(t) + \epsilon_q(t), \quad t \in [0,T],
\end{equation}
for some $\epsilon_q: [0,T]\to \Rr$. Using the residuals $\epsilon$, $\epsilon_T$, and $\epsilon_q$, we estimate the difference between both $\tilde{v}$ and $\tilde{\varpi}$ solving \eqref{eq:EL single agent residual} from $v^*$ and $\varpi$ solving \eqref{eq:EL single agent}. To this end, we adopt the particle approximation approach introduced in Section \ref{sec:The mfg price problem} with the finite-player game. Given $x_0^i \in \Rr$, $i=1,\ldots,N$, the Hamiltonian system of the $N$-player price formation model (see \eqref{eq:Functional N agent}) is
\begin{equation}
\label{eq:HSystem N agent}
	\begin{cases}
	\dot{P}^i(t) = H_x(X^i(t),P^i(t)+\varpi^N(t)) ,
	\\
	P^i(T) = u'_T(X^i(T)),
	\\
	\dot{X}^i(t) = -H_p(X^i(t),P^i(t)+\varpi^N(t)) ,
	\\
	X^i(0) = x^i_0,
	\\
	\frac{1}{N}\sum\limits_{i=1}^N -H_p(X^i(t),P^i(t)+\varpi^N(t)) =Q(t), 
	\end{cases} t \in [0,T],
\end{equation} 
where $P^i(t) := - \left( L_v(X^i(t),v^i(t)) + \varpi^N(t) \right)$ for $1 \leq i \leq N$. Let $\tilde{\varpi}^N$ and $(\tilde{X}^i,\tilde{P}^i)$ solve
\begin{equation}\label{eq:HSystem N agents penalized}
	\begin{cases}
	\dot{\tilde{P}}^i(t) = H_x(\tilde{X}^i(t),\tilde{P}^i(t)+\tilde{\varpi}^N(t))+\epsilon^i(t),
	\\
	\tilde{P}^i(T) = u'_T(\tilde{X}^i(T))-\epsilon^i_T,
	\\
	\dot{\tilde{X}}^i(t) = -H_p(\tilde{X}^i(t),\tilde{P}^i(t)+\tilde{\varpi}^N(t)) ,
	\\
	X^i(0) = x^i_0,
	\\
	\frac{1}{N}\sum\limits_{i=1}^N -H_p(\tilde{X}^i(t),\tilde{P}^i(t)+\tilde{\varpi}^N(t)) =Q(t)+\epsilon_q(t), 
	\end{cases} t \in [0,T],
\end{equation}
for $1 \leq i \leq N$. Let
\begin{align*}
	\tilde{\bX}=(\tilde{X}^1,\ldots,\tilde{X}^N), \quad \bX=(X^1,\ldots,X^N), \quad \bepsilon = (\epsilon^1 , \ldots ,\epsilon^N), 
	\\
	\tilde{\bP} = (\tilde{P}^1,\ldots , \tilde{P}^N), \quad \bP=(P^1,\ldots , P^N), \quad \bepsilon_T=(\epsilon^1_T,\ldots,\epsilon^N_T).
\end{align*}
Because $\tilde{X}$ and $X$ are driven by $\tilde{v}$ and $v^*$, respectively, studying the difference between $\tilde{v}$ and $v^*$ is equivalent to studying the difference between $\tilde{X}$ and $X$. In the following lemma, we estimate the distance of the perturbation $\tilde{\bP}$ in \eqref{eq:HSystem N agents penalized} from $\bP$ solving \eqref{eq:HSystem N agent}.

\begin{lemma}\label{lem:Ps bound}
Suppose that Assumptions \ref{hyp:L separability DS} and \ref{hyp:Lipschitz} hold. Let $(\bX,\bP)$ and $\varpi^N$ solve \eqref{eq:HSystem N agent}, and let $(\tilde{\bX},\tilde{\bP})$ and $\tilde{\varpi}^N$ satisfy \eqref{eq:HSystem N agents penalized} for some $\bepsilon:[0,T]\to \Rr^N$, $\bepsilon_T \in \Rr^N$, and $\epsilon_q:[0,T]\to\Rr$. Then
\begin{align}\label{eq:ABound 2}
	\|P^i - \tilde{P}^i\|_{L^2([0,T])}^2 & \leq 4T \bigg( (\mbox{Lip}(u_T'))^2 (X^i(T) - \tilde{X}^i(T))^2 + T (\mbox{Lip}(H_x))^2 \|X^i - \tilde{X}^i\|_{L^2([0,T])}^2  \nonumber
	\\
	& \quad\quad \quad \quad + (\epsilon^i_T)^2 + T\|\epsilon^i\|^2_2 \bigg)
\end{align}
for $1\leq i \leq N$.
\end{lemma}
\begin{proof}
Using \eqref{eq:HSystem N agent} and \eqref{eq:HSystem N agents penalized}, we have
\begin{align*}
	P^i(t) - \tilde{P}^i(t) & = u_T'(X^i(T))-u_T'(\tilde{X}^i(T)) -\epsilon_T^i
	\\
	& \quad + \int_t^T \bigg( \epsilon^i(s) - \left( H_x(X^i,P^i + \varpi^N) -H_x(\tilde{X}^i,\tilde{P}^i +\tilde{\varpi}^N) \right) \bigg) ds,
\end{align*}
for $1\leq i \leq N$. From the previous identity, and using Assumptions \ref{hyp:L separability DS} and \ref{hyp:Lipschitz}, we get
\begin{align*}
	|P^i(t) - \tilde{P}^i(t)| & \leq \mbox{Lip}(u_T') |X^i(T)-\tilde{X}^i(T)| + |\epsilon_T^i| + \int_0^T \bigg( |\epsilon^i(s)| + \mbox{Lip}(H_x) |X^i-\tilde{X}^i| \bigg) ds
\end{align*}
for $1\leq i \leq N$. Taking squares in the previous and integrating on $[0,T]$, we obtain the bound \eqref{eq:ABound 2} as stated.
\end{proof}

Next, using the previous lemma, we bound the distance of $(\tilde{\bX},\tilde{\bP})$ and $\tilde{\varpi}^N$ in \eqref{eq:HSystem N agents penalized} from $(\bX,\bP)$ and $\varpi^N$ solving \eqref{eq:HSystem N agent}.

\begin{proposition}\label{Prop:PricevsELError}
Suppose that Assumptions \ref{hyp:L uniformly convex}, \ref{hyp:uT convex DS}, \ref{hyp:L separability DS}, and \ref{hyp:Lipschitz} hold. 
Let $(\bX,\bP)$ and $\varpi^N$ solve \eqref{eq:HSystem N agent}, and let $(\tilde{\bX},\tilde{\bP})$ and $\tilde{\varpi}^N$ satisfy \eqref{eq:HSystem N agents penalized} for some $\bepsilon:[0,T]\to \Rr^N$, $\bepsilon_T \in \Rr^N$, and $\epsilon_q:[0,T]\to\Rr$. Then, there exists $C(T,H,u_T,N)>0$ such that 
\begin{align*}
	& \|\bP + \mathds{1}\varpi^N - (\tilde{\bP} + \mathds{1} \tilde{\varpi}^N)\|_{L^2([0,T])}^2 + \| \bX - \tilde{\bX} \|_{L^2([0,T])}^2  
	\\
	& \leq C \left( \|\bepsilon\|_{L^2([0,T])}^2 + |\bepsilon_T|^2 +\|\epsilon_q \|_{L^2([0,T])}^2 \right),
\end{align*}
where $\mathds{1} = (1,\ldots,1)\in \Rr^N$. 
\end{proposition}
In the above result, to be specific, $C(T,H,u_T,N)$ depends on the Lipschitz constants of $u'_T$ and of the partial derivative of $H$ w.r.t. $x$.
\begin{proof}
Let us write
\[
	\|\cdot\|_{L^2([0,T])} = \|\cdot \|_2.
\]
By Assumptions \ref{hyp:L uniformly convex} and \ref{hyp:L separability DS}, $H$ given by \eqref{eq:Legendre transform} is uniformly convex in $p$; that is,
\[
	\frac{\gamma_p}{2}(p-q)^2 \leq \left( H_p(x,p) - H_p(x,q) \right)(p-q)
\]
for some $\gamma_p >0$, and strongly concave in $x$; that is, 
\[
	\frac{\gamma_x}{2}(x-y)^2 \leq -\left( H_x(x,p) - H_x(y,p) \right)(x-y)
\]
for some $\gamma_x >0$. Using the previous two inequalities, \eqref{eq:HSystem N agent}, and \eqref{eq:HSystem N agents penalized}, we get
\begin{align*}
	& \frac{\gamma_p}{2}\| P^i + \varpi^N - (\tilde{P}^i +\tilde{\varpi}^N) \|_2^2 + \frac{\gamma_x}{2}\|X^i - \tilde{X}^i\|_2^2 
	\\
	& \leq \int_0^T \bigg(\left( H_p(X^i,P^i + \varpi^N) -H_p(\tilde{X}^i,\tilde{P}^i + \tilde{\varpi}^N) \right) (P^i + \varpi^N - (\tilde{P}^i +\tilde{\varpi}^N)) 
	\\
	& \quad \quad \quad \quad - \left( H_x(X^i,P^i + \varpi^N) -H_x(\tilde{X}^i,\tilde{P}^i + \tilde{\varpi}^N) \right) (X^i -  \tilde{X}^i) \bigg)dt
	\\
	& = \int_0^T \bigg( -\left( \dot{X}^i - \dot{\tilde{X}}^i \right) ( P^i - \tilde{P}^i) -\left( \dot{X}^i - \dot{\tilde{X}}^i \right) ( \varpi^N -\tilde{\varpi}^N)
	\\
	& \quad \quad \quad \quad -\left( \dot{P}^i - \dot{\tilde{P}}^i + \epsilon^i \right) ( X^i - \tilde{X}^i) \bigg) dt
	\\
	& = \int_0^T \bigg( - \frac{d}{dt}\left( ( X^i - \tilde{X}^i) ( P^i - \tilde{P}^i) \right) -\left( \dot{X}^i - \dot{\tilde{X}}^i \right) ( \varpi^N -\tilde{\varpi}^N) - \epsilon^i( X^i - \tilde{X}^i) \bigg) dt
	\\
	& = -( X^i(T) - \tilde{X}^i(T))( P^i(T) - \tilde{P}^i(T)) - \int_0^T \bigg(  \left( \dot{X}^i - \dot{\tilde{X}}^i \right) ( \varpi^N -\tilde{\varpi}^N) + \epsilon^i( X^i - \tilde{X}^i) \bigg) dt
	\\
	& = -( X^i(T) - \tilde{X}^i(T))( u_T'(X^i(T)) - u_T'(\tilde{X}^i(T))-\epsilon_T^i) 
	\\
	& \quad - \int_0^T \bigg(  \left( \dot{X}^i - \dot{\tilde{X}}^i \right) ( \varpi^N -\tilde{\varpi}^N) + \epsilon^i( X^i - \tilde{X}^i) \bigg) dt,
\end{align*}
for $1\leq i \leq N$. By Assumption \ref{hyp:uT convex DS}, 
the previous inequality gives
\begin{align*}
	& \frac{\gamma_p}{2}\| P^i + \varpi^N - (\tilde{P}^i +\tilde{\varpi}^N) \|_2^2 + \frac{\gamma_x}{2}\|X^i - \tilde{X}^i\|_2^2  \nonumber
	\\
	& \leq -\frac{\gamma_T}{2} ( X^i(T) - \tilde{X}^i(T))^2 + \epsilon^i_T ( X^i(T) - \tilde{X}^i(T)) \nonumber
	\\
	& \quad - \int_0^T \bigg(  \left( \dot{X}^i - \dot{\tilde{X}}^i \right) ( \varpi^N -\tilde{\varpi}^N) + \epsilon^i( X^i - \tilde{X}^i) \bigg) dt \nonumber
	\\
	& \leq \left( -\frac{\gamma_T}{2} + \frac{1}{4 \delta_1} \right) ( X^i(T) - \tilde{X}^i(T))^2 + \delta_1 (\epsilon_T^i)^2 \nonumber
	\\
	& \quad - \int_0^T \bigg(  \left( \dot{X}^i - \dot{\tilde{X}}^i \right) ( \varpi^N -\tilde{\varpi}^N) + \epsilon^i( X^i - \tilde{X}^i) \bigg) dt	
\end{align*}
for $\delta_1>0$ to be selected. Adding the previous inequality over $i$, and using the third equation in \eqref{eq:HSystem N agent} and \eqref{eq:HSystem N agents penalized}, we get
\begin{align*}
	& \frac{\gamma_p}{2}\| \bP + \mathds{1}\varpi^N - (\tilde{\bP} +\mathds{1} \tilde{\varpi}^N) \|_2^2 + \frac{\gamma_x}{2}\|\bX - \tilde{\bX}\|_2^2  \nonumber
	\\
	& \leq \left( -\frac{\gamma_T}{2} + \frac{1}{4 \delta_1} \right) |\bX(T) - \tilde{\bX}(T)|^2 + \delta_1 |\epsilon_T|^2 \nonumber
	\\
	& \quad + \int_0^T N \epsilon_q ( \varpi^N -\tilde{\varpi}^N) dt	 + \delta_2 \|\bepsilon\|_2^2 + \frac{1}{4\delta_2}\|\bX-\tilde{\bX}\|_2^2
\end{align*}
for some $\delta_2>0$ to be selected. By adding and subtracting $P^i-\tilde{P}^i$, we write the previous inequality as
\begin{align*}
	& \frac{\gamma_p}{2}\| \bP + \mathds{1}\varpi^N - (\tilde{\bP} +\mathds{1} \tilde{\varpi}^N) \|_2^2 + \frac{\gamma_x}{2}\|\bX - \tilde{\bX}\|_2^2  \nonumber
	\\
	& \leq \left( -\frac{\gamma_T}{2} + \frac{1}{4 \delta_1} \right) |\bX(T) - \tilde{\bX}(T)|^2 + \delta_1 |\epsilon_T|^2 \nonumber
	\\
	& \quad + \int_0^T \epsilon_q \left( \sum_{i=1}^N \left( P^i + \varpi^N -(\tilde{P}^i + \tilde{\varpi}^N) \right) - \sum_{i=1}^N (P^i-\tilde{P}^i) \right) dt	 + \delta_2 \|\bepsilon\|_2^2 + \frac{1}{4\delta_2}\|\bX-\tilde{\bX}\|_2^2
	\\
	& \leq \left( -\frac{\gamma_T}{2} + \frac{1}{4 \delta_1} \right) |\bX(T) - \tilde{\bX}(T)|^2 + \delta_1 |\epsilon_T|^2 \nonumber
	\\
	& \quad + \sum_{i=1}^N \left( \delta_3 \|\epsilon_q\|_2^2 + \frac{1}{4 \delta_3} \|P^i + \varpi^N -(\tilde{P}^i + \tilde{\varpi}^N)\|_2^2 \right) 
	\\
	& \quad + \sum_{i=1}^N \left( \delta_4 \|\epsilon_q\|_2^2 + \frac{1}{4 \delta_4} \|P^i-\tilde{P}^i\|_2^2 \right) +\delta_2 \|\bepsilon\|_2^2 + \frac{1}{4\delta_2}\|\bX-\tilde{\bX}\|_2^2	
\end{align*}
for some $\delta_3,\delta_4>0$ to be selected. By Lemma \ref{lem:Ps bound}, we use the estimate \eqref{eq:ABound 2} in the previous inequality to obtain
\begin{align*}
	& \frac{\gamma_p}{2}\| \bP + \mathds{1}\varpi^N - (\tilde{\bP} +\mathds{1} \tilde{\varpi}^N) \|_2^2 + \frac{\gamma_x}{2}\|\bX - \tilde{\bX}\|_2^2  \nonumber
	\\
	& \leq \left( -\frac{\gamma_T}{2} + \frac{1}{4 \delta_1} + \frac{T}{\delta_4}(\mbox{Lip}(u'_T))^2\right) |\bX(T) - \tilde{\bX}(T)|^2 + \left( \delta_1 +\frac{T}{\delta_4}\right)|\epsilon_T|^2 \nonumber
	\\
	& \quad + \left( N \delta_3 +N \delta_4\right) \|\epsilon_q\|_2^2 + \frac{1}{4 \delta_3} \|\bP + \mathds{1}\varpi^N -(\tilde{P} + \mathds{1}\tilde{\varpi}^N)\|_2^2 
	\\
	& \quad +\left( \delta_2 + \frac{T^2}{\delta_4} \right) \|\bepsilon\|_2^2 + \left( \frac{1}{4\delta_2} + \frac{T^2}{\delta_4}(\mbox{Lip}(H_x))^2\right)\|\bX-\tilde{\bX}\|_2^2.	
\end{align*}
Grouping terms in the previous inequality, we get
\begin{align}\label{eq:ABound 3}
	& \left(\frac{\gamma_p}{2}-\frac{1}{4 \delta_3}\right) \| \bP + \mathds{1}\varpi^N - (\tilde{\bP} +\mathds{1} \tilde{\varpi}^N) \|_2^2 + \left( \frac{\gamma_x}{2}- \frac{1}{4\delta_2} - \frac{T^2}{\delta_4}(\mbox{Lip}(H_x))^2\right) \|\bX - \tilde{\bX}\|_2^2  \nonumber
	\\
	& \leq \left( -\frac{\gamma_T}{2} + \frac{1}{4 \delta_1} + \frac{T}{\delta_4}(\mbox{Lip}(u'_T))^2\right) |\bX(T) - \tilde{\bX}(T)|^2 + \left( \delta_1 +\frac{T}{\delta_4}\right)|\epsilon_T|^2 \nonumber
	\\
	& \quad + N\left(  \delta_3 + \delta_4\right) \|\epsilon_q\|_2^2 +\left( \delta_2 + \frac{T^2}{\delta_4} \right) \|\bepsilon\|_2^2.
\end{align}
Let $\delta>0$ satisfy $\gamma_T > 2 \delta$, and set
\[
	\delta_1 = \frac{1}{2(\gamma_T - 2 \delta)}, \quad \delta_2 > \frac{1}{2\left(\gamma_x - 2 T \delta \left(\frac{\mbox{Lip}(H_x)}{\mbox{Lip}(u'_T)} \right)^2\right)}, \quad \delta_3 > \frac{1}{2\gamma_p}, \quad \delta_4 = \frac{T}{\delta}(\mbox{Lip}(u'_T))^2.
\]
Then, with the previous selection, the constants in \eqref{eq:ABound 3} satisfy
\begin{align*}
	& \left(\frac{\gamma_p}{2}-\frac{1}{4 \delta_3}\right)>0, \quad \left( \frac{\gamma_x}{2}- \frac{1}{4\delta_2} - \frac{T^2}{\delta_4}(\mbox{Lip}(H_x))^2\right) > 0, 
	\\
	& \left( -\frac{\gamma_T}{2} + \frac{1}{4 \delta_1} + \frac{T}{\delta_4}(\mbox{Lip}(u'_T))^2\right) = 0.
\end{align*}
Let $C_1 = \min \left\{ \left(\frac{\gamma_p}{2}-\frac{1}{4 \delta_3}\right) , \; \left( \frac{\gamma_x}{2}- \frac{1}{4\delta_2} - \delta T\left( \frac{\mbox{Lip}(H_x)}{\mbox{Lip}(u'_T)}\right)^2\right) \right\}$. Then, \eqref{eq:ABound 3} gives
\begin{align}\label{eq:ABound 4}
	& \| \bP + \mathds{1}\varpi^N - (\tilde{\bP} +\mathds{1} \tilde{\varpi}^N) \|_2^2 + \|\bX - \tilde{\bX}\|_2^2  \nonumber
	\\
	& \leq \frac{1}{C_1} \left( \left( \frac{1}{2(\gamma_T - 2\delta)} +\frac{\delta}{(\mbox{Lip}(u'_T))^2}\right)|\epsilon_T|^2 + N\left(  \delta_3 + \frac{T}{\delta}(\mbox{Lip}(u'_T))^2\right) \|\epsilon_q\|_2^2 \right. \nonumber
	\\
	& \quad\quad\quad\; \left. + \left( \delta_2 + \frac{\delta T}{(\mbox{Lip}(u'_T))^2} \right) \|\bepsilon\|_2^2\right).
\end{align}
The result follows from \eqref{eq:ABound 4}.
\end{proof}

\begin{proof}[Proof of Theorem \ref{Prop:Main result}]
Let $(\bX,\bP)$ and $\varpi^N$ solve \eqref{eq:HSystem N agent}, and let $(\tilde{\bX},\tilde{\bP})$ and $\tilde{\varpi}^N$ satisfy \eqref{eq:HSystem N agents penalized} for $\bepsilon:[0,T]\to \Rr^N$, $\bepsilon_T \in \Rr^N$, and $\epsilon_q:[0,T]\to\Rr$. By triangle inequality,
\begin{equation*}
	\frac{1}{2}\| \varpi^N - \tilde{\varpi}^N \|_2^2  \leq \|P^i + \varpi^N - (\tilde{P}^i + \tilde{\varpi}^N) \|_2^2 + \|P^i - \tilde{P}^i\|_2^2,
\end{equation*}
for $1 \leq i \leq N$. Adding the previous inequalities over $i$, and using Lemma \ref{lem:Ps bound} and Proposition \ref{Prop:PricevsELError}, we have 
\begin{align}
\label{eq:Final bound 1}
	\frac{N}{2}\| \varpi^N - \tilde{\varpi}^N \|_2^2  &  \leq \bigg((4T^2 (\mbox{Lip}(H_x))^2+1) C + 4T^2 \bigg) \|\bepsilon\|_{L^2([0,T])}^2 \nonumber
	\\
	& \quad + \bigg((4T^2 (\mbox{Lip}(H_x))^2+1) C + 4T \bigg) |\bepsilon_T|^2  \nonumber
	\\
	& \quad + (4T^2 (\mbox{Lip}(H_x))^2+1) C \|\epsilon_q \|_{L^2([0,T])}^2  \nonumber
	\\
	& \quad + 4T \mbox{Lip}(u_T'))^2 |\bX(T) - \tilde{\bX}(T)|^2.
\end{align}
By Assumptions \ref{hyp:L separability DS} and \ref{hyp:Lipschitz 2}, using \eqref{eq:HSystem N agent} and \eqref{eq:HSystem N agents penalized}, we have
\begin{align*}
	|X^i(T)-\tilde{X}^i(T)|^2 & \leq T (\mbox{Lip}(H_p))^2 \|P^i + \varpi^N - (\tilde{P}^i + \tilde{\varpi}^N)\|_{L^2([0,T])}^2,
\end{align*}
for $1 \leq i \leq N$. Adding the previous inequalities over $i$, and using Proposition \ref{Prop:PricevsELError}, we get
\begin{align*}
	|\bX(T)-\tilde{\bX}(T)|^2 & \leq T (\mbox{Lip}(H_p))^2 \|\bP + \mathds{1}\varpi^N - (\tilde{\bP} + \mathds{1}\tilde{\varpi}^N)\|_{L^2([0,T])}^2
	\\
	& \leq T (\mbox{Lip}(H_p))^2 C \left( \|\bepsilon\|_{L^2([0,T])}^2 + |\bepsilon_T|^2 +\|\epsilon_q \|_{L^2([0,T])}^2 \right).
\end{align*}
The result follows from the previous inequality and \eqref{eq:Final bound 1}.
%
%
\end{proof}

\begin{remark}
Let $\varpi^N$ solve \eqref{eq:HSystem N agent}, $\tilde{\varpi}^N$ satisfy \eqref{eq:HSystem N agents penalized}, and $\varpi$ solve \eqref{eq:MFG system}. By triangle inequality,
\begin{equation*}
	\| \varpi - \tilde{\varpi}^N \|_2  \leq \| \varpi - \varpi^N\|_2 + \|\varpi^N - \tilde{\varpi}^N \|_2.
\end{equation*}
The first term on the right-hand side of the previous inequality relates to the convergence of the $N$-player price to the MFG price, while the second term relates to the results of  Lemma \ref{lem:Ps bound} and Proposition \ref{Prop:PricevsELError}. In our numerical method, a NN provides $\tilde{\varpi}^N$ for a fixed population of size $N$. The population changes as we train the NN, and we regard $\tilde{\varpi}^N$ as an approximation of $\varpi$. Therefore, the a posteriori estimates in Lemma \ref{lem:Ps bound} and Proposition \ref{Prop:PricevsELError}, together with the rate of convergence of the finite to the continuum game are essential to properly assess the quality in our approximation of $\varpi$ in \eqref{eq:MFG system} using the NN approach we propose. Because the convergence of finite population games to MFG is a matter of its own, we plan to study this convergence for the price formation problem in a separate work.
\end{remark}

\begin{remark}
Instead of adopting the particle approximation to estimate the distance of $\tilde{v}$ and $\tilde{\varpi}$ from $v^*$ and $\varpi$, respectively, we can adopt a continuous approach. For instance, using Wasserstein metrics, we can estimate the distance between $\tilde{m}$ solving \eqref{eq:Auxiliary transport eq} and $m$ solving the continuity equation in \eqref{eq:MFG system} in terms of the vector fields $\tilde{v}$ and $v^*$, the latter given by \eqref{eq:optimal feedback}. However, as we present in Section \ref{sec:NN for price}, the ML approach we propose utilizes a finite number of particle trajectories to optimize the parameters of the NNs, providing the approximations of $v^*$ and $\varpi$. Therefore, the particle approach analysis we adopt is better suited to the convergence analysis of our  numerical method.
\end{remark}

\begin{remark}
As we proved in Propositions \ref{Prop:FirstBound} and \ref{Prop:PricevsELError}, the residual that a perturbation incurs in a necessary condition is related to how far the perturbation is from the optimizer. A further consideration is whether the perturbed necessary condition characterizes an underlying perturbed optimization problem. Using the concept of Metric regularity (see \cite{MetricRegBook}, Chapter 3), it is possible to estimate the distance between a solution and a perturbation in terms of the residual the perturbation incurs in a condition characterizing solutions, such as a set inclusion or an Euler-Lagrange equation. For instance, in the context of linear optimal control, \cite{Quin1} interpreted a perturbed Pontryagin principle as the first-order condition of a perturbed optimization problem. In our case, assume that $\mathcal{E}: [0,T] \to \Rr$ satisfies $\dot{\mathcal{E}} = \epsilon$ and $\mathcal{E}(T) = -\epsilon_T$. Then, we can write \eqref{eq:EL single agent residual} as
\begin{equation}\label{eq:EL single agent alternative}
	\begin{cases}
	L_x(\tilde{X}(t),\tilde{v}(t,\tilde{X}(t)))-\frac{d}{dt} \left( L_v(\tilde{X}(t),\tilde{v}(t,\tilde{X}(t))) + \tilde{\varpi}(t)+\mathcal{E}(t)\right) = 0 &  t \in [0,T],
	\\
	L_v(\tilde{X}(T),\tilde{v}(T,\tilde{X}(T))) +\tilde{\varpi}(T) +\mathcal{E}(T)+ u_T'\left(\tilde{X}(T)\right)=0.
	\end{cases}
\end{equation}
Moreover, we can regard the right-hand side of  \eqref{eq:Balance residual} as a perturbation of the supply function $Q$; that is,
\begin{equation}\label{eq:Balance residual alternative}
	\int_{\Rr} \tilde{v}(t,x)\tilde{m}(t,x) \dx = \tilde{Q}(t), \quad t \in [0,T],
\end{equation}
where $\tilde{Q}=Q+\epsilon_q$. Then, we can consider if \eqref{eq:EL single agent alternative} and \eqref{eq:Balance residual alternative} characterize a price formation problem for the supply $\tilde{Q}$. In such case, Propositions \ref{Prop:FirstBound} and \ref{Prop:PricevsELError} allow further studying of the stability of solutions of \eqref{eq:MFG system} w.r.t. perturbation in the supply function $Q$. However, there is no guarantee that the duality relation established by \eqref{eq:Legendre transform}, which defines the price formation system \eqref{eq:MFG system}, holds for the perturbations $\tilde{v}$ and $\tilde{\varpi}$; that is, the perturbations are not necessarily solving an underlying mean-field game problem. 
\end{remark}

\section{Neural Networks for price formation models}
\label{sec:NN for price}
Here, we introduce the two architectures we use, which are based on the multi-layer perceptron (MLP) and the Recurrent Neural Network (RNN) structures. Both turn out to have a recurrent structure, as explained below. The architectures provide an approximation for $v^*$ and $\varpi$. Further material related to NN and ML can be found in \cite{HighamML}. We close this section with the numerical formulation of the loss function and the a posteriori estimate, which we implement in Section \ref{sec: Numerical results}.

\subsection{Neural Network architectures}
In this part, we introduce the notation for NN and elaborate on the two architectures we study. Each architecture is composed of two NN, one approximating $v^*$ and the other approximating $\varpi$. In the first configuration, we use a MLP that takes as input the current time and state of a player and returns $v^*$. Thus, it behaves as an approximation to the right-hand side of \eqref{eq:optimal feedback}. The price is approximated using another MLP. In the second configuration, we use two RNNs that encode the history of the supply up to time $t$ by keeping along the way an auxiliary state that encodes the relevant information about the past. In particular, it provides a non-anticipating control. While the price problem admits Markovian controls, like the first option, these are contained in the set of non-anticipating controls, like those obtained using RNN. Although these two architectures should have equivalent results for deterministic problems, this is not the case when the supply is random (common noise problem).

\subsubsection{General RNN architectures}
The RNN is a class of NN suitable when the input variable $\mathrm{x}$ is a temporal sequence. For instance, time discretization with $M$ time steps of the supply provides a temporal sequence $\mathrm{Q}=\left(\mathrm{Q}^{\langle 0 \rangle },\ldots,\mathrm{Q}^{\langle M \rangle }\right)$. The architecture of a single cell specifies the RNN. This cell is iterated over the time steps of the temporal sequence. Using a hidden state $\mathrm{h}$, the RNN connects layers in the temporal direction. The cell of the RNN consists of a MLP, which is a fully connected feed-forward NN of $L$ layers. Each layer $j$, $1\leq j \leq L$, is specified by a weight matrix $\mathrm{W}^{[j]}$ and a bias vector $\mathrm{b}^{[j]}$. Their dimensions depend on the number of neurons $n^{[j]}$ in that layer. We denote the MLP parameters by $\Theta$; that is, 
\[
	\Theta=\left(\mathrm{W}^{[1]},\mathrm{b}^{[1]},\ldots,\mathrm{W}^{[L]},\mathrm{b}^{[L]}\right).
\]
The output of layer $j$, which we denote by $\mathrm{a}^{[j]}$, is given by
\begin{equation*}
	\mathrm{a}^{[j]}=\sigma^{[j]}\left(\mathrm{W}^{[j]} \mathrm{a}^{[j-1]}+\mathrm{b}^{[j]}\right),\quad 1\leq j \leq L, \quad a^{[0]}=\mathrm{x},
\end{equation*}
where $\sigma^{[j]}$ is the activation function of layer $j$ applied coordinate-wise, and $\mathrm{x}$ denotes the input variable for the MLP that defines the cell. For instance, Figure \ref{fig:Standard RNN iteration} depicts the iteration of a standard RNN cell consisting of a MLP with three layers. In all architectures, the hidden state $\mathrm{h}$ is initialized by taking $\mathrm{h}^{\langle -1 \rangle}$ to be zero. The blue arrows highlight the connection in the temporal direction obtained by using the hidden state $\mathrm{h}$.

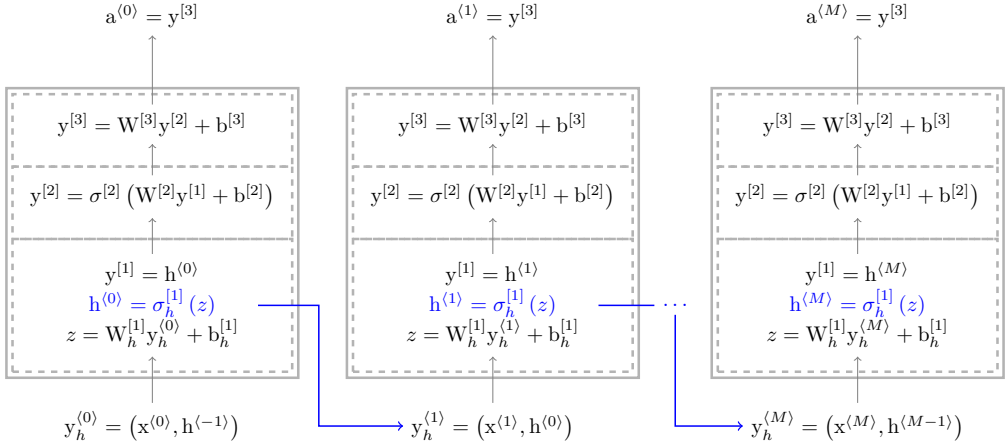
\begin{figure}[!htb]
	\centering
	\scalebox{0.8}{
\begin{tikzpicture}[
roundnode/.style={circle, draw=gray!60, fill=gray!5, very thick, minimum size=6mm},
squarednode/.style={rectangle, draw=gray!60, fill=gray!5, very thick, dashed},
]
	\node[squarednode] at (1, -1.5) [fill=none,minimum width=46mm,minimum height=22mm,dashed](RNNCellL1) {} ;
	\node[squarednode] at (1, 0.2) [fill=none,minimum width=46mm,minimum height=12mm,dashed](RNNCelL2) {} ;
	\node[squarednode] at (1, 1.4) [fill=none,minimum width=46mm,minimum height=12mm,dashed](RNNCelL3) {} ;	
	\node[squarednode] at (1, -0.3) [fill=none,minimum width=48mm,minimum height=48mm,solid](RNNCelL3) {} ;		
	\node[squarednode] at (1,-3.5) [draw=none,fill=none](Input0) {$\mathrm{y}_h^{\langle 0 \rangle}=\left( \mathrm{x}^{\langle 0 \rangle}, \mathrm{h}^{\langle -1 \rangle}\right)$} ;

	\node[squarednode] at (1,-1.5) [draw=none,fill=none](OutputL1at0) {$\begin{array}{c}
	\mathrm{y}^{[1]}=\mathrm{h}^{\langle 0 \rangle}
	\\	
	{\color{blue}\mathrm{h}^{\langle 0 \rangle} = \sigma_h^{[1]}\left( z\right)}
	\\
	z = \mathrm{W}_h^{[1]}\mathrm{y}_h^{\langle 0\rangle} + \mathrm{b}_h^{[1]}
	\end{array}$} ;
	\draw[->,draw=gray] (Input0.north).. controls +(up:0mm) and +(down:0mm) .. (OutputL1at0.south);

	\node[squarednode] at (1,0.3) [draw=none,fill=none](OutputL2at0) {$\mathrm{y}^{[2]}=\sigma^{[2]}\left( \mathrm{W}^{[2]} \mathrm{y}^{[1]}+ \mathrm{b}^{[2]}\right)$} ;
	\draw[->,draw=gray] (OutputL1at0.north).. controls +(up:0mm) and +(down:0mm) .. (OutputL2at0.south);

	\node[squarednode] at (1,1.5) [draw=none,fill=none](OutputL3at0) {$\mathrm{y}^{[3]}= \mathrm{W}^{[3]} \mathrm{y}^{[2]}+ \mathrm{b}^{[3]}$} ;
	\draw[->,draw=gray] (OutputL2at0.north).. controls +(up:0mm) and +(down:0mm) .. (OutputL3at0.south);

	\node[squarednode] at (1,3.3) [draw=none,fill=none](OutputVat0) {$\mathrm{a}^{\langle 0\rangle} = \mathrm{y}^{[3]}$} ;
	\draw[->,draw=gray] (OutputL3at0.north).. controls +(up:0mm) and +(down:0mm) .. (OutputVat0.south);		
	
	\node[squarednode] at (6.6, -1.5) [fill=none,minimum width=46mm,minimum height=22mm,dashed](RNNCellL1) {} ;
	\node[squarednode] at (6.6, 0.2) [fill=none,minimum width=46mm,minimum height=12mm,dashed](RNNCelL2) {} ;
	\node[squarednode] at (6.6, 1.4) [fill=none,minimum width=46mm,minimum height=12mm,dashed](RNNCelL3) {} ;	
	\node[squarednode] at (6.6, -0.3) [fill=none,minimum width=48mm,minimum height=48mm,solid](RNNCelL3) {} ;		
	\node[squarednode] at (6.6,-3.5) [draw=none,fill=none](Input1) {$\mathrm{y}_h^{\langle 1 \rangle}=\left(\mathrm{x}^{\langle 1 \rangle},\mathrm{h}^{\langle 0 \rangle}\right)$} ;
	\draw[->,draw=blue,line width=0.25mm] (OutputL1at0.east) |- ++(1,0) |- (Input1.west);

	\node[squarednode] at (6.6,-1.5) [draw=none,fill=none](OutputL1at1) {$\begin{array}{c}
	\mathrm{y}^{[1]}=\mathrm{h}^{\langle 1 \rangle}
	\\	
	{\color{blue}\mathrm{h}^{\langle 1 \rangle} = \sigma_h^{[1]}\left(z\right)}
	\\
	z= \mathrm{W}_h^{[1]}\mathrm{y}_h^{\langle 1\rangle} + \mathrm{b}_h^{[1]}
	\end{array}$} ;
	\draw[->,draw=gray] (Input1.north).. controls +(up:0mm) and +(down:0mm) .. (OutputL1at1.south);				

	\node[squarednode] at (6.6,0.3) [draw=none,fill=none](OutputL2at1) {$\mathrm{y}^{[2]}=\sigma^{[2]}\left( \mathrm{W}^{[2]} \mathrm{y}^{[1]}+ \mathrm{b}^{[2]}\right)$} ;
	\draw[->,draw=gray] (OutputL1at1.north).. controls +(up:0mm) and +(down:0mm) .. (OutputL2at1.south);

	\node[squarednode] at (6.6,1.5) [draw=none,fill=none](OutputL3at1) {$\mathrm{y}^{[3]}= \mathrm{W}^{[3]} \mathrm{y}^{[2]}+ \mathrm{b}^{[3]}$} ;
	\draw[->,draw=gray] (OutputL2at1.north).. controls +(up:0mm) and +(down:0mm) .. (OutputL3at1.south);

	\node[squarednode] at (6.6,3.3) [draw=none,fill=none](OutputVat1) {$\mathrm{a}^{\langle 1\rangle} = \mathrm{y}^{[3]}$} ;
	\draw[->,draw=gray] (OutputL3at1.north).. controls +(up:0mm) and +(down:0mm) .. (OutputVat1.south);		

\node[squarednode] at (9.6,-1.5) [draw=none,fill=none](Pts) {${\color{blue}\ldots}$} ;
\draw[-,draw=blue,line width=0.25mm] (OutputL1at1.east)  |- ++(0,0) -- (Pts.west);	

	\node[squarednode] at (12.6, -1.5) [fill=none,minimum width=46mm,minimum height=22mm,dashed](RNNCellL1) {} ;
	\node[squarednode] at (12.6, 0.2) [fill=none,minimum width=46mm,minimum height=12mm,dashed](RNNCelL2) {} ;
	\node[squarednode] at (12.6, 1.4) [fill=none,minimum width=46mm,minimum height=12mm,dashed](RNNCelL3) {} ;	
	\node[squarednode] at (12.6, -0.3) [fill=none,minimum width=48mm,minimum height=48mm,solid](RNNCelL3) {} ;		
	\node[squarednode] at (12.6,-3.5) [draw=none,fill=none](InputM) {$\mathrm{y}_h^{\langle M\rangle}=\left(\mathrm{x}^{\langle M \rangle}, \mathrm{h}^{\langle M-1 \rangle}\right)$} ;
\draw[->,draw=blue,line width=0.25mm] (Pts.south)  ++(0,0) |- (InputM.west);

	\node[squarednode] at (12.6,-1.5) [draw=none,fill=none](OutputL1atM) {$\begin{array}{c}
	\mathrm{y}^{[1]}=\mathrm{h}^{\langle M \rangle} 
	\\	
	{\color{blue}\mathrm{h}^{\langle M \rangle} = \sigma_h^{[1]}\left(z\right)}
	\\
	z= \mathrm{W}_h^{[1]}\mathrm{y}_h^{\langle M\rangle} + \mathrm{b}_h^{[1]}
	\end{array}$} ;
	\draw[->,draw=gray] (InputM.north).. controls +(up:0mm) and +(down:0mm) .. (OutputL1atM.south);				

	\node[squarednode] at (12.6,0.3) [draw=none,fill=none](OutputL2atM) {$\mathrm{y}^{[2]}=\sigma^{[2]}\left( \mathrm{W}^{[2]} \mathrm{y}^{[1]}+ \mathrm{b}^{[2]}\right)$} ;
	\draw[->,draw=gray] (OutputL1atM.north).. controls +(up:0mm) and +(down:0mm) .. (OutputL2atM.south);

	\node[squarednode] at (12.6,1.5) [draw=none,fill=none](OutputL3atM) {$\mathrm{y}^{[3]}= \mathrm{W}^{[3]} \mathrm{y}^{[2]}+ \mathrm{b}^{[3]}$} ;
	\draw[->,draw=gray] (OutputL2atM.north).. controls +(up:0mm) and +(down:0mm) .. (OutputL3atM.south);

	\node[squarednode] at (12.6,3.3) [draw=none,fill=none](OutputVatM) {$\mathrm{a}^{\langle M\rangle} = \mathrm{y}^{[3]}$} ;
	\draw[->,draw=gray] (OutputL3atM.north).. controls +(up:0mm) and +(down:0mm) .. (OutputVatM.south);		
\end{tikzpicture}
}
\caption{Iteration of a standard RNN.}
\label{fig:Standard RNN iteration}
\end{figure}
When necessary, we emphasize the dependence of the approximated values on the parameter $\Theta$. For instance, if we approximate $v^*$ at time $t_k$ through $\mathrm{v}^{\langle k \rangle}$, which is computed by a NN with a parameter $\Theta$, we write
\[
	\mathrm{v}^{\langle k\rangle }(\Theta).
\]
When it is clear that we are considering a specific neural network architecture, we omit the dependence on $\Theta$.

For both architectures, we use a MLP with the following hyper-parameters: the number of layers is $L_\varpi=L_v=3$, $\sigma^{[1]}$, and $\sigma^{[2]}$ are the sigmoid activation functions, and $\sigma^{[3]}$ is equal to the identity function. For the number of neurons, we select $n^{[1]}=n^{[2]}=64$, $n^{[3]}=1$ for the MLP approximating $v$, and $n^{[1]}=n^{[2]}=32$, $n^{[3]}=1$ for the MLP approximating $\varpi$. Thus 
\begin{align*}
\Theta_v=\left(W^{[1]}_{64\times 3},b^{[1]}_{64\times 1},W^{[2]}_{64\times 64},b^{[2]}_{64\times 1},W^{[3]}_{1\times 64},b^{[3]}_{1\times 1}\right), 
\\
\Theta_\varpi=\left(W^{[1]}_{32\times 2},b^{[1]}_{32\times 1},W^{[2]}_{32\times 32},b^{[2]}_{32\times 1},W^{[3]}_{1\times 32},b^{[3]}_{1\times 1}\right).
\end{align*}
When using the RNN, we add to the previous a hidden state of dimension $32$. Thus
\begin{align*}
\Theta_v=\left(W^{[1]}_{64\times 4},b^{[1]}_{64\times 1},W^{[2]}_{64\times 64},b^{[2]}_{64\times 1},W^{[3]}_{1\times 64},b^{[3]}_{1\times 1},{W_h^{[1]}}_{32\times 2},{b_h^{[1]}}_{32\times 1},{W_h^{[2]}}_{32\times 1},{b_h^{[2]}}_{1\times 1}\right), 
\\
\Theta_\varpi=\left(W^{[1]}_{32\times 2},b^{[1]}_{32\times 1},W^{[2]}_{32\times 32},b^{[2]}_{32\times 1},W^{[3]}_{1\times 32},b^{[3]}_{1\times 1},{W_h^{[1]}}_{32\times 2},{b_h^{[1]}}_{32\times 1},{W_h^{[2]}}_{32\times 1},{b_h^{[2]}}_{1\times 1}\right).
\end{align*}

We introduce the time discretization because we iterate the NN across the temporal direction. Let $\Delta_t =T/M$ be the time-step size, and $t_k = k \Delta_t$, $k=0,\ldots,M$. We use a forward-Euler discretization of \eqref{eq:Agent dynamics}
to update the state variable; that is,
\begin{equation}\label{eq:NN Forward dynamics}
	X^{\langle k+1 \rangle}=X^{\langle k \rangle}+  \Delta_t \mathrm{v}^{\langle k \rangle}(\Theta_v), \quad X^{\langle 0 \rangle}=x_0,
\end{equation}
where $\mathrm{v}^{\langle k \rangle}(\Theta_v)$ for $k=0,\ldots,M-1$ is computed by a NN, $\mathrm{NN}_v$, whose cell is specified by the parameter $\Theta_v$. Another NN with parameter $\Theta_\varpi$ computes $\varpi^{\langle k \rangle}$, for $k=0,\ldots,M$. 

\begin{remark}
In this work, we have fixed the hyper-parameters (number of layers, number of neurons, and activation functions) in order to simplify our presentation. Moreover, the choice of the hyper-parameter values was driven by exhaustive numerical experiments. However, different values are viable, and the exposition with arbitrary hyper-parameters follows the same structure we present here.
\end{remark}

\begin{remark}
The NN approximations of $v^*$ in \eqref{eq:optimal feedback} provided by the two architectures we study are in feedback form; that is, $\mathrm{v}^{\langle k \rangle}(\Theta_v)$ in \eqref{eq:NN Forward dynamics} depends on the current position $X^{\langle k \rangle}$ and price $\varpi^{\langle k \rangle}$. More precisely, 
\begin{equation}\label{eq:Neural ODE discrete}
	X^{\langle k+1 \rangle}=X^{\langle k \rangle}+  \Delta_t \mathrm{v}^{\langle k \rangle}(\Theta_v)(X^{\langle k \rangle},\varpi^{\langle k \rangle}), \quad X^{\langle 0 \rangle}=x_0,
\end{equation}
for $k=0,\ldots,M-1$. We rely on the previous discrete dynamics to obtain our discrete approximation of $v^*$ and $\varpi$. However, taking the limit of \eqref{eq:Neural ODE discrete} as $\Delta_t \to 0$ formally, we can consider a  NN approximation of $v^*$ of the form
\begin{equation*}
	\dot{X}(t) = \mathrm{v}(\Theta_v)(X(t),\varpi(t)), \quad X(0)=x_0,
\end{equation*}
obtaining a so-called Neural Ordinary Differential Equation, as proposed in \cite{NEURIPS2018_69386f6b}. Neural ODEs are formally derived as the limit of a Residual NN (see \cite{7780459}) when the number of layers in the skip connections goes to infinity. Because of the continuous structure, Neural ODEs are trained using an adjoint state equation. Therefore, an alternative continuous approach to approximate $v^*$ and $\varpi$ is to use Neural ODEs, which goes beyond the objective of the current study.
\end{remark}

\subsection{First architecture: MLP with instantaneous feedback.}

For this architecture, we consider the dependence of the optimal control according to \eqref{eq:optimal feedback}; that is, dependence at time $t$ is through the time, state, and price variables at time $t$. For the price, we consider its dependence on time and supply at time $t$. Let $\mathrm{MLP}_v$ and $\mathrm{MLP}_\varpi$ denote a MLPs approximating $v^*$ and $\varpi$, with $L_v$ and $L_\varpi$ layers, respectively. At time step $t_k$, the input for $\mathrm{MLP}_v$ is $\left(t_{k},X^{(i)\langle k \rangle}, \varpi^{\langle k \rangle} \right)$, and that for $\mathrm{MLP}_\varpi$ is $\left(t_{k}, Q^{\langle k \rangle} \right)$. The outputs are $\mathrm{v}^{(i)\langle k \rangle}(\Theta_v)$ and $\varpi^{\langle k \rangle}(\Theta_\varpi)$, respectively. The iteration of the architecture $\mathrm{MLP}_v$ is depicted in Figure \ref{fig:MLPv instantaneous}, where the velocity introduces a connection along the temporal direction due to the Forward-Euler discretization of the sate. Figure \ref{fig:MLPw instantaneous} illustrates the iteration of the architecture $\mathrm{MLP}_\varpi$, for which no temporal connection is used. 

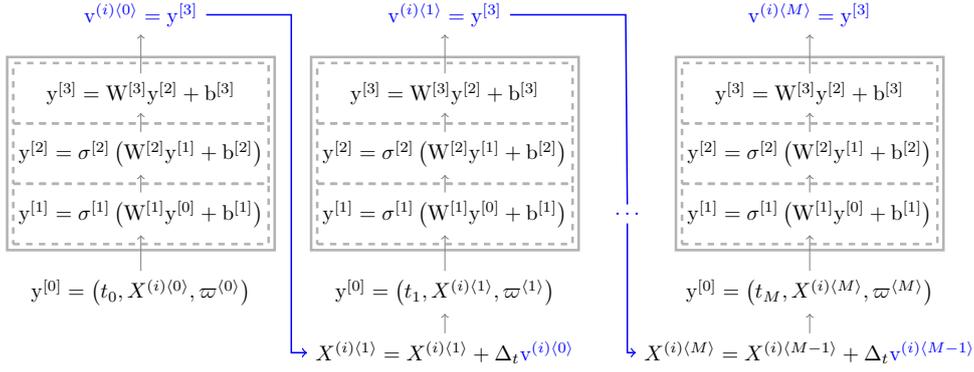
\begin{figure}[!htb]
	\centering
	\scalebox{0.8}{
 \begin{tikzpicture}[
roundnode/.style={circle, draw=gray!60, fill=gray!5, very thick, minimum size=6mm},
squarednode/.style={rectangle, draw=gray!60, fill=gray!5, very thick, dashed},
]
	\node[squarednode] at (1, -1.5) [fill=none,minimum width=42mm,minimum height=10mm,dashed](RNNCellL1) {} ;
	\node[squarednode] at (1, -0.5) [fill=none,minimum width=42mm,minimum height=10mm,dashed](RNNCelL2) {} ;
	\node[squarednode] at (1, 0.5) [fill=none,minimum width=42mm,minimum height=10mm,dashed](RNNCelL3) {} ;	
	\node[squarednode] at (1, -0.5) [fill=none,minimum width=44mm,minimum height=32mm,solid](RNNCelL3) {} ;		
	\node[squarednode] at (1,-2.8) [draw=none,fill=none](Input0) {$\mathrm{y}^{[0]}=\left( t_0, X^{(i)\langle 0 \rangle},\varpi^{\langle 0 \rangle}\right)$} ;
	\node[squarednode] at (1,-1.5) [draw=none,fill=none](OutputL1at0) {$\mathrm{y}^{[1]}=\sigma^{[1]}\left( \mathrm{W}^{[1]}\mathrm{y}^{[0]} + \mathrm{b}^{[1]}\right)$} ;
	\draw[->,draw=gray] (Input0.north).. controls +(up:0mm) and +(down:0mm) .. (OutputL1at0.south);

	\node[squarednode] at (1,-0.5) [draw=none,fill=none](OutputL2at0) {$\mathrm{y}^{[2]}=\sigma^{[2]}\left( \mathrm{W}^{[2]} \mathrm{y}^{[1]}+ \mathrm{b}^{[2]}\right)$} ;
	\draw[->,draw=gray] (OutputL1at0.north).. controls +(up:0mm) and +(down:0mm) .. (OutputL2at0.south);

	\node[squarednode] at (1,0.5) [draw=none,fill=none](OutputL3at0) {$\mathrm{y}^{[3]}= \mathrm{W}^{[3]} \mathrm{y}^{[2]}+ \mathrm{b}^{[3]}$} ;
	\draw[->,draw=gray] (OutputL2at0.north).. controls +(up:0mm) and +(down:0mm) .. (OutputL3at0.south);

	\node[squarednode] at (1,1.8) [draw=none,fill=none](OutputVat0) {${\color{blue}\mathrm{v}^{(i)\langle 0\rangle} = \mathrm{y}^{[3]}}$} ;
	\draw[->,draw=gray] (OutputL3at0.north).. controls +(up:0mm) and +(down:0mm) .. (OutputVat0.south);

	\node[squarednode] at (6, -1.5) [fill=none,minimum width=42mm,minimum height=10mm,dashed](RNNCellL1) {} ;
	\node[squarednode] at (6, -0.5) [fill=none,minimum width=42mm,minimum height=10mm,dashed](RNNCelL2) {} ;
	\node[squarednode] at (6, 0.5) [fill=none,minimum width=42mm,minimum height=10mm,dashed](RNNCelL3) {} ;	
	\node[squarednode] at (6, -0.5) [fill=none,minimum width=44mm,minimum height=32mm,solid](RNNCelL3) {} ;		
	\node[squarednode] at (6,-3.8) [draw=none,fill=none](X1) {$X^{(i)\langle 1 \rangle} = X^{(i)\langle 1 \rangle} +  \Delta_t {\color{blue} \mathrm{v}^{(i)\langle 0 \rangle}}$} ;
\draw[->,draw=blue,line width=0.25mm] (OutputVat0.east) |- ++(1.4,0) |- (X1.west);	
	
	\node[squarednode] at (6,-2.8) [draw=none,fill=none](Input1) {$\mathrm{y}^{[0]}=\left( t_1, X^{(i)\langle 1 \rangle}, \varpi^{\langle 1 \rangle}\right)$} ;
	\draw[->,draw=gray] (X1.north).. controls +(up:0mm) and +(down:0mm) .. (Input1.south);

	\node[squarednode] at (6,-1.5) [draw=none,fill=none](OutputL1at1) {$\mathrm{y}^{[1]}=\sigma^{[1]}\left( \mathrm{W}^{[1]}\mathrm{y}^{[0]} + \mathrm{b}^{[1]}\right)$} ;
	\draw[->,draw=gray] (Input1.north).. controls +(up:0mm) and +(down:0mm) .. (OutputL1at1.south);

	\node[squarednode] at (6,-0.5) [draw=none,fill=none](OutputL2at1) {$\mathrm{y}^{[2]}=\sigma^{[2]}\left( \mathrm{W}^{[2]} \mathrm{y}^{[1]}+ \mathrm{b}^{[2]}\right)$} ;
	\draw[->,draw=gray] (OutputL1at1.north).. controls +(up:0mm) and +(down:0mm) .. (OutputL2at1.south);

	\node[squarednode] at (6,0.5) [draw=none,fill=none](OutputL3at1) {$\mathrm{y}^{[3]}= \mathrm{W}^{[3]} \mathrm{y}^{[2]}+ \mathrm{b}^{[3]}$} ;
	\draw[->,draw=gray] (OutputL2at1.north).. controls +(up:0mm) and +(down:0mm) .. (OutputL3at1.south);

	\node[squarednode] at (6,1.8) [draw=none,fill=none](OutputVat1) {${\color{blue}\mathrm{v}^{(i)\langle 1\rangle} = \mathrm{y}^{[3]}}$} ;
	\draw[->,draw=gray] (OutputL3at1.north).. controls +(up:0mm) and +(down:0mm) .. (OutputVat1.south);		

\node[squarednode] at (9,-1.5) [draw=none,fill=none](Pts) {${\color{blue}\ldots}$} ;
\draw[-,draw=blue,line width=0.25mm] (OutputVat1.east)  |- ++(1.9,0) -- (Pts.north);	

	\node[squarednode] at (12, -1.5) [fill=none,minimum width=42mm,minimum height=10mm,dashed](RNNCellL1) {} ;
	\node[squarednode] at (12, -0.5) [fill=none,minimum width=42mm,minimum height=10mm,dashed](RNNCelL2) {} ;
	\node[squarednode] at (12, 0.5) [fill=none,minimum width=42mm,minimum height=10mm,dashed](RNNCelL3) {} ;	
	\node[squarednode] at (12, -0.5) [fill=none,minimum width=44mm,minimum height=32mm,solid](RNNCelL3) {} ;		
	\node[squarednode] at (12,-3.8) [draw=none,fill=none](XM) {$X^{(i)\langle M \rangle} = X^{(i)\langle M-1 \rangle} +  \Delta_t {\color{blue} \mathrm{v}^{(i)\langle M-1 \rangle}}$} ;
\draw[->,draw=blue,line width=0.25mm] (Pts.south)  ++(0,0) |- (XM.west);	
	
	\node[squarednode] at (12,-2.8) [draw=none,fill=none](InputM) {$\mathrm{y}^{[0]}=\left( t_M, X^{(i)\langle M \rangle},\varpi^{\langle M \rangle}\right)$} ;
	\draw[->,draw=gray] (XM.north).. controls +(up:0mm) and +(down:0mm) .. (InputM.south);

	\node[squarednode] at (12,-1.5) [draw=none,fill=none](OutputL1atM) {$\mathrm{y}^{[1]}=\sigma^{[1]}\left( \mathrm{W}^{[1]}\mathrm{y}^{[0]} + \mathrm{b}^{[1]}\right)$} ;
	\draw[->,draw=gray] (InputM.north).. controls +(up:0mm) and +(down:0mm) .. (OutputL1atM.south);

	\node[squarednode] at (12,-0.5) [draw=none,fill=none](OutputL2atM) {$\mathrm{y}^{[2]}=\sigma^{[2]}\left( \mathrm{W}^{[2]} \mathrm{y}^{[1]}+ \mathrm{b}^{[2]}\right)$} ;
	\draw[->,draw=gray] (OutputL1atM.north).. controls +(up:0mm) and +(down:0mm) .. (OutputL2atM.south);

	\node[squarednode] at (12,0.5) [draw=none,fill=none](OutputL3atM) {$\mathrm{y}^{[3]}= \mathrm{W}^{[3]} \mathrm{y}^{[2]}+ \mathrm{b}^{[3]}$} ;
	\draw[->,draw=gray] (OutputL2atM.north).. controls +(up:0mm) and +(down:0mm) .. (OutputL3atM.south);

	\node[squarednode] at (12,1.8) [draw=none,fill=none](OutputVatM) {${\color{blue}\mathrm{v}^{(i)\langle M\rangle} = \mathrm{y}^{[3]}}$} ;
	\draw[->,draw=gray] (OutputL3atM.north).. controls +(up:0mm) and +(down:0mm) .. (OutputVatM.south);		
(Nplus1.west);
\end{tikzpicture}
}
\caption{Iteration of the MLP for $v^*$, $\mathrm{MLP}_v$, with instantaneous feedback.}
\label{fig:MLPv instantaneous}
\end{figure}

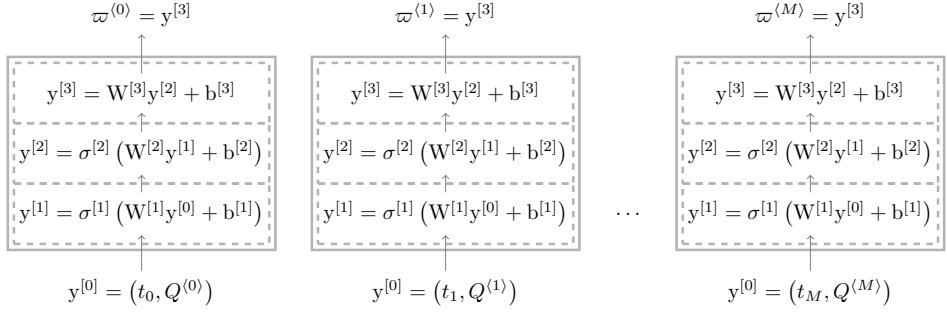
\begin{figure}[!htb]
	\centering
	\scalebox{0.8}{
 \begin{tikzpicture}[
roundnode/.style={circle, draw=gray!60, fill=gray!5, very thick, minimum size=6mm},
squarednode/.style={rectangle, draw=gray!60, fill=gray!5, very thick, dashed},
]
	\node[squarednode] at (1, -1.5) [fill=none,minimum width=42mm,minimum height=10mm,dashed](RNNCellL1) {} ;
	\node[squarednode] at (1, -0.5) [fill=none,minimum width=42mm,minimum height=10mm,dashed](RNNCelL2) {} ;
	\node[squarednode] at (1, 0.5) [fill=none,minimum width=42mm,minimum height=10mm,dashed](RNNCelL3) {} ;	
	\node[squarednode] at (1, -0.5) [fill=none,minimum width=44mm,minimum height=32mm,solid](RNNCelL3) {} ;		
	\node[squarednode] at (1,-2.8) [draw=none,fill=none](Input0) {$\mathrm{y}^{[0]}=\left( t_0, Q^{\langle 0 \rangle}\right)$} ;
	\node[squarednode] at (1,-1.5) [draw=none,fill=none](OutputL1at0) {$\mathrm{y}^{[1]}=\sigma^{[1]}\left( \mathrm{W}^{[1]}\mathrm{y}^{[0]} + \mathrm{b}^{[1]}\right)$} ;
	\draw[->,draw=gray] (Input0.north).. controls +(up:0mm) and +(down:0mm) .. (OutputL1at0.south);

	\node[squarednode] at (1,-0.5) [draw=none,fill=none](OutputL2at0) {$\mathrm{y}^{[2]}=\sigma^{[2]}\left( \mathrm{W}^{[2]} \mathrm{y}^{[1]}+ \mathrm{b}^{[2]}\right)$} ;
	\draw[->,draw=gray] (OutputL1at0.north).. controls +(up:0mm) and +(down:0mm) .. (OutputL2at0.south);

	\node[squarednode] at (1,0.5) [draw=none,fill=none](OutputL3at0) {$\mathrm{y}^{[3]}= \mathrm{W}^{[3]} \mathrm{y}^{[2]}+ \mathrm{b}^{[3]}$} ;
	\draw[->,draw=gray] (OutputL2at0.north).. controls +(up:0mm) and +(down:0mm) .. (OutputL3at0.south);

	\node[squarednode] at (1,1.8) [draw=none,fill=none](OutputVat0) {$\varpi^{\langle 0\rangle} = \mathrm{y}^{[3]}$} ;
	\draw[->,draw=gray] (OutputL3at0.north).. controls +(up:0mm) and +(down:0mm) .. (OutputVat0.south);

	\node[squarednode] at (6, -1.5) [fill=none,minimum width=42mm,minimum height=10mm,dashed](RNNCellL1) {} ;
	\node[squarednode] at (6, -0.5) [fill=none,minimum width=42mm,minimum height=10mm,dashed](RNNCelL2) {} ;
	\node[squarednode] at (6, 0.5) [fill=none,minimum width=42mm,minimum height=10mm,dashed](RNNCelL3) {} ;	
	\node[squarednode] at (6, -0.5) [fill=none,minimum width=44mm,minimum height=32mm,solid](RNNCelL3) {} ;		
	
	\node[squarednode] at (6,-2.8) [draw=none,fill=none](Input1) {$\mathrm{y}^{[0]}=\left( t_1, Q^{\langle 1 \rangle}\right)$} ;
	
	\node[squarednode] at (6,-1.5) [draw=none,fill=none](OutputL1at1) {$\mathrm{y}^{[1]}=\sigma^{[1]}\left( \mathrm{W}^{[1]}\mathrm{y}^{[0]} + \mathrm{b}^{[1]}\right)$} ;
	\draw[->,draw=gray] (Input1.north).. controls +(up:0mm) and +(down:0mm) .. (OutputL1at1.south);

	\node[squarednode] at (6,-0.5) [draw=none,fill=none](OutputL2at1) {$\mathrm{y}^{[2]}=\sigma^{[2]}\left( \mathrm{W}^{[2]} \mathrm{y}^{[1]}+ \mathrm{b}^{[2]}\right)$} ;
	\draw[->,draw=gray] (OutputL1at1.north).. controls +(up:0mm) and +(down:0mm) .. (OutputL2at1.south);

	\node[squarednode] at (6,0.5) [draw=none,fill=none](OutputL3at1) {$\mathrm{y}^{[3]}= \mathrm{W}^{[3]} \mathrm{y}^{[2]}+ \mathrm{b}^{[3]}$} ;
	\draw[->,draw=gray] (OutputL2at1.north).. controls +(up:0mm) and +(down:0mm) .. (OutputL3at1.south);

	\node[squarednode] at (6,1.8) [draw=none,fill=none](OutputVat1) {$\varpi^{\langle 1\rangle} = \mathrm{y}^{[3]}$} ;
	\draw[->,draw=gray] (OutputL3at1.north).. controls +(up:0mm) and +(down:0mm) .. (OutputVat1.south);		

\node[squarednode] at (9,-1.5) [draw=none,fill=none](Pts) {$\ldots$} ;

	\node[squarednode] at (12, -1.5) [fill=none,minimum width=42mm,minimum height=10mm,dashed](RNNCellL1) {} ;
	\node[squarednode] at (12, -0.5) [fill=none,minimum width=42mm,minimum height=10mm,dashed](RNNCelL2) {} ;
	\node[squarednode] at (12, 0.5) [fill=none,minimum width=42mm,minimum height=10mm,dashed](RNNCelL3) {} ;	
	\node[squarednode] at (12, -0.5) [fill=none,minimum width=44mm,minimum height=32mm,solid](RNNCelL3) {} ;		
	
	\node[squarednode] at (12,-2.8) [draw=none,fill=none](InputM) {$\mathrm{y}^{[0]}=\left( t_M, Q^{\langle M \rangle}\right)$} ;
	
	\node[squarednode] at (12,-1.5) [draw=none,fill=none](OutputL1atM) {$\mathrm{y}^{[1]}=\sigma^{[1]}\left( \mathrm{W}^{[1]}\mathrm{y}^{[0]} + \mathrm{b}^{[1]}\right)$} ;
	\draw[->,draw=gray] (InputM.north).. controls +(up:0mm) and +(down:0mm) .. (OutputL1atM.south);

	\node[squarednode] at (12,-0.5) [draw=none,fill=none](OutputL2atM) {$\mathrm{y}^{[2]}=\sigma^{[2]}\left( \mathrm{W}^{[2]} \mathrm{y}^{[1]}+ \mathrm{b}^{[2]}\right)$} ;
	\draw[->,draw=gray] (OutputL1atM.north).. controls +(up:0mm) and +(down:0mm) .. (OutputL2atM.south);

	\node[squarednode] at (12,0.5) [draw=none,fill=none](OutputL3atM) {$\mathrm{y}^{[3]}= \mathrm{W}^{[3]} \mathrm{y}^{[2]}+ \mathrm{b}^{[3]}$} ;
	\draw[->,draw=gray] (OutputL2atM.north).. controls +(up:0mm) and +(down:0mm) .. (OutputL3atM.south);

	\node[squarednode] at (12,1.8) [draw=none,fill=none](OutputVatM) {$\varpi^{\langle M\rangle} = \mathrm{y}^{[3]}$} ;
	\draw[->,draw=gray] (OutputL3atM.north).. controls +(up:0mm) and +(down:0mm) .. (OutputVatM.south);		
(Nplus1.west);
\end{tikzpicture}
}
\caption{Iteration of the MLP for $\varpi$, $\mathrm{MLP}_\varpi$, with instantaneous feedback.}
\label{fig:MLPw instantaneous}
\end{figure}

\subsection{Second architecture: RNN with supply history}

For this architecture, we consider the dependence of the distribution of players according to \eqref{eq:Xbar dynamics}; that is, the distribution of players at time $t$ depends on the supply history up to time $t$. This dependence is incorporated for approximating both $v^*$ and $\varpi$ by using two recurrent structures. Let $\mathrm{RNN}_v$ and $\mathrm{RNN}_\varpi$ denote RNNs approximating $v^*$ and $\varpi$, respectively. We feed the sequence $(Q^{\langle 0\rangle},\ldots,Q^{\langle M\rangle})$ to both $\mathrm{RNN}_v$ and $\mathrm{RNN}_\varpi$ and obtain the output sequences 
\[
	(\mathrm{a}^{\langle 0\rangle}(\Theta_v),\ldots,\mathrm{a}^{\langle M\rangle}(\Theta_v)) \quad \mbox{and} \quad (\mathrm{a}^{\langle 0\rangle}(\Theta_\varpi),\ldots,\mathrm{a}^{\langle M\rangle}(\Theta_\varpi)),
\]
respectively. At time step $t_k$, the input for $\mathrm{RNN}_v$ is $\left(t_{k}, X^{(i)\langle k \rangle}, \varpi^{\langle k \rangle},\mathrm{a}^{\langle k \rangle}(\Theta_v)\right)$, and the input for $\mathrm{RNN}_\varpi$ is $\left(t_{k},\mathrm{a}^{\langle k \rangle}(\Theta_\varpi)\right)$. The outputs are $\varpi^{\langle k \rangle}(\Theta_\varpi)$ and $\mathrm{v}^{(i)\langle k \rangle}(\Theta_v)$, respectively. Figure \ref{fig:RNN history Q v} illustrates the iteration of $\mathrm{RNN}_v$. The blue arrows show that the velocity and a hidden state carry information in the temporal direction. Notice that the hidden state $\mathrm{h}$ depends only on the supply history. Figure \ref{fig:RNN history Q w} depicts the iteration of $\mathrm{RNN}_\varpi$, in which the hidden state $\mathrm{h}$ is the variable carrying the supply history.

\begin{figure}[!htb]
	\centering
	\scalebox{0.8}{
	\begin{tikzpicture}[
roundnode/.style={circle, draw=gray!60, fill=gray!5, very thick, minimum size=6mm},
squarednode/.style={rectangle, draw=gray!60, fill=gray!5, very thick, dashed},
]
	\node[squarednode] at (1, -1.5) [fill=none,minimum width=46mm,minimum height=22mm,dashed](RNNCellL1) {} ;
	\node[squarednode] at (1, 0.2) [fill=none,minimum width=46mm,minimum height=12mm,dashed](RNNCelL2) {} ;
	\node[squarednode] at (1, 1.4) [fill=none,minimum width=46mm,minimum height=12mm,dashed](RNNCelL3) {} ;	
	\node[squarednode] at (1, -0.3) [fill=none,minimum width=48mm,minimum height=48mm,solid](RNNCelL3) {} ;		
	
	\node[squarednode] at (1,-4) [draw=none,fill=none](Input0) {${\renewcommand{\arraystretch}{1.5} \begin{array}{c} \mathrm{y}_h^{\langle 0 \rangle}=\left( Q^{\langle 0 \rangle}, {\color{blue} \mathrm{h}^{\langle -1 \rangle}}\right) \\ \left(t_0,X^{(i)\langle 0 \rangle},\varpi^{\langle 0 \rangle}\right) \end{array}}$} ;

	\node[squarednode] at (1,-1.5) [draw=none,fill=none](OutputL1at0) {$\begin{array}{c}
	\mathrm{y}^{[1]}=\sigma^{[1]}\left( \mathrm{W}^{[1]}\mathrm{y}^{[0]} + \mathrm{b}^{[1]}\right)
	\\
	\mathrm{y}^{[0]}=\left( t_0, X^{(i)\langle 0 \rangle}, \varpi^{\langle 0 \rangle},\mathrm{a}^{\langle 0 \rangle}\right) 
	\\	
	\mathrm{a}^{\langle 0 \rangle}=\sigma^{[1]}\left(\mathrm{W}_h^{[2]}\mathrm{h}^{\langle 0 \rangle}+\mathrm{b}_h^{[2]}\right)
	\\
	{\color{blue}\mathrm{h}^{\langle 0 \rangle} = \sigma_h^{[1]}\left( \mathrm{W}_h^{[1]}\mathrm{y}_h^{\langle 0\rangle} + \mathrm{b}_h^{[1]}\right)}
	\end{array}$} ;
	\draw[->,draw=gray] (Input0.north).. controls +(up:0mm) and +(down:0mm) .. (OutputL1at0.south);

	\node[squarednode] at (1,0.3) [draw=none,fill=none](OutputL2at0) {$\mathrm{y}^{[2]}=\sigma^{[2]}\left( \mathrm{W}^{[2]} \mathrm{y}^{[1]}+ \mathrm{b}^{[2]}\right)$} ;
	\draw[->,draw=gray] (OutputL1at0.north).. controls +(up:0mm) and +(down:0mm) .. (OutputL2at0.south);

	\node[squarednode] at (1,1.5) [draw=none,fill=none](OutputL3at0) {$\mathrm{y}^{[3]}= \mathrm{W}^{[3]} \mathrm{y}^{[2]}+ \mathrm{b}^{[3]}$} ;
	\draw[->,draw=gray] (OutputL2at0.north).. controls +(up:0mm) and +(down:0mm) .. (OutputL3at0.south);

	\node[squarednode] at (1,3.3) [draw=none,fill=none](OutputVat0) {${\color{blue}\mathrm{v}^{(i)\langle 0\rangle} = \mathrm{y}^{[3]}}$} ;
	\draw[->,draw=gray] (OutputL3at0.north).. controls +(up:0mm) and +(down:0mm) .. (OutputVat0.south);		
	
	\node[squarednode] at (6.6, -1.5) [fill=none,minimum width=46mm,minimum height=22mm,dashed](RNNCellL1) {} ;
	\node[squarednode] at (6.6, 0.2) [fill=none,minimum width=46mm,minimum height=12mm,dashed](RNNCelL2) {} ;
	\node[squarednode] at (6.6, 1.4) [fill=none,minimum width=46mm,minimum height=12mm,dashed](RNNCelL3) {} ;	
	\node[squarednode] at (6.6, -0.3) [fill=none,minimum width=48mm,minimum height=48mm,solid](RNNCelL3) {} ;		
	\node[squarednode] at (6.6,-5.5) [draw=none,fill=none](X1FE) {$X^{(i)\langle 1 \rangle} = X^{(i)\langle 0 \rangle} + \Delta_t {\color{blue} \mathrm{v}^{(i)\langle 0 \rangle}}$} ;

\draw[->,draw=blue,line width=0.25mm] (OutputVat0.east) |- ++(1.7,0) |- (X1FE.west);	

	\node[squarednode] at (6.6,-4) [draw=none,fill=none](Input1) {${\renewcommand{\arraystretch}{1.5} \begin{array}{c} \mathrm{y}_h^{\langle 1 \rangle}=\left( Q^{\langle 1 \rangle}, {\color{blue}\mathrm{h}^{\langle 0 \rangle}}\right) \\ \left(t_1,X^{(i)\langle 1 \rangle},\varpi^{\langle 1 \rangle}\right) \end{array}}$} ;
	\draw[->,draw=gray] (X1FE.north).. controls +(up:0mm) and +(down:0mm) .. (Input1.south);		

\draw[->,draw=blue,line width=0.25mm] (OutputL1at0.east) |- ++(0.31,0) |- (Input1.west);	
		
	\node[squarednode] at (6.6,-1.5) [draw=none,fill=none](OutputL1at1) {$\begin{array}{c}
	\mathrm{y}^{[1]}=\sigma^{[1]}\left( \mathrm{W}^{[1]}\mathrm{y}^{[0]} + \mathrm{b}^{[1]}\right)
	\\
	\mathrm{y}^{[0]}=\left( t_1, X^{(i)\langle 1 \rangle},\varpi^{\langle 1 \rangle}, \mathrm{a}^{\langle 1 \rangle}\right) 
	\\
	\mathrm{a}^{\langle 1 \rangle}=\sigma^{[1]}\left(\mathrm{W}_h^{[2]}\mathrm{h}^{\langle 1 \rangle}+\mathrm{b}_h^{[2]}\right)
	\\	
	{\color{blue}\mathrm{h}^{\langle 1 \rangle} = \sigma_h^{[1]}\left( \mathrm{W}_h^{[1]}\mathrm{y}_h^{\langle 1\rangle} + \mathrm{b}_h^{[1]}\right)}
	\end{array}$} ;
	\draw[->,draw=gray] (Input1.north).. controls +(up:0mm) and +(down:0mm) .. (OutputL1at1.south);				

	\node[squarednode] at (6.6,0.3) [draw=none,fill=none](OutputL2at1) {$\mathrm{y}^{[2]}=\sigma^{[2]}\left( \mathrm{W}^{[2]} \mathrm{y}^{[1]}+ \mathrm{b}^{[2]}\right)$} ;
	\draw[->,draw=gray] (OutputL1at1.north).. controls +(up:0mm) and +(down:0mm) .. (OutputL2at1.south);

	\node[squarednode] at (6.6,1.5) [draw=none,fill=none](OutputL3at1) {$\mathrm{y}^{[3]}= \mathrm{W}^{[3]} \mathrm{y}^{[2]}+ \mathrm{b}^{[3]}$} ;
	\draw[->,draw=gray] (OutputL2at1.north).. controls +(up:0mm) and +(down:0mm) .. (OutputL3at1.south);

	\node[squarednode] at (6.6,3.3) [draw=none,fill=none](OutputVat1) {${\color{blue}\mathrm{v}^{(i)\langle 1\rangle} = \mathrm{y}^{[3]}}$} ;
	\draw[->,draw=gray] (OutputL3at1.north).. controls +(up:0mm) and +(down:0mm) .. (OutputVat1.south);		

\node[squarednode] at (9.6,-1.5) [draw=none,fill=none](Pts) {${\color{blue}\ldots}$} ;
\draw[-,draw=blue,line width=0.25mm] (OutputVat1.east)  |- ++(1.9,0) -- (Pts.north);	
\draw[-,draw=blue,line width=0.25mm] (OutputL1at1.east)  |- ++(0,0) -- (Pts.west);	

	\node[squarednode] at (12.6, -1.5) [fill=none,minimum width=50mm,minimum height=22mm,dashed](RNNCellL1) {} ;
	\node[squarednode] at (12.6, 0.2) [fill=none,minimum width=50mm,minimum height=12mm,dashed](RNNCelL2) {} ;
	\node[squarednode] at (12.6, 1.4) [fill=none,minimum width=50mm,minimum height=12mm,dashed](RNNCelL3) {} ;	
	\node[squarednode] at (12.6, -0.3) [fill=none,minimum width=52mm,minimum height=48mm,solid](RNNCelL3) {} ;		
	\node[squarednode] at (12.6,-5.5) [draw=none,fill=none](XMFE) {$X^{(i)\langle M \rangle} = X^{(i)\langle M-1 \rangle} + \Delta_t {\color{blue} \mathrm{v}^{(i)\langle M-1 \rangle}}$} ;
\draw[->,draw=blue,line width=0.25mm] (Pts.south)  ++(0,0) |- (XMFE.west);	

	\node[squarednode] at (12.6,-4) [draw=none,fill=none](InputM) {${\renewcommand{\arraystretch}{1.5} \begin{array}{c} \mathrm{y}_h^{\langle M \rangle}=\left( Q^{\langle M \rangle}, {\color{blue} \mathrm{h}^{\langle M-1 \rangle}}\right) \\ \left(t_M,X^{(i)\langle M \rangle},\varpi^{\langle M \rangle}\right) \end{array}}$} ;
	\draw[->,draw=gray] (XMFE.north).. controls +(up:0mm) and +(down:0mm) .. (InputM.south);		

	\draw[->,draw=blue,line width=0.25mm] (Pts.south)  ++(0,0) |- (InputM.west);	
		
	\node[squarednode] at (12.6,-1.5) [draw=none,fill=none](OutputL1atM) {$\begin{array}{c}
	\mathrm{y}^{[1]}=\sigma^{[1]}\left( \mathrm{W}^{[1]}\mathrm{y}^{[0]} + \mathrm{b}^{[1]}\right)
	\\
	\mathrm{y}^{[0]}=\left( t_M, X^{(i)\langle M \rangle},\varpi^{\langle M \rangle}, \mathrm{a}^{\langle M \rangle}\right) 
	\\
	\mathrm{a}^{\langle M \rangle}=\sigma^{[1]}\left(\mathrm{W}_h^{[2]}\mathrm{h}^{\langle M \rangle}+\mathrm{b}_h^{[2]}\right)
	\\	
	{\color{blue}\mathrm{h}^{\langle M \rangle} = \sigma_h^{[1]}\left( \mathrm{W}_h^{[1]}\mathrm{y}_h^{\langle M\rangle} + \mathrm{b}_h^{[1]}\right)}
	\end{array}$} ;
	\draw[->,draw=gray] (InputM.north).. controls +(up:0mm) and +(down:0mm) .. (OutputL1atM.south);				

	\node[squarednode] at (12.6,0.3) [draw=none,fill=none](OutputL2atM) {$\mathrm{y}^{[2]}=\sigma^{[2]}\left( \mathrm{W}^{[2]} \mathrm{y}^{[1]}+ \mathrm{b}^{[2]}\right)$} ;
	\draw[->,draw=gray] (OutputL1atM.north).. controls +(up:0mm) and +(down:0mm) .. (OutputL2atM.south);

	\node[squarednode] at (12.6,1.5) [draw=none,fill=none](OutputL3atM) {$\mathrm{y}^{[3]}= \mathrm{W}^{[3]} \mathrm{y}^{[2]}+ \mathrm{b}^{[3]}$} ;
	\draw[->,draw=gray] (OutputL2atM.north).. controls +(up:0mm) and +(down:0mm) .. (OutputL3atM.south);

	\node[squarednode] at (12.6,3.3) [draw=none,fill=none](OutputVatM) {${\color{blue}\mathrm{v}^{(i)\langle M\rangle} = \mathrm{y}^{[3]}}$} ;
	\draw[->,draw=gray] (OutputL3atM.north).. controls +(up:0mm) and +(down:0mm) .. (OutputVatM.south);		
\end{tikzpicture}
}
\caption{Iteration of the RNN for $v^*$, $\mathrm{RNN}_v$, with supply history dependence.}
\label{fig:RNN history Q v}
\end{figure}
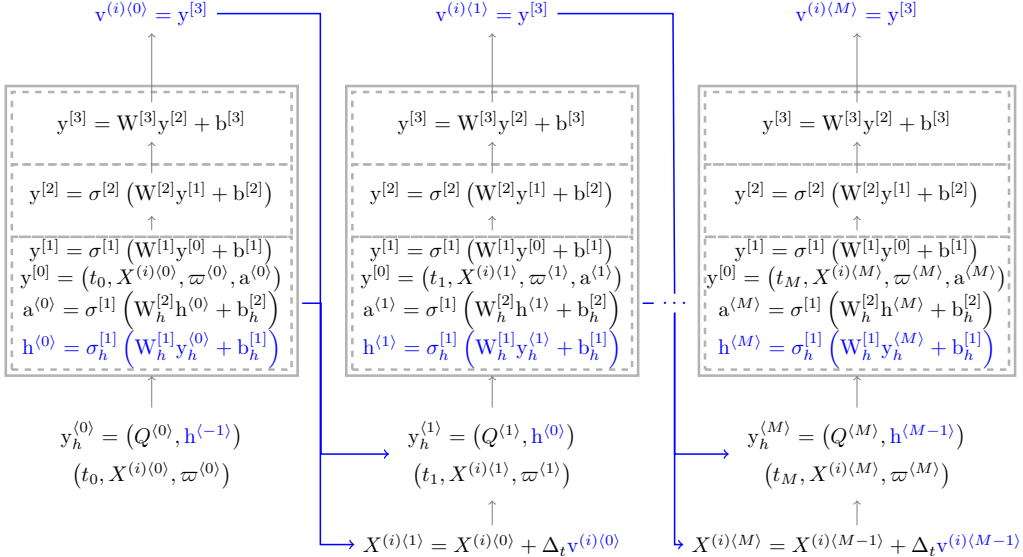

\begin{figure}[!htb]
	\centering
	\scalebox{0.8}{
\begin{tikzpicture}[
roundnode/.style={circle, draw=gray!60, fill=gray!5, very thick, minimum size=6mm},
squarednode/.style={rectangle, draw=gray!60, fill=gray!5, very thick, dashed},
]
	\node[squarednode] at (1, -1.5) [fill=none,minimum width=46mm,minimum height=22mm,dashed](RNNCellL1) {} ;
	\node[squarednode] at (1, 0.2) [fill=none,minimum width=46mm,minimum height=12mm,dashed](RNNCelL2) {} ;
	\node[squarednode] at (1, 1.4) [fill=none,minimum width=46mm,minimum height=12mm,dashed](RNNCelL3) {} ;	
	\node[squarednode] at (1, -0.3) [fill=none,minimum width=48mm,minimum height=48mm,solid](RNNCelL3) {} ;		
	\node[squarednode] at (1,-3.5) [draw=none,fill=none](Input0) {$\mathrm{y}_h^{\langle 0 \rangle}=\left( Q^{\langle 0 \rangle}, \mathrm{h}^{\langle -1 \rangle}\right)$} ;
	
	\node[squarednode] at (1,-1.5) [draw=none,fill=none](OutputL1at0) {$\begin{array}{c}
	\mathrm{y}^{[1]}=\sigma^{[1]}\left( \mathrm{W}^{[1]}\mathrm{y}^{[0]} + \mathrm{b}^{[1]}\right)
	\\
	\mathrm{y}^{[0]}=\left( t_0, \mathrm{a}^{\langle 0 \rangle}\right) 
	\\	
	\mathrm{a}^{\langle 0 \rangle}=\sigma^{[1]}\left(\mathrm{W}_h^{[2]}\mathrm{h}^{\langle 0 \rangle}+\mathrm{b}_h^{[2]}\right)
	\\
	{\color{blue}\mathrm{h}^{\langle 0 \rangle} = \sigma_h^{[1]}\left( \mathrm{W}_h^{[1]}\mathrm{y}_h^{\langle 0\rangle} + \mathrm{b}_h^{[1]}\right)}
	\end{array}$} ;
	\draw[->,draw=gray] (Input0.north).. controls +(up:0mm) and +(down:0mm) .. (OutputL1at0.south);

	\node[squarednode] at (1,0.3) [draw=none,fill=none](OutputL2at0) {$\mathrm{y}^{[2]}=\sigma^{[2]}\left( \mathrm{W}^{[2]} \mathrm{y}^{[1]}+ \mathrm{b}^{[2]}\right)$} ;
	\draw[->,draw=gray] (OutputL1at0.north).. controls +(up:0mm) and +(down:0mm) .. (OutputL2at0.south);

	\node[squarednode] at (1,1.5) [draw=none,fill=none](OutputL3at0) {$\mathrm{y}^{[3]}= \mathrm{W}^{[3]} \mathrm{y}^{[2]}+ \mathrm{b}^{[3]}$} ;
	\draw[->,draw=gray] (OutputL2at0.north).. controls +(up:0mm) and +(down:0mm) .. (OutputL3at0.south);

	\node[squarednode] at (1,3.3) [draw=none,fill=none](OutputVat0) {$\varpi^{\langle 0\rangle} = \mathrm{y}^{[3]}$} ;
	\draw[->,draw=gray] (OutputL3at0.north).. controls +(up:0mm) and +(down:0mm) .. (OutputVat0.south);		
	
	\node[squarednode] at (6.6, -1.5) [fill=none,minimum width=46mm,minimum height=22mm,dashed](RNNCellL1) {} ;
	\node[squarednode] at (6.6, 0.2) [fill=none,minimum width=46mm,minimum height=12mm,dashed](RNNCelL2) {} ;
	\node[squarednode] at (6.6, 1.4) [fill=none,minimum width=46mm,minimum height=12mm,dashed](RNNCelL3) {} ;	
	\node[squarednode] at (6.6, -0.3) [fill=none,minimum width=48mm,minimum height=48mm,solid](RNNCelL3) {} ;		

	\node[squarednode] at (6.6,-3.5) [draw=none,fill=none](Input1) {$\mathrm{y}_h^{\langle 1 \rangle}=\left( Q^{\langle 1 \rangle}, {\color{blue} \mathrm{h}^{\langle 0 \rangle}}\right)$} ;
\draw[->,draw=blue,line width=0.25mm] (OutputL1at0.east) |- ++(0.31,0) |- (Input1.west);

	\node[squarednode] at (6.6,-1.5) [draw=none,fill=none](OutputL1at1) {$\begin{array}{c}
	\mathrm{y}^{[1]}=\sigma^{[1]}\left( \mathrm{W}^{[1]}\mathrm{y}^{[0]} + \mathrm{b}^{[1]}\right)
	\\
	\mathrm{y}^{[0]}=\left( t_1, \mathrm{a}^{\langle 1 \rangle}\right) 
	\\
	\mathrm{a}^{\langle 1 \rangle}=\sigma^{[1]}\left(\mathrm{W}_h^{[2]}\mathrm{h}^{\langle 1 \rangle}+\mathrm{b}_h^{[2]}\right)
	\\	
	{\color{blue}\mathrm{h}^{\langle 1 \rangle} = \sigma_h^{[1]}\left( \mathrm{W}_h^{[1]}\mathrm{y}_h^{\langle 1\rangle} + \mathrm{b}_h^{[1]}\right)}
	\end{array}$} ;
	\draw[->,draw=gray] (Input1.north).. controls +(up:0mm) and +(down:0mm) .. (OutputL1at1.south);				

	\node[squarednode] at (6.6,0.3) [draw=none,fill=none](OutputL2at1) {$\mathrm{y}^{[2]}=\sigma^{[2]}\left( \mathrm{W}^{[2]} \mathrm{y}^{[1]}+ \mathrm{b}^{[2]}\right)$} ;
	\draw[->,draw=gray] (OutputL1at1.north).. controls +(up:0mm) and +(down:0mm) .. (OutputL2at1.south);

	\node[squarednode] at (6.6,1.5) [draw=none,fill=none](OutputL3at1) {$\mathrm{y}^{[3]}= \mathrm{W}^{[3]} \mathrm{y}^{[2]}+ \mathrm{b}^{[3]}$} ;
	\draw[->,draw=gray] (OutputL2at1.north).. controls +(up:0mm) and +(down:0mm) .. (OutputL3at1.south);

	\node[squarednode] at (6.6,3.3) [draw=none,fill=none](OutputVat1) {$\varpi^{\langle 1\rangle} = \mathrm{y}^{[3]}$} ;
	\draw[->,draw=gray] (OutputL3at1.north).. controls +(up:0mm) and +(down:0mm) .. (OutputVat1.south);		

\node[squarednode] at (9.6,-1.5) [draw=none,fill=none](Pts) {${\color{blue}\ldots}$} ;
\draw[-,draw=blue,line width=0.25mm] (OutputL1at1.east)  |- ++(0,0) -- (Pts.west);	

	\node[squarednode] at (12.6, -1.5) [fill=none,minimum width=50mm,minimum height=22mm,dashed](RNNCellL1) {} ;
	\node[squarednode] at (12.6, 0.2) [fill=none,minimum width=50mm,minimum height=12mm,dashed](RNNCelL2) {} ;
	\node[squarednode] at (12.6, 1.4) [fill=none,minimum width=50mm,minimum height=12mm,dashed](RNNCelL3) {} ;	
	\node[squarednode] at (12.6, -0.3) [fill=none,minimum width=52mm,minimum height=48mm,solid](RNNCelL3) {} ;		

	\node[squarednode] at (12.6,-3.5) [draw=none,fill=none](InputM) {$\mathrm{y}_h^{\langle M \rangle}=\left( Q^{\langle M \rangle}, {\color{blue} \mathrm{h}^{\langle M-1 \rangle}}\right)$} ;
\draw[->,draw=blue,line width=0.25mm] (Pts.south)  ++(0,0) |- (InputM.west);	
		
	\node[squarednode] at (12.6,-1.5) [draw=none,fill=none](OutputL1atM) {$\begin{array}{c}
	\mathrm{y}^{[1]}=\sigma^{[1]}\left( \mathrm{W}^{[1]}\mathrm{y}^{[0]} + \mathrm{b}^{[1]}\right)
	\\
	\mathrm{y}^{[0]}=\left( t_M, \mathrm{a}^{\langle M \rangle}\right) 
	\\
	\mathrm{a}^{\langle M \rangle}=\sigma^{[1]}\left(\mathrm{W}_h^{[2]}\mathrm{h}^{\langle M \rangle}+\mathrm{b}_h^{[2]}\right)
	\\	
	{\color{blue}\mathrm{h}^{\langle M \rangle} = \sigma_h^{[1]}\left( \mathrm{W}_h^{[1]}\mathrm{y}_h^{\langle M\rangle} + \mathrm{b}_h^{[1]}\right)}
	\end{array}$} ;
	\draw[->,draw=gray] (InputM.north).. controls +(up:0mm) and +(down:0mm) .. (OutputL1atM.south);				

	\node[squarednode] at (12.6,0.3) [draw=none,fill=none](OutputL2atM) {$\mathrm{y}^{[2]}=\sigma^{[2]}\left( \mathrm{W}^{[2]} \mathrm{y}^{[1]}+ \mathrm{b}^{[2]}\right)$} ;
	\draw[->,draw=gray] (OutputL1atM.north).. controls +(up:0mm) and +(down:0mm) .. (OutputL2atM.south);

	\node[squarednode] at (12.6,1.5) [draw=none,fill=none](OutputL3atM) {$\mathrm{y}^{[3]}= \mathrm{W}^{[3]} \mathrm{y}^{[2]}+ \mathrm{b}^{[3]}$} ;
	\draw[->,draw=gray] (OutputL2atM.north).. controls +(up:0mm) and +(down:0mm) .. (OutputL3atM.south);

	\node[squarednode] at (12.6,3.3) [draw=none,fill=none](OutputVatM) {$\varpi^{\langle M\rangle} = \mathrm{y}^{[3]}$} ;
	\draw[->,draw=gray] (OutputL3atM.north).. controls +(up:0mm) and +(down:0mm) .. (OutputVatM.south);		
\end{tikzpicture}
}
\caption{Iteration of the RNN for $\varpi$, $\mathrm{RNN}_\varpi$, with supply history dependence.}
\label{fig:RNN history Q w}
\end{figure}
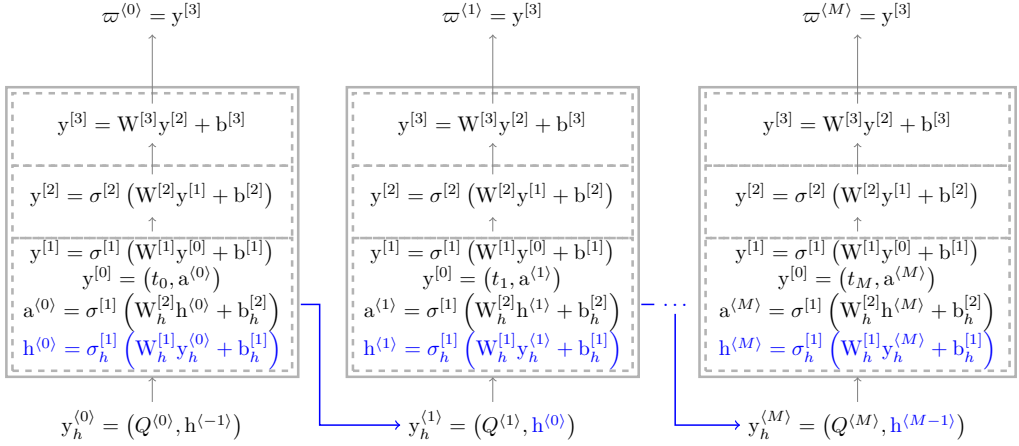

\begin{remark}
The motivation for the implementation of RNN in our case is the following. Rather than looking at the price MFG model as being defined by two variables $u$ and $m$ (in addition to the price), we can take the master equation framework. In this case, we regard $v^*$ as a function of the current time, the state of a player, and the whole distribution of players
\[
	v^*= \mathrm{V}(t,x,m).
\]
Because $m$ is infinite-dimensional, we can not discretize this class of controls directly. However, the dependence on $m$ can be captured by the dependence of $m$ on $Q$. For instance, under the balance condition, the mean distribution of players
\[
	\overline{X}(t) = \int_{\Rr} x m(t,x) dx, \quad t\in[0,T],
\]
satisfies $\dot{\overline{X}} = Q$. Therefore, the mean of $m$ can be recovered using the history of $Q$ only since
\begin{equation*}
	\overline{X}(t) = \int_{\Rr} x m_0(x)dx +\int_0^t Q(s)ds, \quad t \in [0,T].
\end{equation*}
Section \ref{sec: Numerical results} shows that the mean of $m$ characterizes the price equilibrium in the linear-quadratic model. Therefore, capturing the evolution of $Q$ using a recurrent structure provides enough information to approximate the price, which determines $v^*$.

\end{remark}

\begin{remark}
In the first architecture, the approximations of $v^{\langle k \rangle}$ and $\varpi^{\langle k \rangle}$ are of the form
\begin{align*}
	& \mathrm{v}^{\langle k \rangle} = \mathrm{MLP}_v (t_k,X^{\langle k \rangle},\varpi^{\langle k \rangle},\mathrm{v}^{\langle k-1 \rangle}), \quad \varpi^{\langle k \rangle}= \mathrm{MLP}_\varpi (t_k,Q^{\langle k \rangle}),
\end{align*}
for $k=0,\ldots,M-1$, respectively. Therefore, $\mathrm{v}^{\langle k \rangle}$ itself propagates information along the temporal direction, inducing a recurrent dependence, whereas $\varpi^{\langle k \rangle}$ has no recurrent dependence. For the second architecture, we have
\begin{align*}
	& \mathrm{v}^{\langle k \rangle} = \mathrm{RNN}_v (t_k,X^{\langle k \rangle},\varpi^{\langle k \rangle},\mathrm{v}^{\langle k-1 \rangle},Q^{\langle k \rangle},\mathrm{h}^{\langle k-1 \rangle}), \quad \varpi^{\langle k \rangle}= \mathrm{RNN}_\varpi (t_k,Q^{\langle k \rangle},\mathrm{h}^{\langle k-1 \rangle}),
\end{align*}
for $k=0,\ldots,M-1$. Thus, the recurrent dependence in the second architecture relies not only on $\mathrm{v}^{\langle k \rangle}$ but also on the hidden state $\mathrm{h}$ to propagate information for both $v^*$ and $\varpi$. Therefore, compared to the first one, the second architecture represents a bigger class of functions in which the approximations of $v^*$ and $\varpi$ are obtained. To see this, notice that taking all weights associated with the hidden state (which includes the supply) in the second architecture equal to zero, we formally recover the first architecture. 
\end{remark}

\subsection{Neural Network loss function}

As usual in ML framework, we obtain an approximation in a class of NN by minimizing a loss function $\mathcal{L}$. For a fixed architecture, $\mathcal{L}$ maps a parameter $\Theta$  to a real number $\mathcal{L}(\Theta)$. 
Our goal is to minimize $\mathcal{L}$ over the parameters defined by the ML architecture we select.

Recalling the discussion in Section \ref{sec:The mfg price problem}, we consider two NN. The first NN, $\mathrm{NN}_{v}$, approximates $v^*$ in \eqref{eq:optimal feedback} using the parameters $\Theta_{v}$. The second NN, $\mathrm{NN}_{\varpi}$ approximates $\varpi$ using the parameters $\Theta_{\varpi}$. The loss function that we consider is given by the time-discretization of the functional \eqref{eq:Saddle points functional}. Because this functional depends on the approximation of both $v^*$ and $\varpi$ provided by $\mathrm{NN}_{\varpi}$ and $\mathrm{NN}_{v}$, respectively, the loss function is the same regardless of the architecture that we consider. We discretize the functional \eqref{eq:Saddle points functional} using the time discretization. At a time step $t_k$, $\mathrm{NN}_\varpi$ computes $\varpi^{\langle k \rangle}$. The forward-Euler discretization \eqref{eq:NN Forward dynamics} updates the state variable $X^{(i)\langle k+1 \rangle}$ for the $i$-th trajectory, where $\mathrm{NN}_v$ computes $\mathrm{v}^{(i)\langle k \rangle}(\Theta_v)$ for $k=0,\ldots,M-1$ and $i=1,\ldots,N$. After generating these trajectories, we compute the following loss function
\begin{align}\label{eq:Loss adversarial}
\mathcal{L}\left( \Theta_{v},\Theta_{\varpi}\right) 
& = \frac{1}{N} \sum_{i=1}^N \Bigg( \sum_{k=0}^{M-1}   \Delta_t \Big( L(X^{(i)\langle k \rangle},\mathrm{v}^{(i)\langle k \rangle}(\Theta_{v})) 
\\
& \qquad \qquad \qquad + \varpi^{\langle k \rangle}(\Theta_{\varpi})\left( \mathrm{v}^{(i)\langle k \rangle}(\Theta_{v}) - Q^{\langle k \rangle}\right)\Big) + u_T(X^{(i)\langle M \rangle})\Bigg) .\nonumber
\end{align}
Using $\mathcal{L}$, we train $\mathrm{NN}_{v}$ and $\mathrm{NN}_\varpi$ using the algorithm we present in Section \ref{sec:training algorithm}. 

\subsection{Numerical implementation of a posteriori estimates}
Given $v^{(i)\langle k \rangle}$ for $1 \leq i \leq N$  and $1 \leq k \leq M-1$, we compute $X^{(i)\langle k \rangle}$ for $1 \leq i \leq N$  and $1 \leq k \leq M$ using the forward-Euler dynamics \eqref{eq:NN Forward dynamics}. Next, we compute the adjoint state according to
\[
	P^{(i)\langle k \rangle} = - \left( L_v(X^{(i)\langle k \rangle},v^{(i)\langle k \rangle}) + {\varpi^N}^{\langle k \rangle} \right),
\]
for $1 \leq i \leq N$ and $1 \leq k \leq M-1$. Then, the discretization \eqref{eq:NN Forward dynamics} is equivalent to 
\[
	\frac{1}{\Delta_t} \left( X^{(i)\langle k+1 \rangle}-X^{(i)\langle k \rangle} \right) = -H_p(X^{(i)\langle k \rangle}(t),P^{(i)\langle k \rangle}+{\varpi^N}^{\langle k \rangle}), \quad X^{(i)\langle 0 \rangle} = x^i_0,
\]
for $1 \leq i \leq N$ and $1 \leq k \leq M-1$. We compute the error terms in \eqref{eq:HSystem N agents penalized} according to
\begin{equation*}
	\begin{cases}
	\epsilon^{(i)\langle k \rangle} = \frac{1}{\Delta_t} \left( P^{(i)\langle k+1 \rangle}-P^{(i)\langle k \rangle} \right) - H_x(X^{(i)\langle k \rangle}(t),P^{(i)\langle k \rangle}+{\varpi^N}^{\langle k \rangle}) ,
	\\
	\epsilon_T{(i)} = u'_T(X^{(i)\langle M \rangle}) - P^{(i)\langle M \rangle}
	\\
	\epsilon_q^{\langle k \rangle} = \frac{1}{N}\sum\limits_{i=1}^N -H_p\left(X^{(i)\langle k \rangle},P^{(i)\langle k \rangle}+{\varpi^N}^{\langle k \rangle}\right) -Q^{\langle k \rangle}, 
	\end{cases}
\end{equation*} 
for $1 \leq i \leq N$ and $1 \leq k \leq M-1$. We use the discrete $L^2$ norm of the vectors 
\begin{align*}
	\bepsilon^{\langle k \rangle} = (\epsilon^{(1)\langle k \rangle} , \ldots ,\epsilon^{(N)\langle k \rangle}), \quad \bepsilon_T=(\epsilon^{(1)}_T,\ldots,\epsilon^{(N)}_T), \quad \epsilon_q^{\langle k \rangle}
\end{align*}
for $1 \leq k \leq M-1$ to implement the estimate in Theorem \ref{Prop:Main result} and assess the convergence of the training process that we introduce in Section \ref{sec:training algorithm}.


\section{Algorithm formulation and convergence discussion}
\label{sec:training algorithm}

Here, we present the numerical algorithm to approximate $\varpi$ and $v^*$ using NN, and we study its convergence properties. This algorithm relies on a dual-ascent interpretation of the variational problem \eqref{eq:Nagent minmax problem}. 

Notice that the MFG \eqref{eq:MFG system} can be decoupled once $\varpi$ is known. Moreover, if the optimal vector field $v^*$ in \eqref{eq:optimal feedback} is known, the transport equation, the second equation in \eqref{eq:MFG system} for $m$, can be solved by computing the push-forward of $m_0$ under $v^*$. 

Next, we recall some results of convex optimization (see \cite{BoydConvex}, \cite{Ryu2015APO}). Let 
\[
\mathrm{L}: \Rr^n \times \Rr^m \to \Rr\cup\{\pm \infty\}
\] 
be convex in the first variable and concave in the second. In the context of convex optimization, we call
\[
	\inf_{x\in\Rr^n} \sup_{\lambda \in \Rr^m} \mathrm{L}(x,\lambda)\quad \mbox{ and } \quad \sup_{\lambda\in \Rr^m} \inf_{x\in\Rr^n}\mathrm{L}(x,\lambda)
\]
the primal and dual problems generated by $\mathrm{L}$, and we denote its values by $p^\star$ and $d^\star$, respectively. Weak duality states that 
\[
d^\star \leq p^\star,
\]
which always holds. Strong duality states that
\[
d^\star = p^\star,
\]
which is not always satisfied. In finite dimensions, a stronger result can be obtained from the existence of saddle points. We say that $(x^\star,\lambda^\star) \in \Rr^n \times \Rr^m$ is a saddle point of $\mathrm{L}$ if
\[
	\mathrm{L}(x^\star,\lambda) \leq \mathrm{L}(x^\star,\lambda^\star) \leq \mathrm{L}(x,\lambda^\star), \quad x \in \Rr^n, \lambda \in \Rr^m.
\]
If $\mathrm{L}$ possesses a saddle point, then strong duality holds and the infimum and supremum in the primal and dual problems are attained. Conversely, if both the primal and dual problems have solutions, strong duality implies the existence of at least one saddle point.

Constrained optimization offers an important instance of saddle point problems. Let $f: \Rr^n \to \Rr \cup \{\pm \infty \}$, and $g: \Rr^n \to \Rr^m$. To solve the constraint optimization problem
\begin{equation}\label{eq:Constraint min problem general}
 \inf_{x \in \Rr^n} f(x), \quad \mbox{subject to } g(x) =0,
\end{equation}
using the method of multipliers, we select
\begin{equation}\label{eq:Lagrangian method multipliers}
	\mathrm{L}(x,\lambda)= f(x) + \lambda \cdot g(x), \quad x\in\Rr^n, \lambda \in \Rr^m.
\end{equation}
In the case that $\mathrm{L}$ is convex in $x$ and concave in $\lambda$,  
the saddle sub-differential operator
\[
(x,\lambda)\mapsto (\mathrm{L}_x(x,\lambda),-\mathrm{L}_\lambda(x,\lambda))
\]
is a monotone operator. In particular, for $\mathrm{L}$ given by \eqref{eq:Lagrangian method multipliers}, with $f$ convex and $g$ affine, the saddle sub-differential operator defines the Karush-Kuhn-Tucker condition 
\[
	 (\mathrm{L}_x(x,\lambda),-\mathrm{L}_\lambda(x,\lambda)) = (f_x(x) + \lambda g_x(x),-g(x)) = 0,
\]
which characterizes solutions $x^\star$ of the constraint minimization problem \eqref{eq:Constraint min problem general} using a Lagrange multiplier $\lambda^\star$. 

One method to find saddle points of convex-concave functions is the dual ascent method, which considers the sequence of updates
\[
	x^{j+1} = \argmin_{x \in \Rr^n} \mathrm{L}(x,\lambda^j),\quad \lambda^{j+1} = \lambda^j +  \Delta_j g(x^{j+1}),
\]
where $ \Delta_j>0$ is the dual step size. The update on the variable $x$ can be relaxed, as introduced by the Arrow-Hurwicz-Uzawa algorithm (see \cite{Ryu2015APO}). 
\begin{equation*}
	x^{j+1} = x^{j} -  \Delta_j \left( f_x(x^j) + \lambda^j  g_x(x^j)\right),\quad \lambda^{j+1} = \lambda^j +  \Delta_j g(x^{j+1}).
\end{equation*}

As discussed in Section \ref{sec:The mfg price problem}, we are interested in finding saddle points of the functional \eqref{eq:Saddle points functional}. Under Assumptions \ref{hyp:L uniformly convex} and \ref{hyp:uT convex DS}, this functional is convex-concave in $(\bv,\tilde{\varpi}^N)$, and  \eqref{eq:Nagent minmax problem} corresponds to its dual problem. However, the parameters in the NN are not $v$ and $\varpi$ but rather the weights $\Theta_v$ and $\Theta_\varpi$. Thus, instead of performing a descent step in $v$ and ascent step in $\varpi$, we modify Uzawa's algorithm to perform a descent step in $\Theta_v$ and ascent step in $\Theta_\varpi$. This procedure is outlined in Algorithm \ref{alg:our algorithm}. Notice that the updates in Algorithm \ref{alg:our algorithm} do not benefit from the monotone structure of the functional \eqref{eq:Saddle points functional} because the convex-concave structure is not preserved by the dependence of $\mathcal{L}$ on $(\Theta_v,\Theta_\varpi)$. Nevertheless, this algorithm works well in practice, and its convergence can be checked through our a posteriori estimates. 
 
\begin{algorithm}
    \SetKwInOut{Input}{Input}
    \SetKwInOut{Output}{Output}
    \Input{initial density $m_0$, supply $Q = (Q(t_k))_{k =0,\dots,M}$, number of outer iterations $J$, number of time steps $M$, number of samples $N$}
    Initialize $\Theta_v^1,\Theta_\varpi^1$\; 
	\For{$j=1,\ldots,J$}{
	sample $x_0^1,\ldots,x_0^N$ according to $m_0$\;
		\For{$k=0,\ldots,M$}{
			compute $\varpi^{\langle k \rangle}$ using $\mathrm{NN}_\varpi(\Theta_\varpi^j)$\;
			\For{$i=1,\ldots,N$}{
				compute $\mathrm{v}^{(i)\langle k \rangle}$ using $\mathrm{NN}_v(\Theta_v^j)$\;
				compute $X^{(i)\langle k+1 \rangle}$ according to \eqref{eq:NN Forward dynamics}\;
				}
			}				
	compute $\mathcal{L}(\Theta_v^j,\Theta_\varpi^j)$ according to \eqref{eq:Loss adversarial}\;
	compute $\Theta_v^{j+1}$ by updating $\Theta_v^{j}$ in the descent direction $\mathcal{L}_{\Theta_v}(\Theta_v^j,\Theta_\varpi^j)$\;
	\For{$k=0,\ldots,M$}{
			compute $\varpi^{\langle k \rangle}$ using $\mathrm{NN}_\varpi(\Theta_\varpi^j)$\;
			\For{$i=1,\ldots,N$}{
				compute $\mathrm{v}^{(i)\langle k \rangle}$ using $\mathrm{NN}_v(\Theta_v^{j+1})$\;
				compute $X^{(i)\langle k+1 \rangle}$ according to \eqref{eq:NN Forward dynamics}\;
				}
			}	
	compute $\mathcal{L}(\Theta_v^{j+1},\Theta_\varpi^j)$ according to \eqref{eq:Loss adversarial}\;		
	compute $\Theta_\varpi^{j+1}$ by updating $\Theta_\varpi^{j}$ in the ascent direction $\mathcal{L}_{\Theta_\varpi}(\Theta_v^{j+1},\Theta_\varpi^j)$\;
	}
    \Output{$\Theta_\varpi^J$, $\Theta_v^J$}
    \caption{Training algorithm}
    \label{alg:our algorithm}
\end{algorithm}


The training procedure in Algorithm \ref{alg:our algorithm} can be interpreted as an adversarial-like training between $\mathrm{NN}_v$ and $\mathrm{NN}_\varpi$. At each training step, by sampling $N$ initial positions according to $m_0$, we introduce a sample of the optimization problem faced by a representative agent. Updating $\Theta_v$ in the descent direction of  \eqref{eq:Loss adversarial} when $\Theta_\varpi$ is fixed corresponds to the approximation of the best strategy to minimize the cost functional per agent. Maximizing $\Theta_\varpi$ in the ascent direction of  \eqref{eq:Loss adversarial} when $\Theta_v$ is fixed penalizes the deviation of the current best strategy from the balance condition. Thus, iterating over training steps, both NN updates their values according to the best response of the other.


\section{Numerical results}
\label{sec: Numerical results}

Here, we present the results of implementing Algorithm \ref{alg:our algorithm} to the price formation model with a mean-reverting supply. We consider quadratic and non-quadratic cost structures. We illustrate the use of the a posteriori estimates presented in Section \ref{sec:A posteriori estimates} to assess the convergence of the training method. We consider two different behaviors for the supply function to compare the performance of the MLP and the RNN architectures. The results show that both NN provide accurate approximations for $v^*$ and $\varpi$. 

We set $T=1$ and $M=30$ time steps equally spaced for the time discretization. We assume that the supply satisfies the mean-reverting ODE 
\[
	Q'(t) = \left(\overline{Q}(t) -Q(t)\right) , \quad Q(0)=q_0, \quad t \in [0,T]
\]
towards the mean supply $\overline{Q}$. We consider two different functions $\overline{Q}$: constant and oscillating functions. Thus,
\[
	Q(t)=q_0 e^{-t} + e^{-t} \int_0^t e^s \overline{Q}(s)ds, \quad t \in [0,T].
\]
The choice of $M$ guarantees that, after discretization of the time variable, we retain the features of the quasi-periodic supply corresponding to the oscillatory $\overline{Q}$. We take $\overline{m}_0=-\tfrac{1}{4}$ and $m_0 \sim \mathcal{N}(\overline{m}_0,0.4)$ for the initial distribution. After several numerical experiments, we set the sample size for the training process to $N=10$. The choice of $N$ guarantees accurate results with low computational cost. 

To avoid tailored results depending on the cost structure (quadratic or non-quadratic) or the supply dynamics (oscillatory or non-oscillatory), we use the NN hyper-parameters (the number of layers and neurons) specified in Section \ref{sec:NN for price}. Moreover, we iterate Algorithm \ref{alg:our algorithm} for $200.000$ iterations without regard for the cost or supply form. The a posteriori estimate in Theorem \ref{Prop:Main result} is computed for the training set of initial positions every $500$ iterations (1 epoch). After training, we compare the results of the two architectures in each scenario.

\subsection{Quadratic cost}
We select
\[
	L(x,v)=\frac{\eta}{2} \left( x - \kappa\right)^2 + \frac{c}{2} v^2, \quad \mbox{and} \quad u_T\left(x\right) = \frac{\gamma}{2}\left(x-\zeta\right)^2,
\]
where $\kappa,\zeta \in \Rr$, $\eta,\gamma\geq 0$, and $c>0$. In this setting, the Hamilton-Jacobi equation in \eqref{eq:MFG system} admits quadratic solutions in $x$ with time-dependent coefficients 
\[
	u(t,x)=a_0(t) + a_1(t)x + a_2(t)x^2, \quad t\in[0,T], \; x \in \Rr,
\]
where the coefficients $a_0$, $a_1$, and $a_2$ are determined by an ODE system (\cite{gomes2018mean}, \cite{gomes2021randomsupply}). Moreover, the price admits the following explicit formula
\begin{equation}\label{eq:LQ price formula}
	\varpi(t)=\eta \left(\kappa-\overline{m}_0\right)\left(T-t\right)+\gamma\left( \zeta - \overline{m}_0\right) - \eta \int_t^T \int_0^s Q(r) dr ds - \gamma \int_0^T Q(s) ds - c Q(t),
\end{equation}
for $t\in [0,T]$, where $\overline{m}_0 = \int_{\Rr} x m_0(x) dx$, and \eqref{eq:optimal feedback} reduces to
\[
	v^*(t,x) = - \frac{1}{c}\left(\varpi(t) + a_1(t) + 2a_2(t)x\right), \quad t\in[0,T], \; x \in \Rr.
\]
We use the previous expressions as benchmarks for the approximation obtained using NN.
To obtain closed-form solutions, we select the parameters 
\[
	c=1,\quad  \gamma=e^{-1},\quad \zeta=1,\quad \eta=1,\quad \kappa=1.
\]
Because of the dependence of the price on the mean of $m_0$ (see \eqref{eq:LQ price formula}), the price approximation improves as the sample mean approaches $\overline{m}_0$. However, because the convergence in the law of large numbers is not monotone in $N$, a large sample, for instance, $N=1000$, does not necessarily provide better results than a small sample, for instance, $N=10$. 

\subsubsection{\bf Constant mean-reverting function}
The first case we consider is 
\[
	\overline{Q}(t) \equiv 1, \quad q_0=\tfrac{1}{10}.
\]
We illustrate the analytical solutions for $\varpi$, $v^*$, and $m$ in Figure \ref{fig:Qbar constant analres}. The plot for $m$ includes the characteristic curves with initial position $x_0 \in [-1,1]$. Because in Algorithm \ref{alg:our algorithm} we sample $x_0$ according to $m_0$, we see that regions in the $(t,x)$ plane where $m(t,x)$ is high are likely to be better explored by $\mathrm{MLP}_v$. Moreover, the approximation error should be highlighted in dense areas. Thus, we evaluate the error in the approximation of $v^*$  over the domain $[0,T]\times [-1,1]$ using not only the absolute error but also the $m$-weighted error; that is,
\[
	\left|v^*(t,x)-\mathrm{MLP}_v(t,x) \right|m(t,x), \quad (t,x)\in[0,T]\times[-1,1].
\]
 
\begin{figure}[htp]
     \centering
     \begin{subfigure}[t]{0.32\textwidth}
         \centering
         \caption{Supply and price}
         \includegraphics[width=\textwidth]{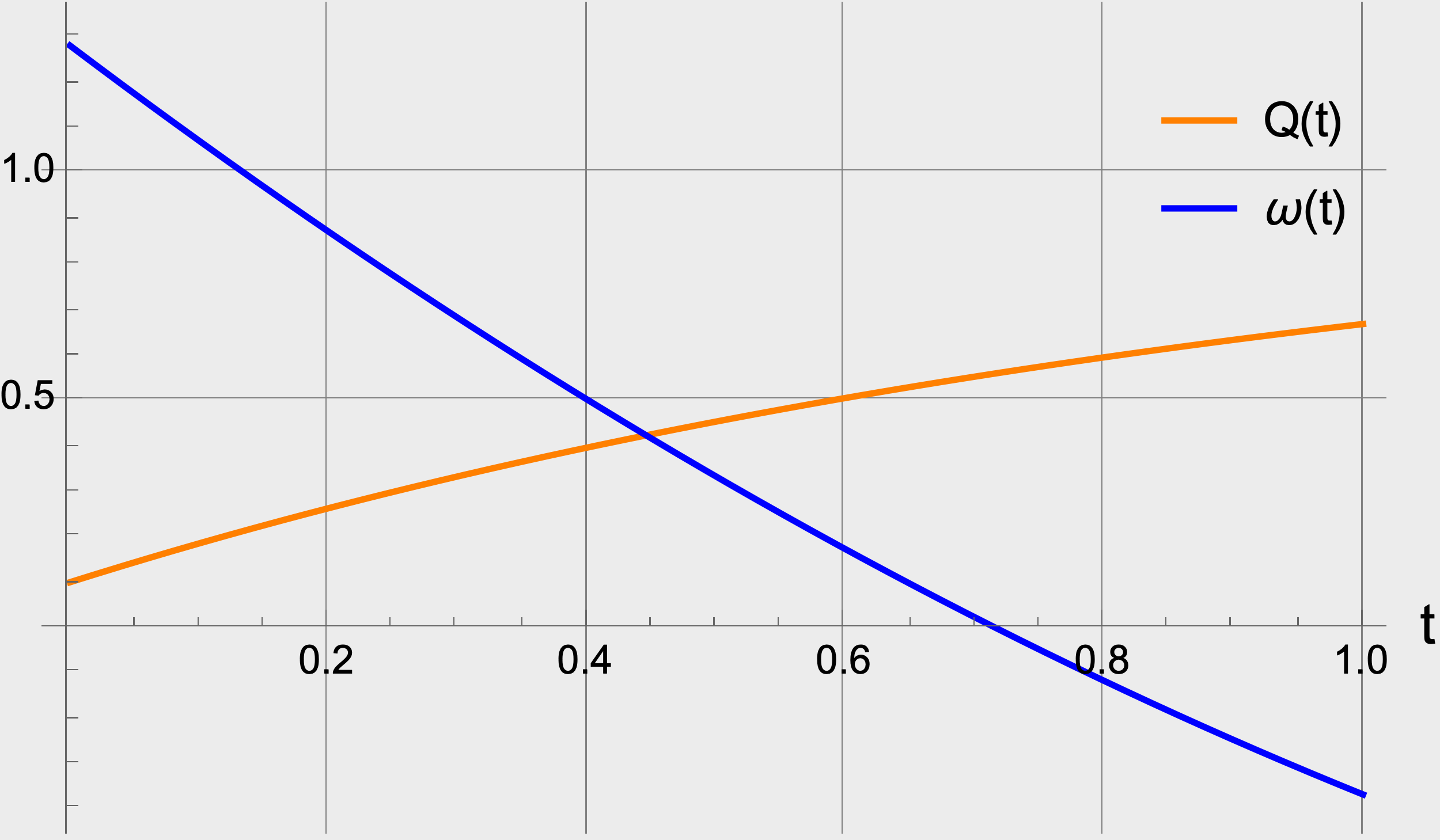}
     \end{subfigure}
     \hfill
     \begin{subfigure}[t]{0.32\textwidth}
         \centering
        \caption{Optimal feedback $v^*$}
         \includegraphics[width=\textwidth]{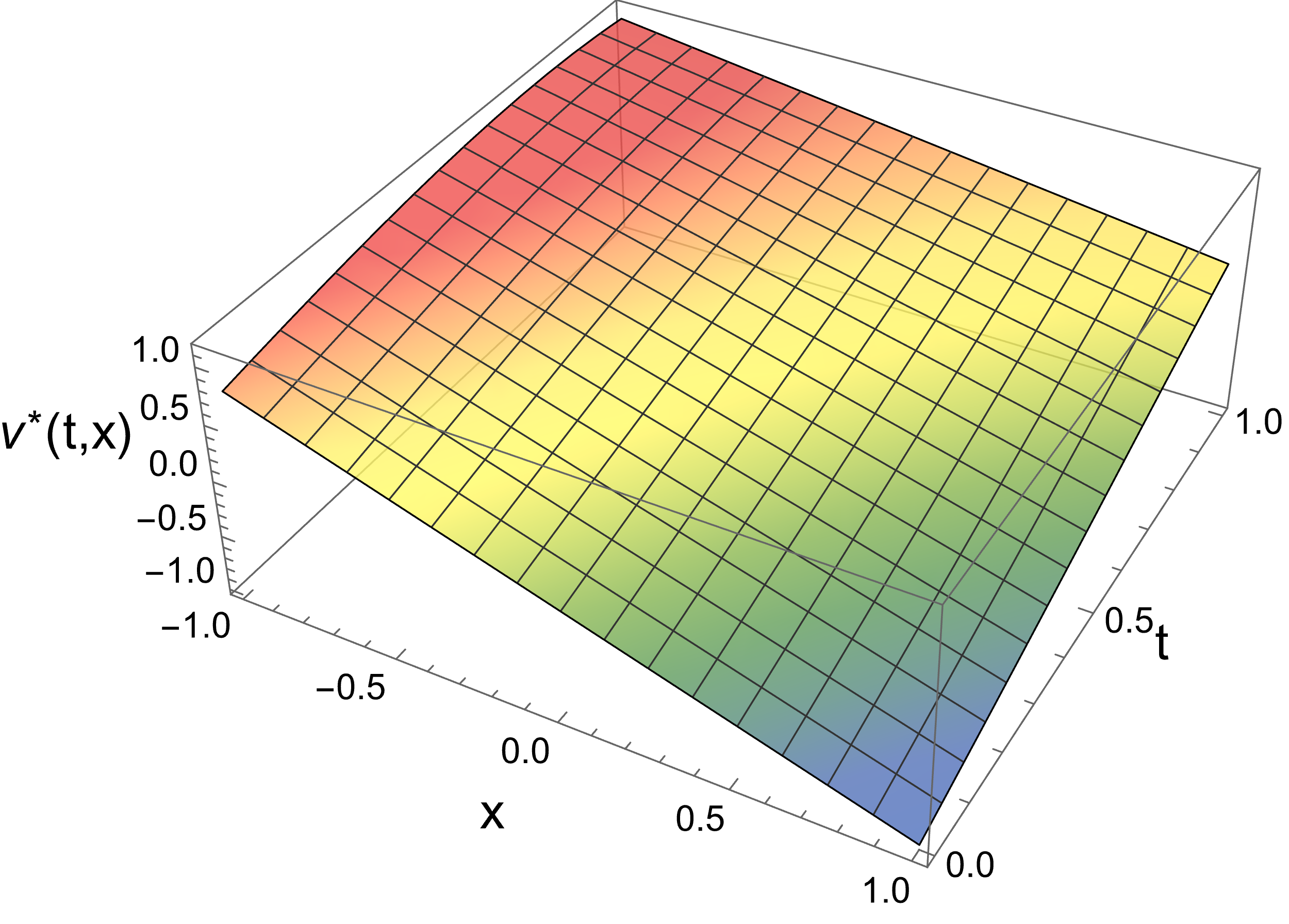}
     \end{subfigure}
     \hfill
     \begin{subfigure}[t]{0.32\textwidth}
         \centering
        \caption{$m$ and characteristics}
         \includegraphics[width=\textwidth]{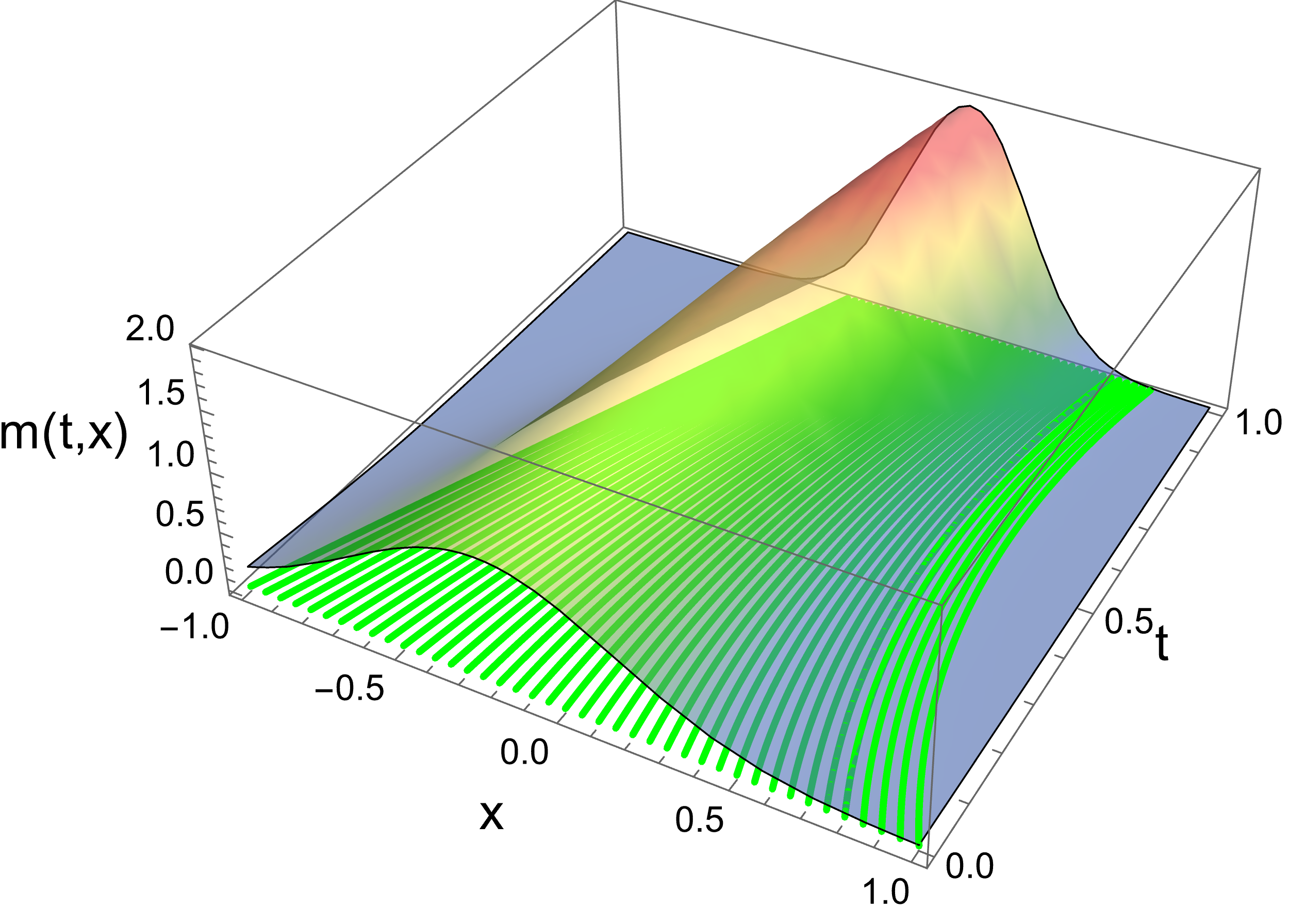}
     \end{subfigure}
        \caption{Analytical solutions for $\overline{Q}=1$ with quadratic cost.}
        \label{fig:Qbar constant analres}
\end{figure}

{\bf MLP with instantaneous feedback.} 
In Figure \ref{fig:MLP Qbar constant results}, we show the price approximation obtained by $\mathrm{MLP}_\varpi$ and the error on the approximated values for $v^*$ obtained by $\mathrm{MLP}_v$. We see that the $m$-weighted error for $v^*$ increases as we march on time. It is important to recall that the benefit of anticipating the optimal control decreases as we approach the terminal time. Notice also that the terminal cost $u_T$ incentives agents to concentrate towards $\zeta$, which causes $m$ to peak as we reach $T$. The error approximation for both $\varpi$ and $v^*$ is of an order $10^{-2}$. 

\begin{figure}[htp]
     \centering
     \begin{subfigure}[b]{0.4\textwidth}
         \centering
        \caption{Analytic price vs. approx.}
         \includegraphics[width=\textwidth]{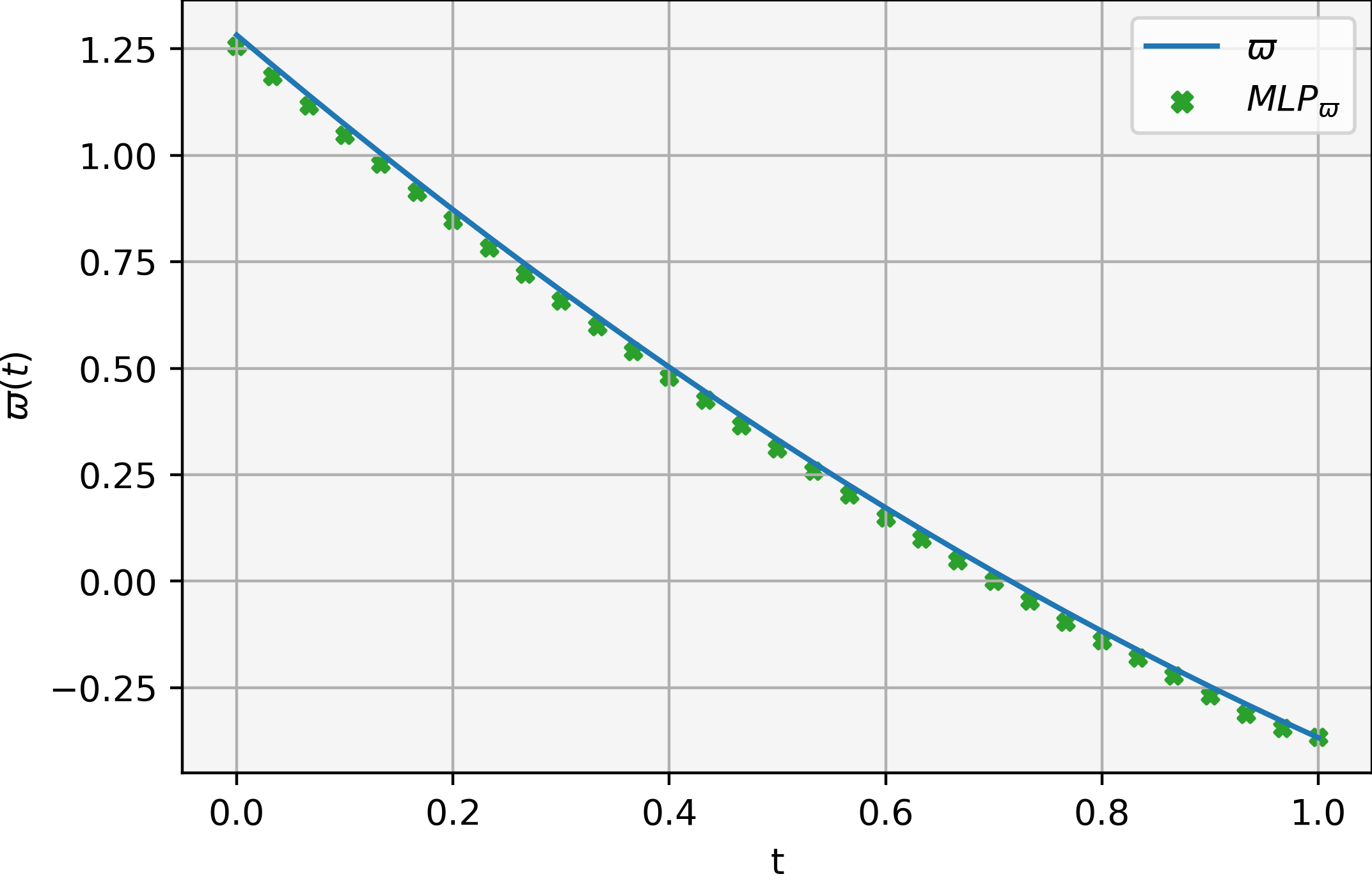}
     \end{subfigure}   
     \hskip0.3cm  
     \begin{subfigure}[b]{0.4\textwidth}
         \centering
        \caption{Price approx. error}
         \includegraphics[width=\textwidth]{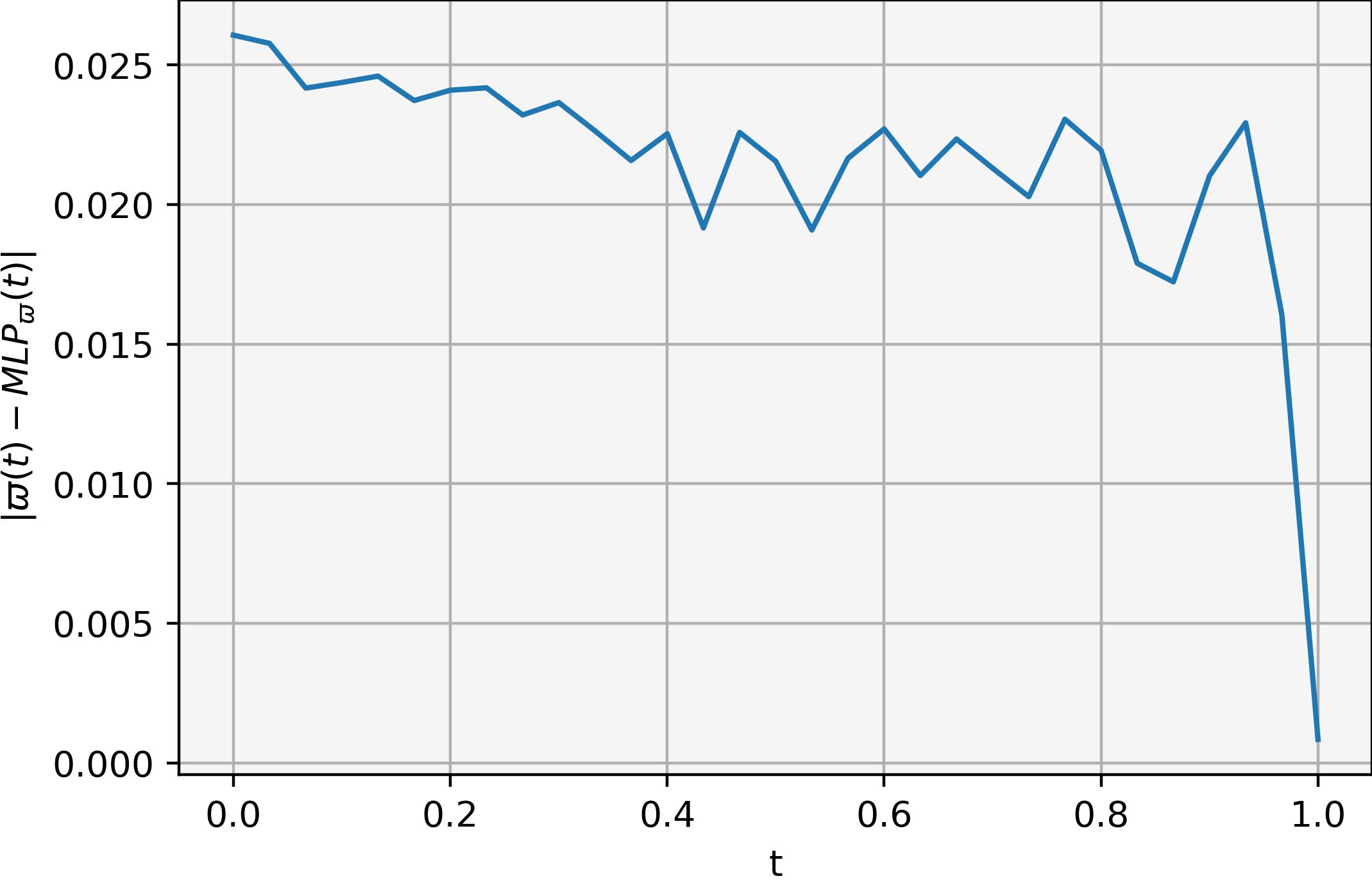}
     \end{subfigure}      
     
	\begin{subfigure}[b]{0.29\textwidth}
         \centering
         \caption{$v^*$ obtained by $\mathrm{MLP}_v$}
         \includegraphics[width=\textwidth]{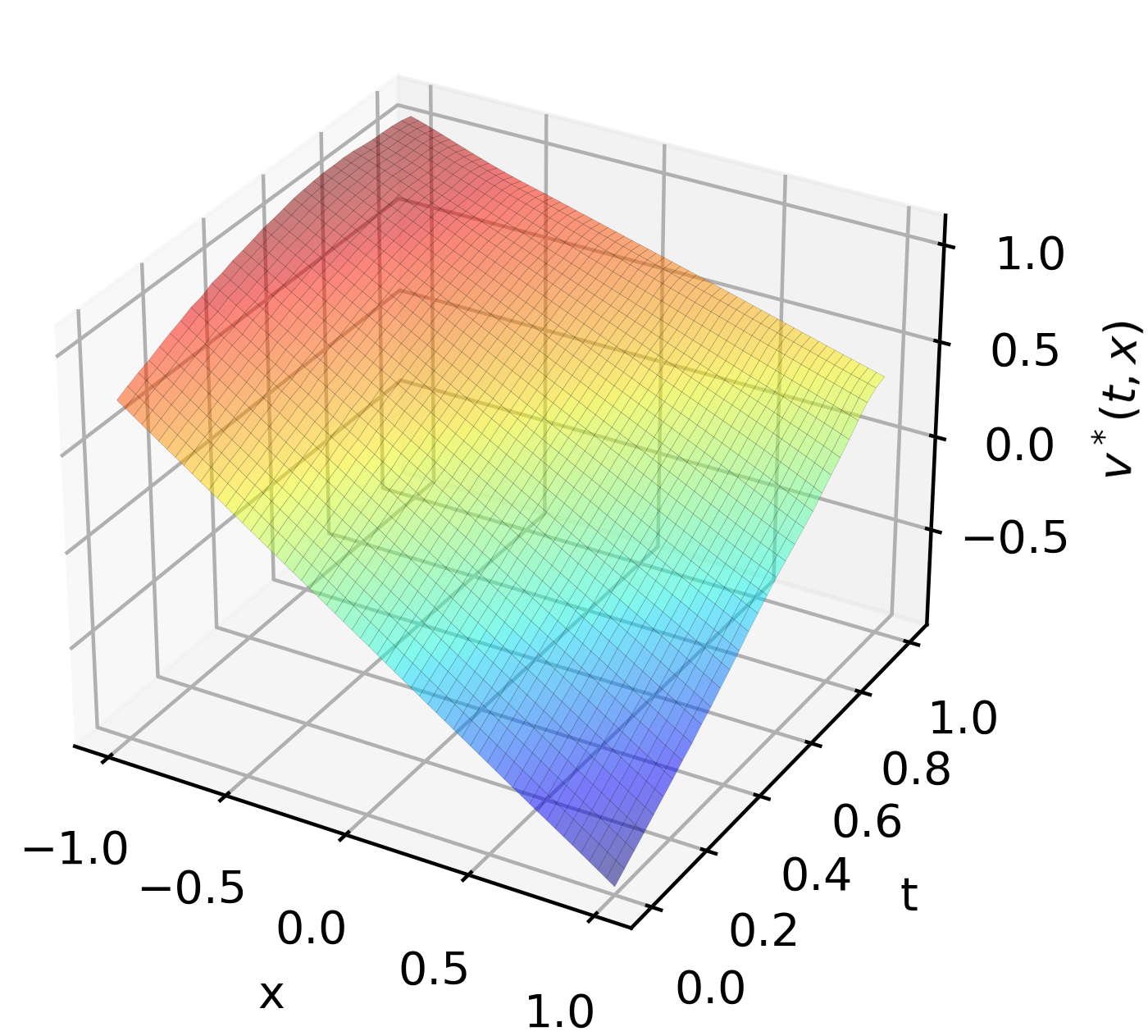}
     \end{subfigure}  
      \hfill  
     \begin{subfigure}[b]{0.29\textwidth}
         \centering
         \caption{Absolute error ($v^*$)}
         \includegraphics[width=\textwidth]{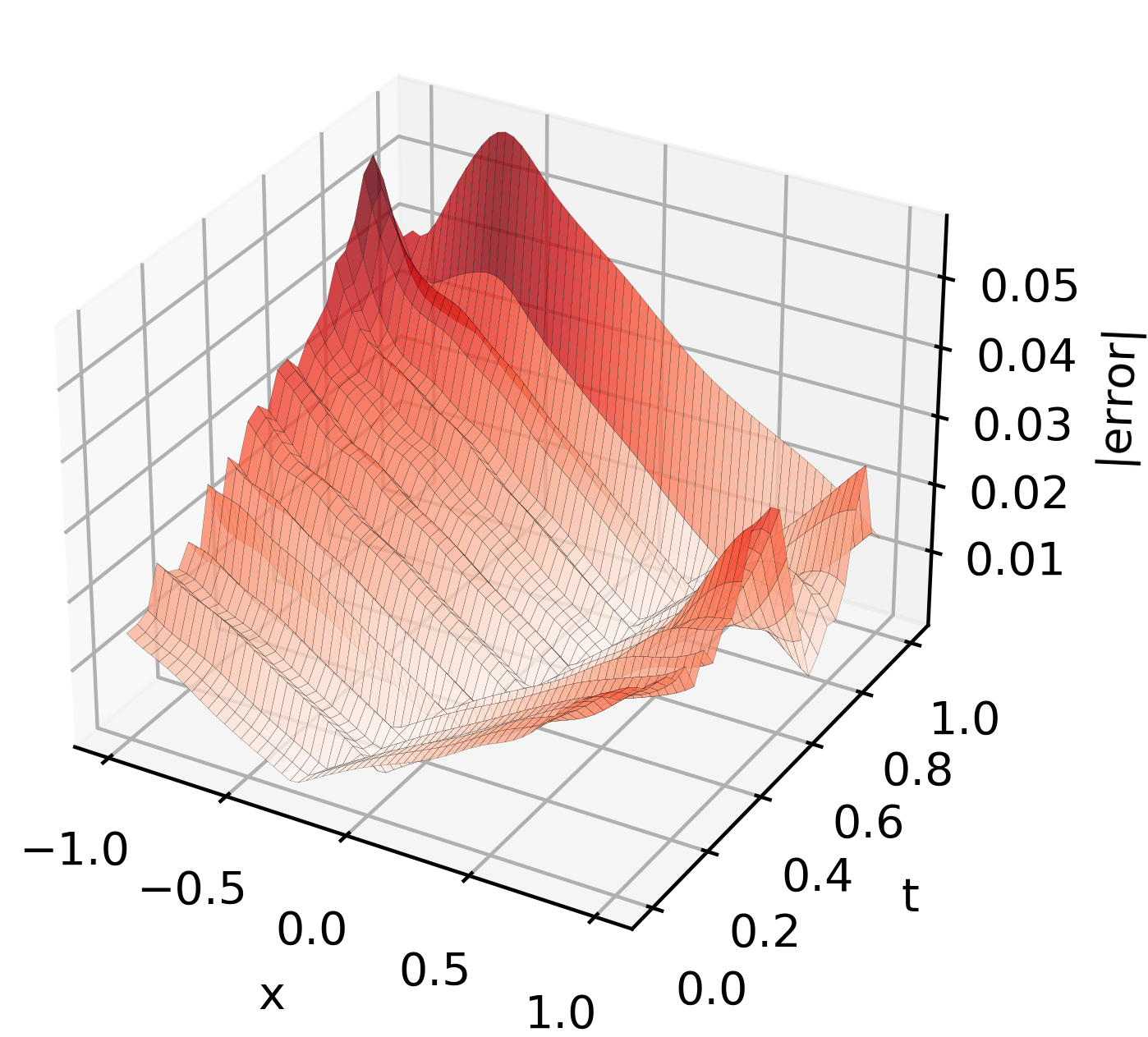}
     \end{subfigure}
     \hfill
     \begin{subfigure}[b]{0.29\textwidth}
         \centering
         \caption{$m$-weighted error ($v^*$)}
         \includegraphics[width=\textwidth]{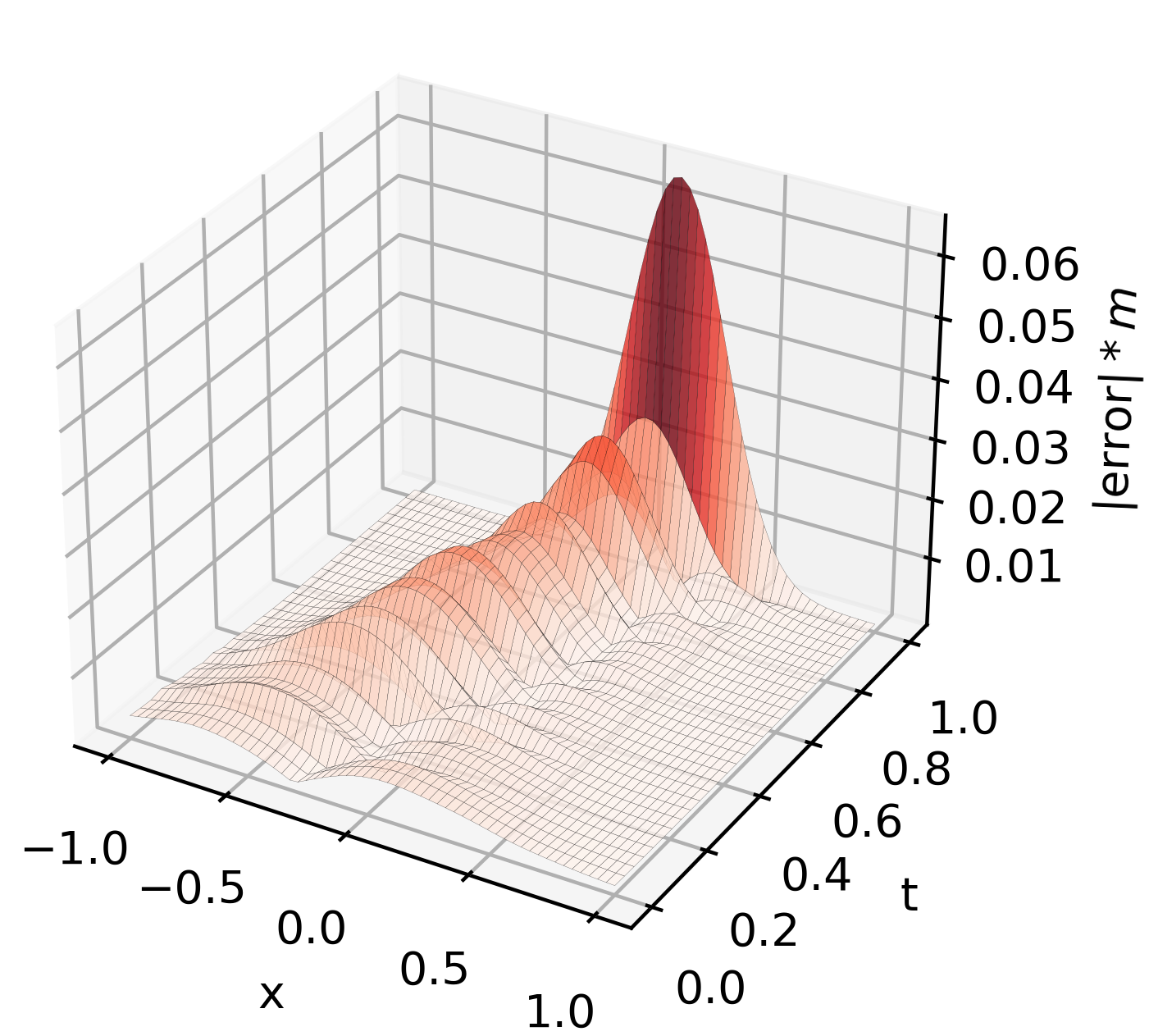}
     \end{subfigure}
        \caption{$\varpi$ (top) and $v^*$ (bottom) using $\mathrm{MLP}_\varpi$ and $\mathrm{MLP}_v$, respectively, for $\overline{Q}=1$ with quadratic cost.}
        \label{fig:MLP Qbar constant results}
\end{figure}

The a posteriori estimate \eqref{eq:EL discrete a posteriori} is depicted in Figure \ref{fig:MLP Qbar constant training} for the sample of initial positions used during training. The error at terminal positions, $\bepsilon_T$, stabilizes after 150 epochs, while the running error, $\bepsilon$, seems to stabilize after 350 epochs. The error in the balance condition, $\epsilon_q$, has a dominating oscillatory behavior. Aggregating all errors, we observe a persistent decay in the first 100 epochs, reaching an order $10^{-1}$ in the a posteriori estimate. The estimate oscillates in the remaining epochs (100 to 400), ending with a magnitude of an order $10^{-1}$. The price error oscillates with a consistent decay, showing the approximation improvement.

\begin{figure}[htp]
     \centering
     \begin{subfigure}[b]{0.48\textwidth}
     \centering
        \caption{Residuals (training)}
         \includegraphics[width=\textwidth]{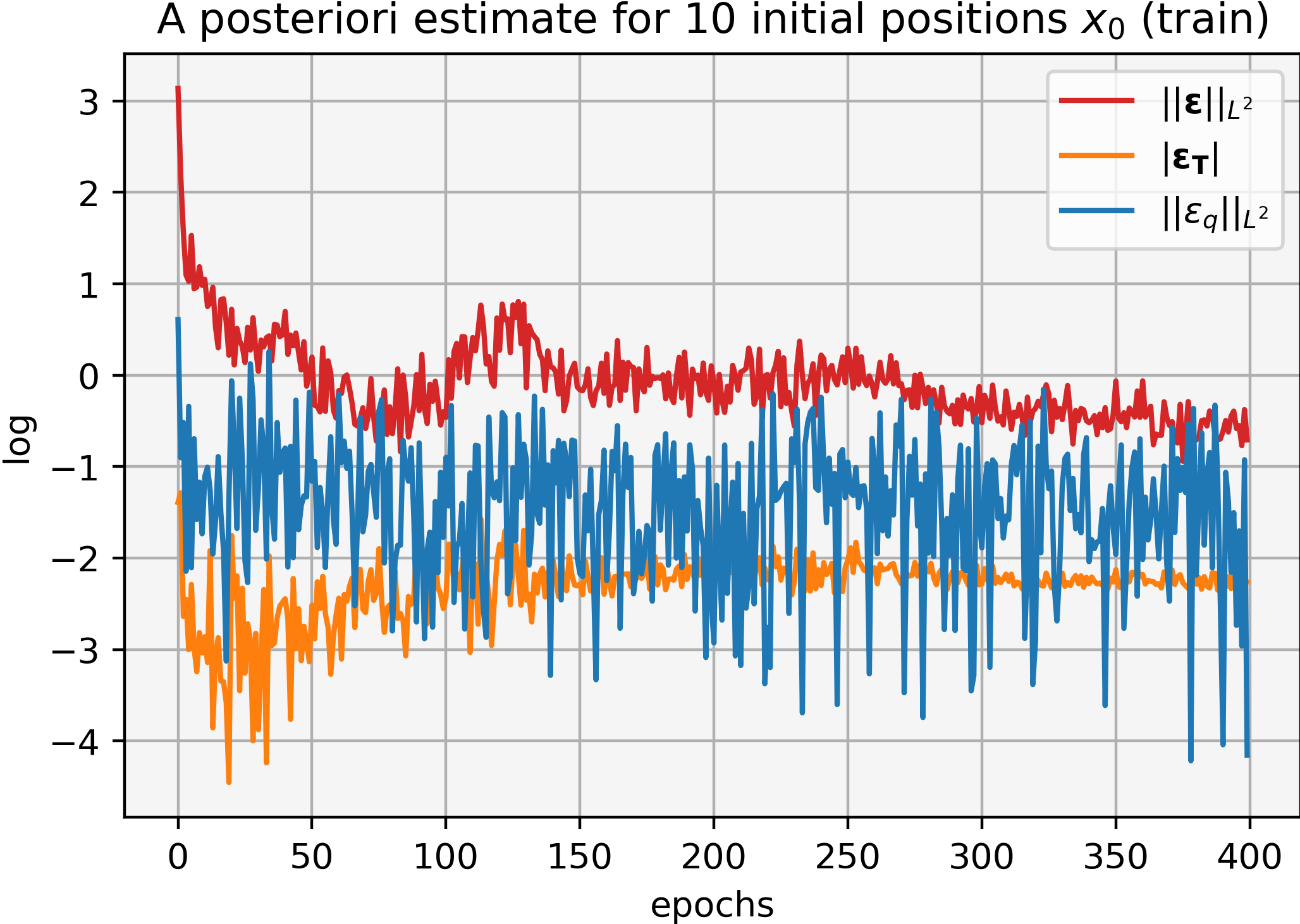}
          \end{subfigure}
     \hfill
     \begin{subfigure}[b]{0.48\textwidth}
     \centering
        \caption{Estimate vs. price error (training)}
         \includegraphics[width=\textwidth]{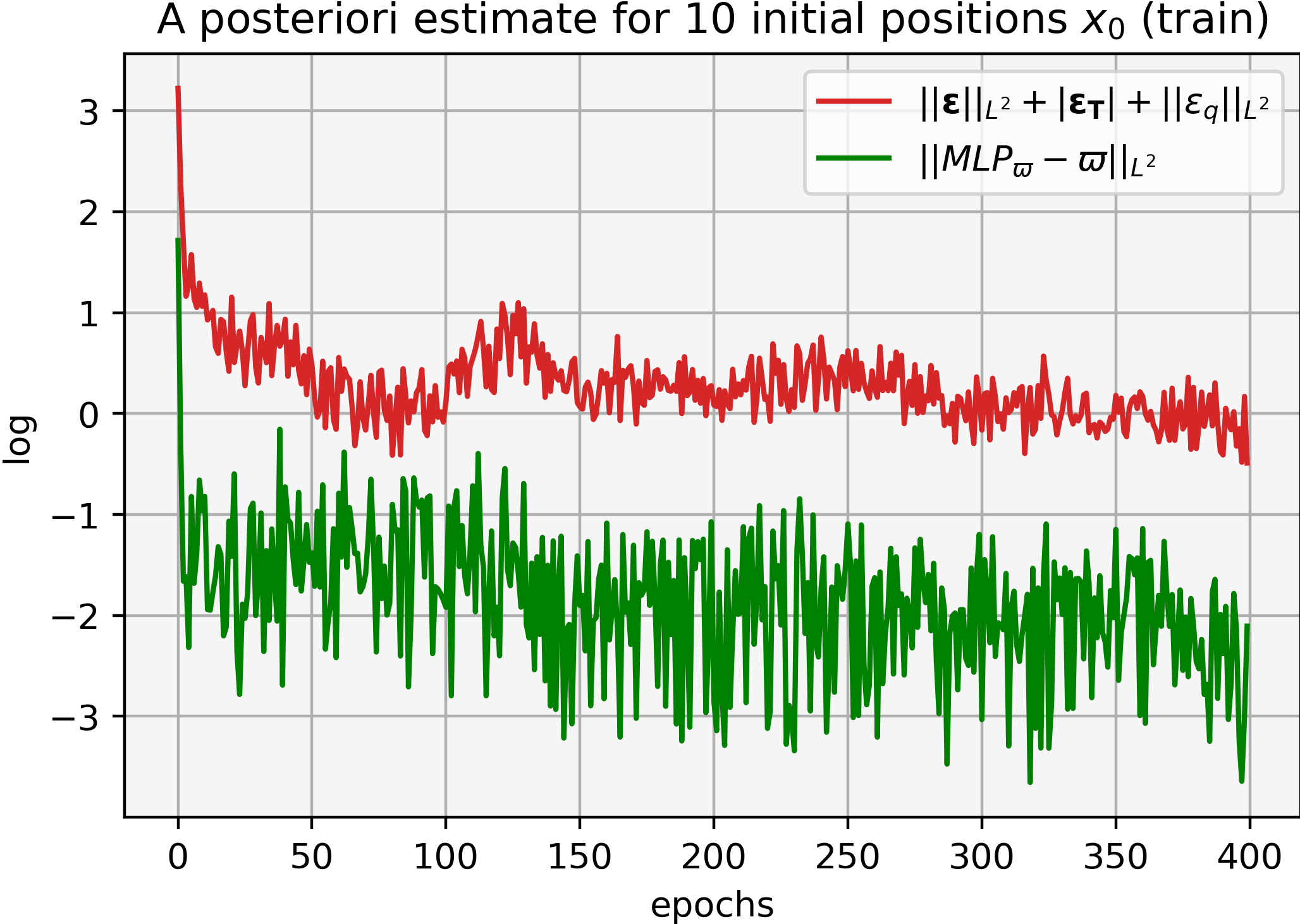}
          \end{subfigure}
         \caption{Residuals (left), a posteriori estimate \eqref{eq:EL discrete a posteriori} and price error (right) during training (log. scale)  using $\mathrm{MLP}_\varpi$ and $\mathrm{MLP}_v$ for $\overline{Q}=1$ with quadratic cost.}
        \label{fig:MLP Qbar constant training}
\end{figure}

{\bf RNN with supply history.} 
Figure \ref{fig:RNNMLP Qbar constant results} shows the price approximation and error values for $v^*$ using $\mathrm{RNN}_\varpi$ and $\mathrm{RNN}_v$. As before, the $m$-weighted error for $v^*$ increases as we march on time, due to the reasoning presented previously. The price and optimal control approximation error are of an order $10^{-2}$.

\begin{figure}[htp]
     \centering
     \begin{subfigure}[b]{0.4\textwidth}
         \centering
        \caption{Analytic price vs. approx.}
         \includegraphics[width=\textwidth]{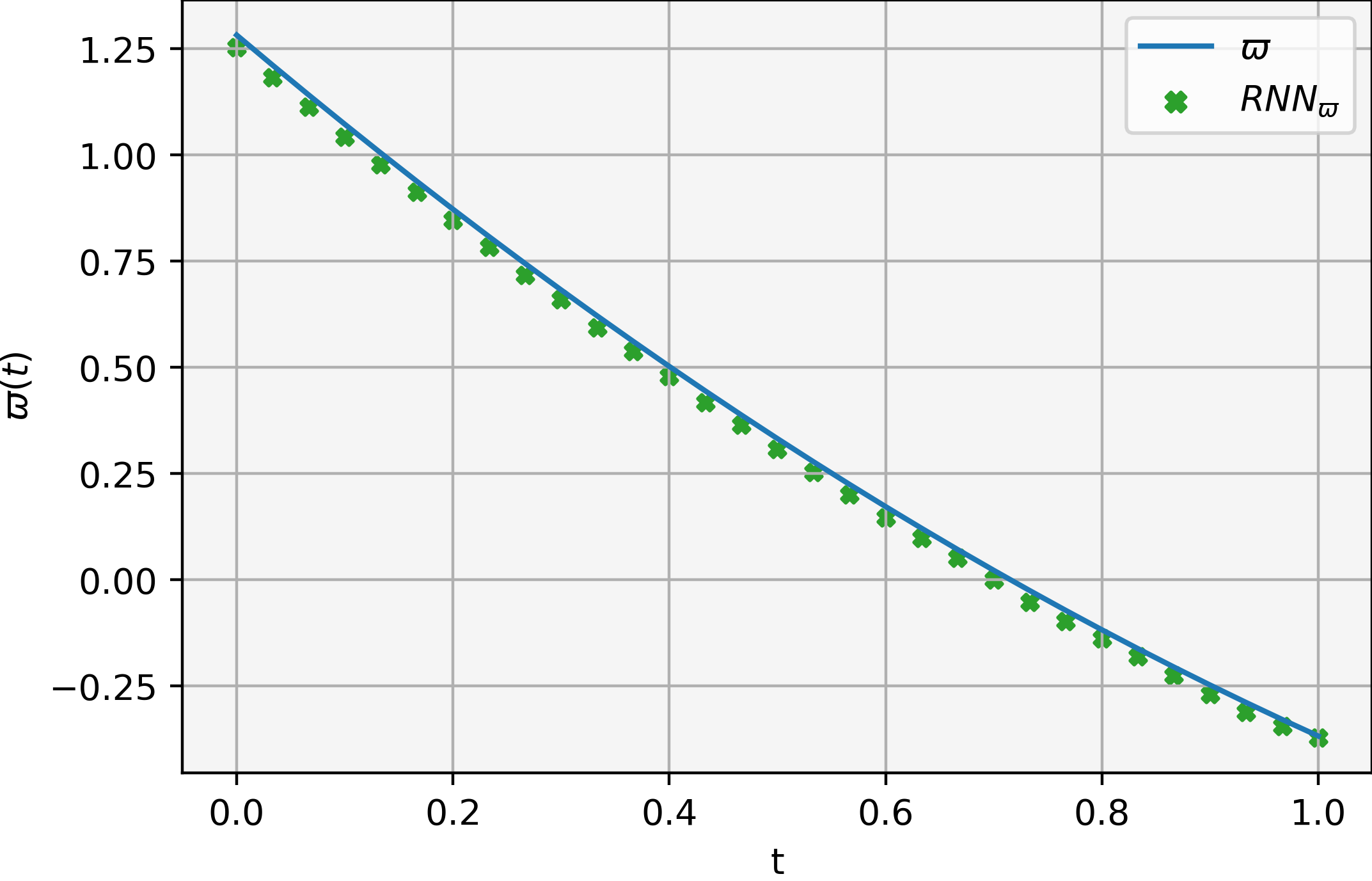}
     \end{subfigure}   
     \hskip0.3cm
     \begin{subfigure}[b]{0.4\textwidth}
         \centering
        \caption{Price approx. error}
         \includegraphics[width=\textwidth]{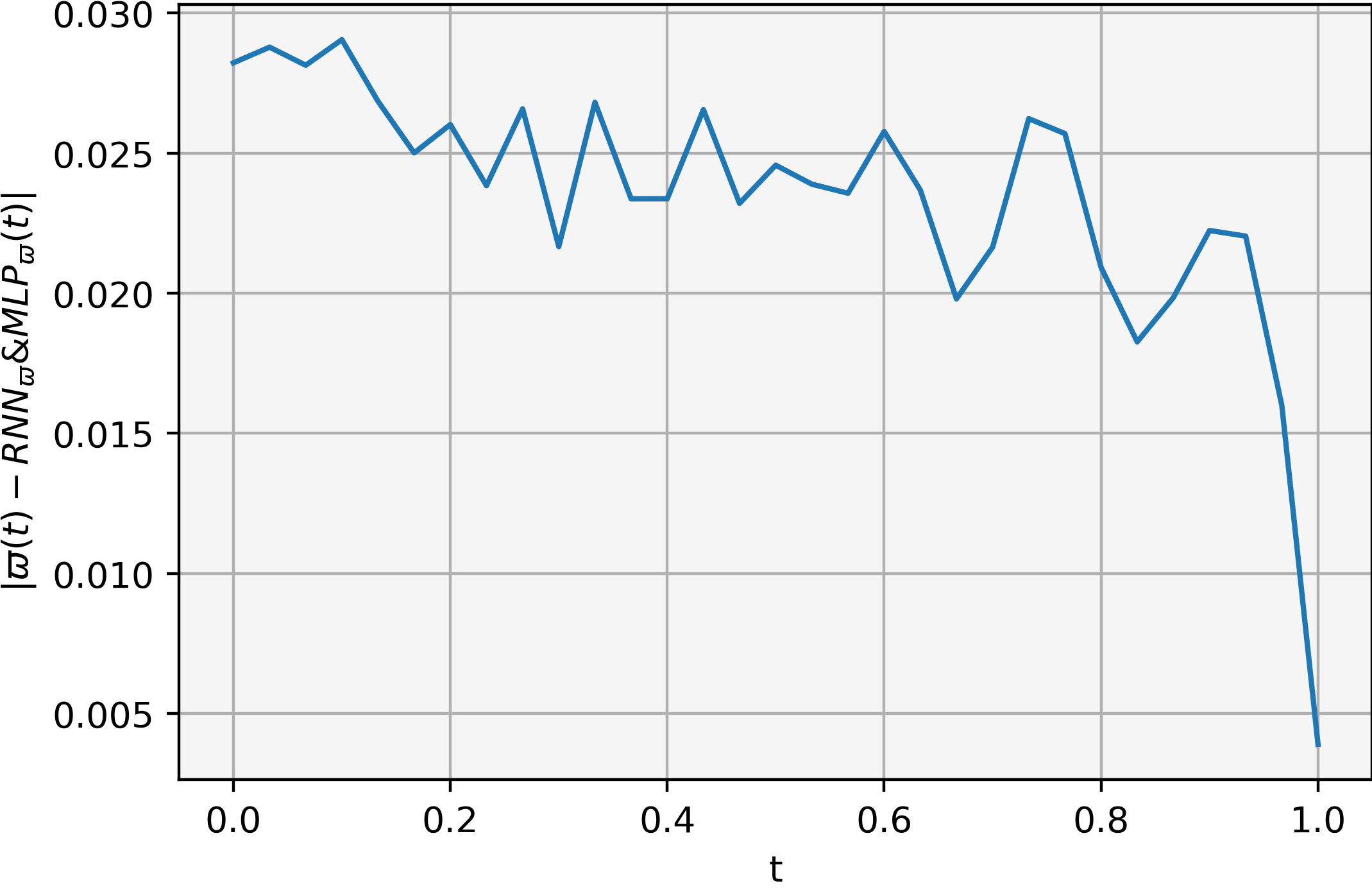}
     \end{subfigure}   
     
     \begin{subfigure}[b]{0.29\textwidth}
         \centering
         \caption{$v^*$ obtained by $\mathrm{RNN}_v$}
         \includegraphics[width=\textwidth]{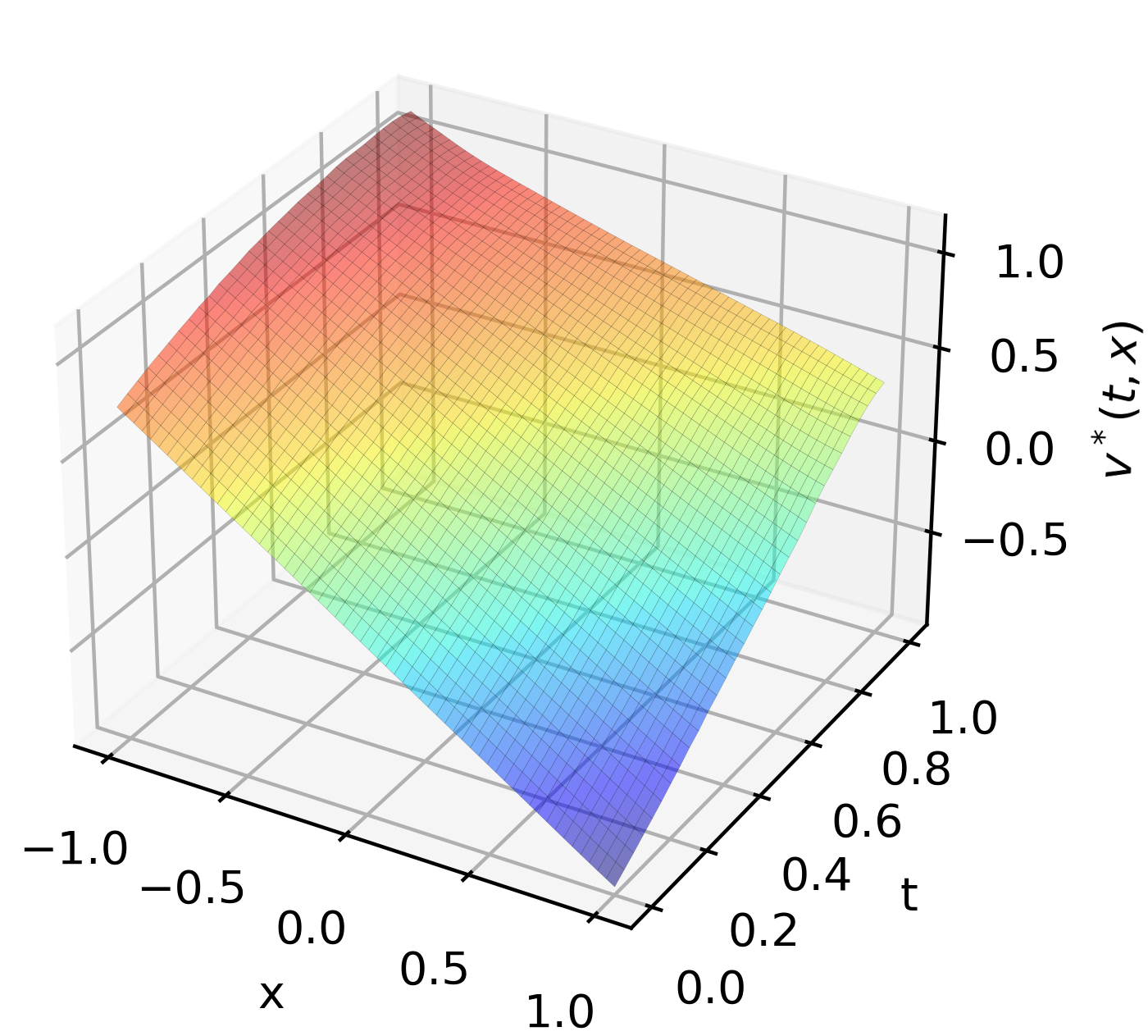}
     \end{subfigure}
     \hfill     
     \begin{subfigure}[b]{0.29\textwidth}
         \centering
         \caption{Absolute error ($v^*$)}
         \includegraphics[width=\textwidth]{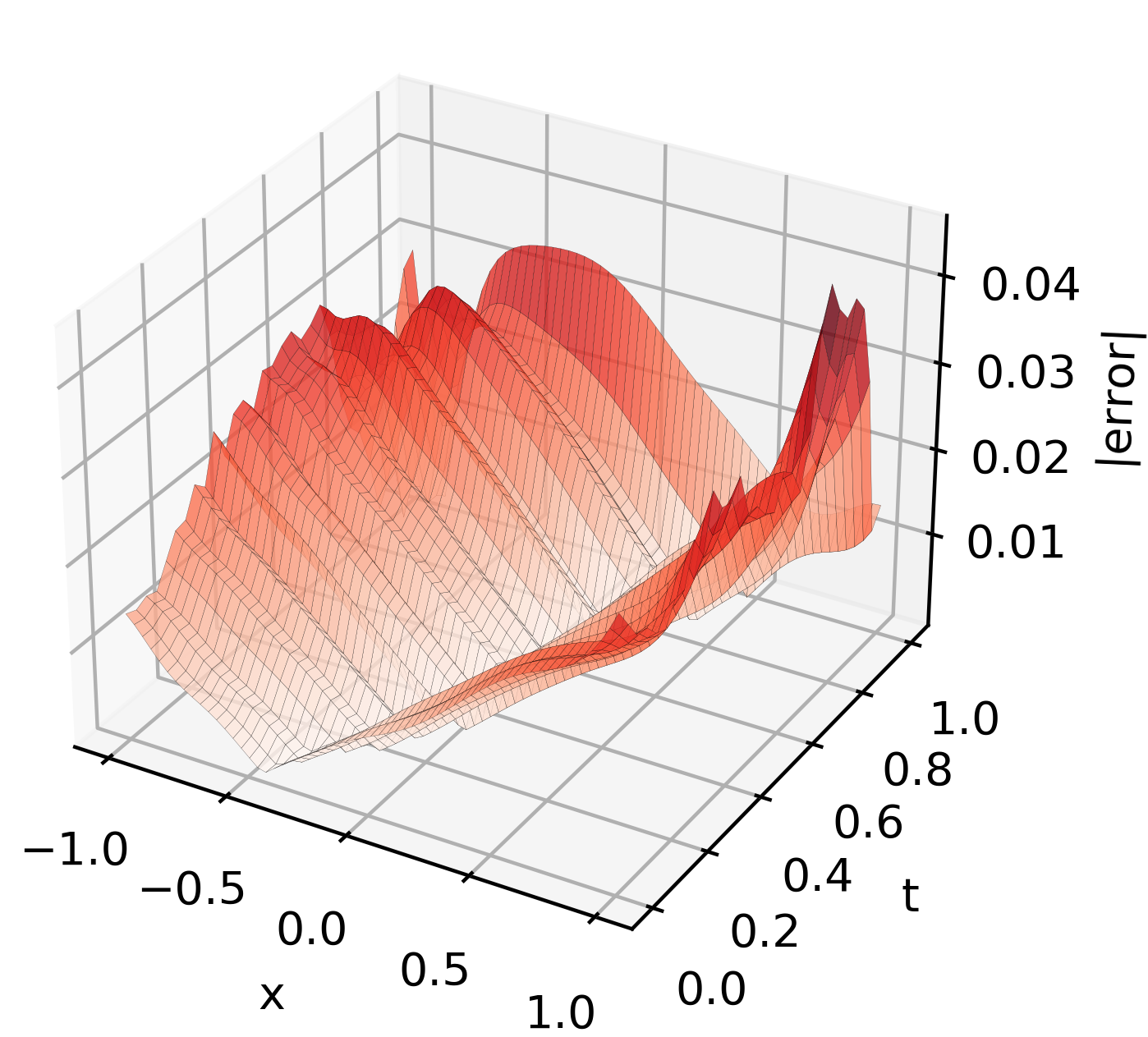}
     \end{subfigure}
     \hfill
     \begin{subfigure}[b]{0.29\textwidth}
         \centering
         \caption{$m$-weighted error ($v^*$)}
         \includegraphics[width=\textwidth]{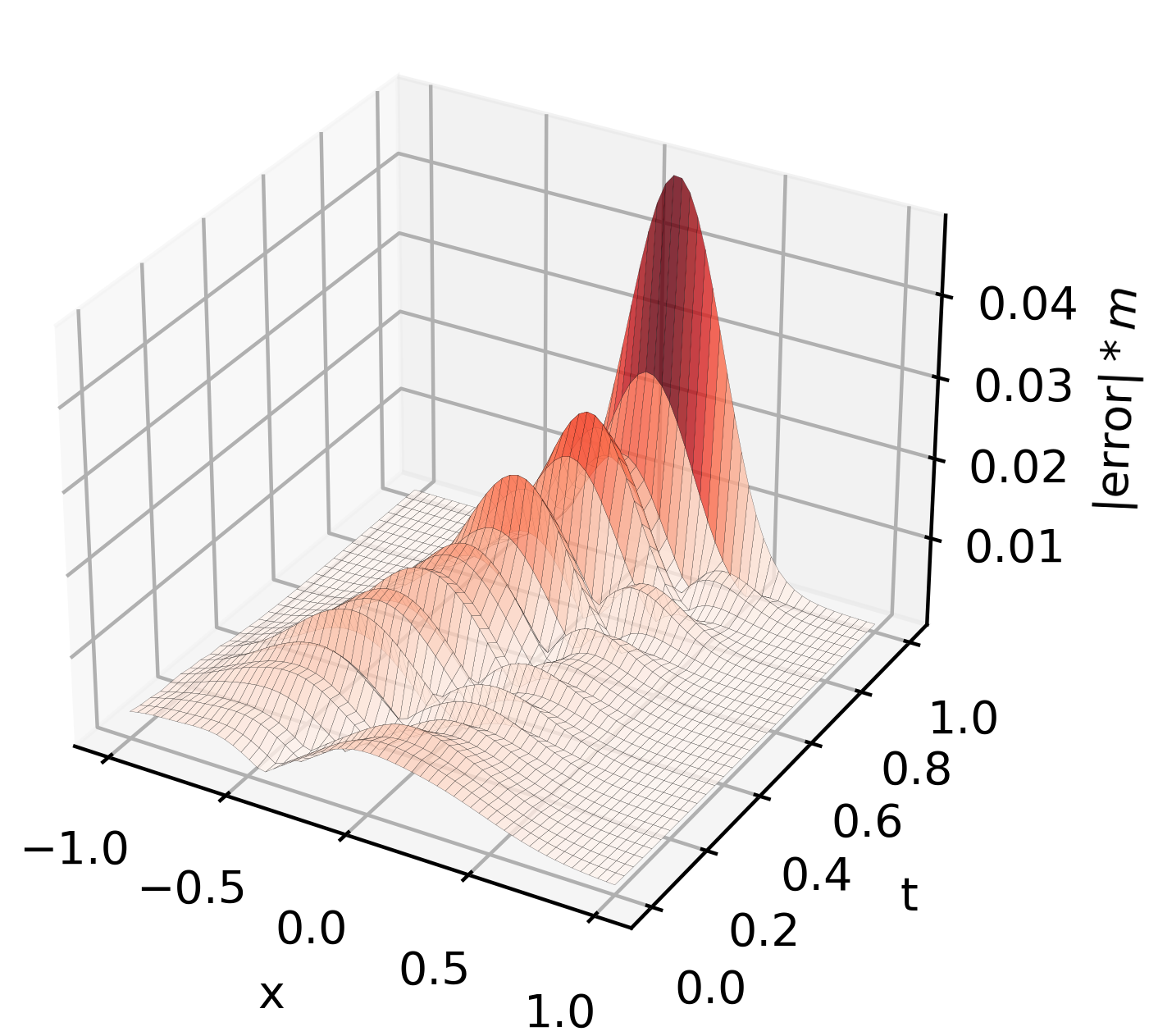}
     \end{subfigure}
        \caption{$\varpi$ (top) and $v^*$ (bottom) using $\mathrm{RNN}_\varpi$ and $\mathrm{RNN}_v$, respectively, for $\overline{Q}=1$ with quadratic cost.}
        \label{fig:RNNMLP Qbar constant results}
\end{figure}

As Figure \ref{fig:RNNMLP Qbar constant training} shows, the a posteriori estimate \eqref{eq:EL discrete a posteriori} stabilizes around the order $10^{-1}$ more consistently than the previous architecture. However, no significant difference is observed in the price approximation. The error at terminal positions, $\bepsilon_T$, stabilizes after 100 epochs, while the running error, $\bepsilon$, has a consistent decay starting from 150 epochs, with a sudden peak around 250 epochs. As in the previous architecture, the error in the balance condition, $\epsilon_q$, has a dominating oscillatory behavior.

\begin{figure}[htp]
     \centering
     \begin{subfigure}[b]{0.48\textwidth}
     \centering
        \caption{Residuals (training)}
         \includegraphics[width=\textwidth]{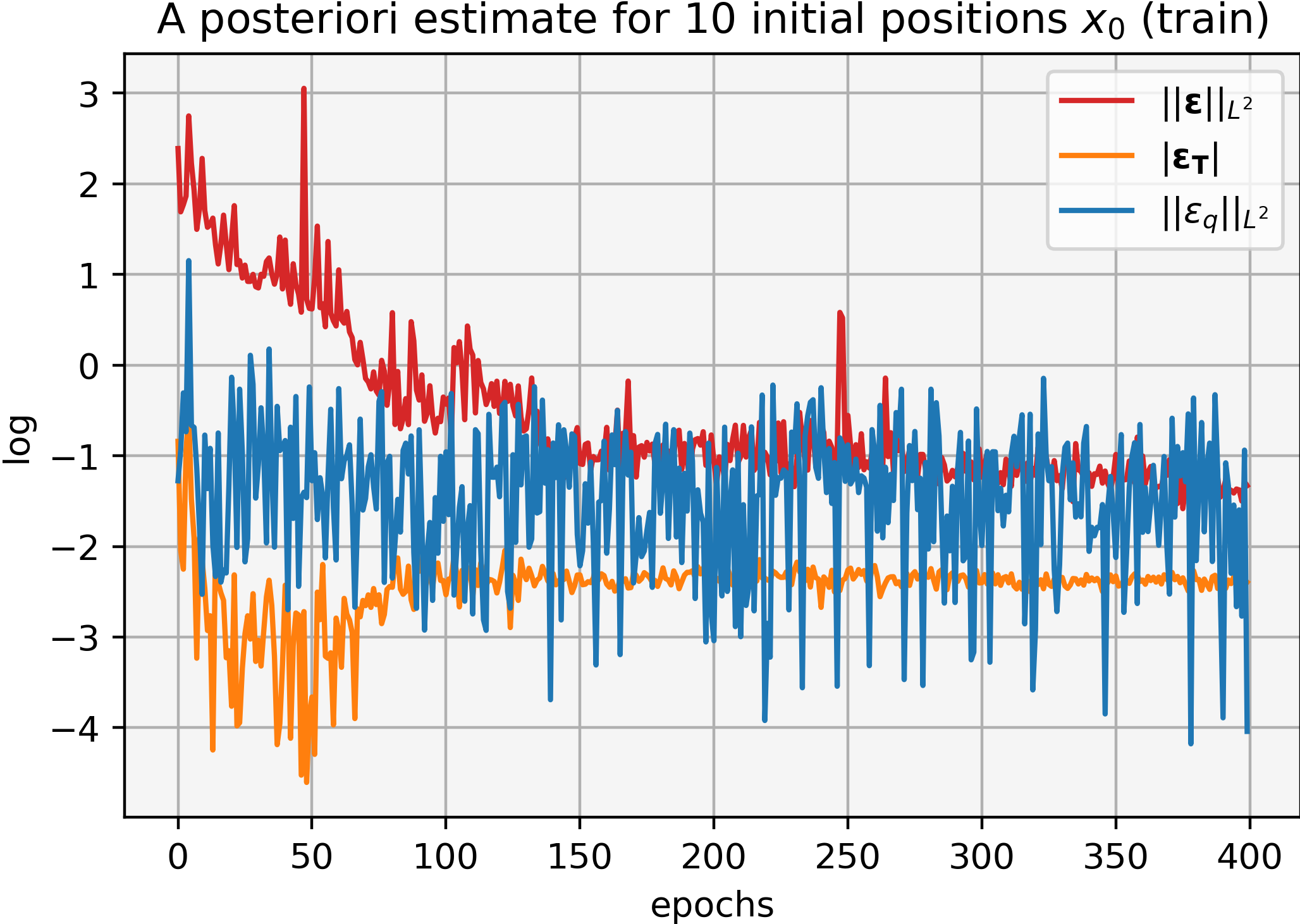}
          \end{subfigure}
     \hfill
     \begin{subfigure}[b]{0.48\textwidth}
     \centering
        \caption{Estimate vs. price error (training)}
         \includegraphics[width=\textwidth]{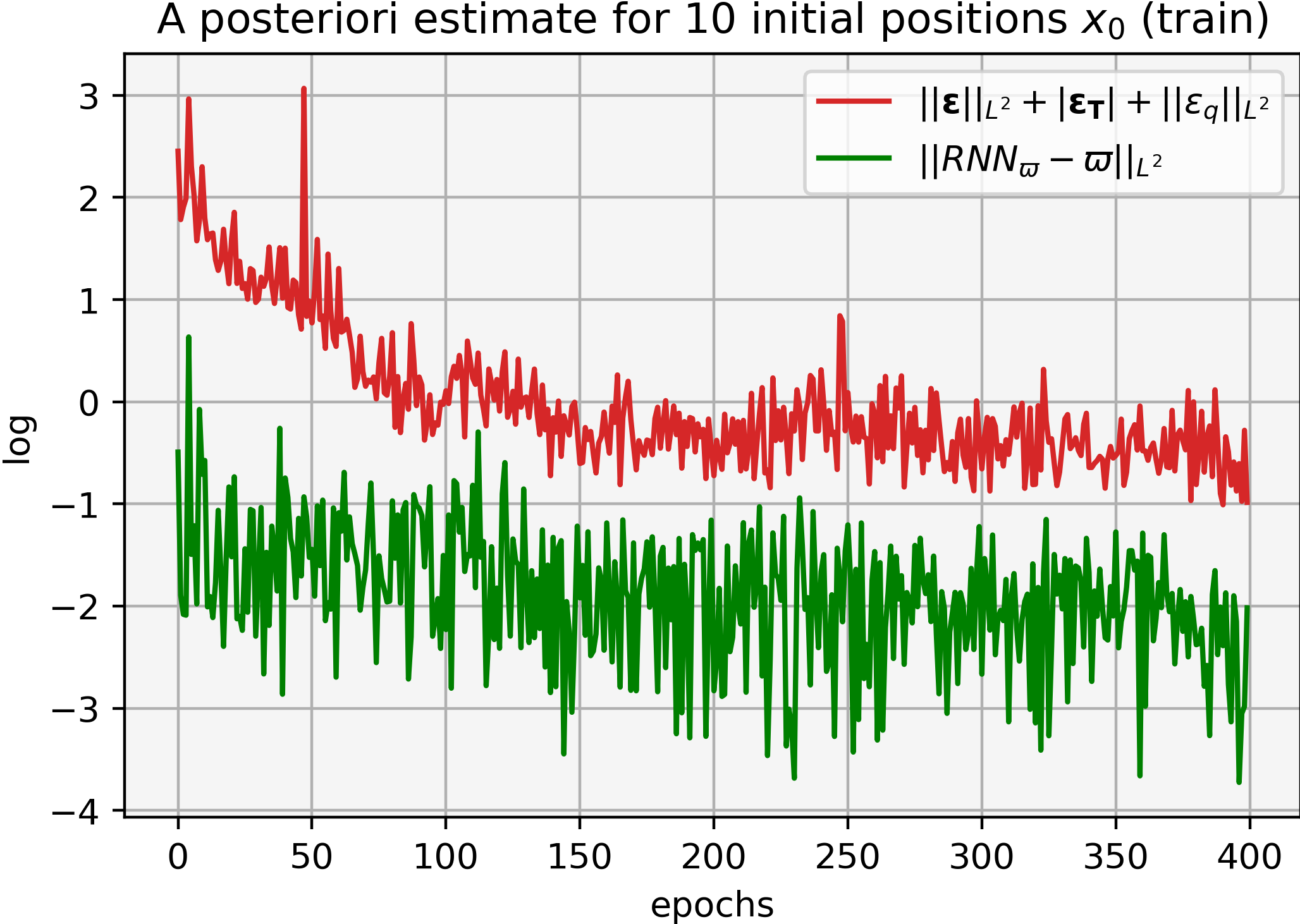}
          \end{subfigure}
         \caption{Residuals (left), a posteriori estimate \eqref{eq:EL discrete a posteriori} and price error (right) during training (log. scale) using $\mathrm{RNN}_\varpi$ and $\mathrm{RNN}_v$ for $\overline{Q}=1$ with quadratic cost.}
        \label{fig:RNNMLP Qbar constant training}
\end{figure}

{\bf Comparison.} Regarding the a posteriori estimate, we observe that the $RNN$ architecture achieves a consistent and earlier decay (150 epochs) compared to the $MLP$. However, no significant difference is observed in the price approximation.

\subsubsection{\bf Oscillating mean-reverting function}
The second case we consider is 
\[
	\overline{Q}(t)=7 e^{-t} \sin(3 \pi t), \quad q_0=0.
\]
Figure \ref{fig:Qbar osc analres} illustrates the analytical solutions for $\varpi$, $v^*$, and $m$. As before, we include the characteristic curves with initial position $x_0 \in [-1,1]$ in the plot for $m$. Notice the oscillating behavior inherited from $Q$. Proceeding as before,  we evaluate the error in the approximation of $v^*$  over the domain $[0,T]\times [-1,1]$ using the absolute error and the $m$-weighted error.
 
\begin{figure}[htp]
     \centering
     \begin{subfigure}[t]{0.32\textwidth}
         \centering
         \caption{Supply and price}
         \includegraphics[width=\textwidth]{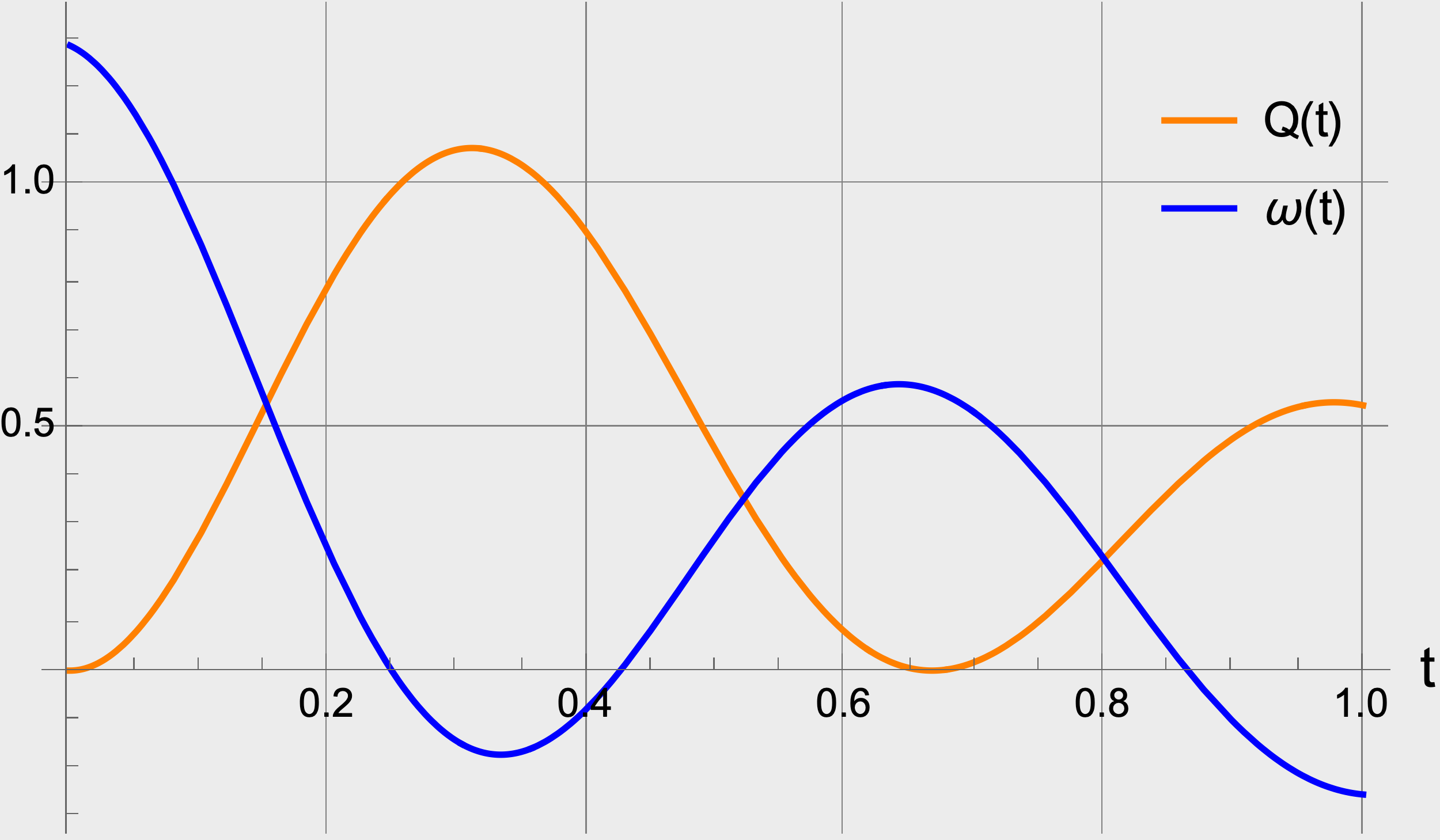}
     \end{subfigure}
     \hfill
     \begin{subfigure}[t]{0.32\textwidth}
         \centering
        \caption{Optimal feedback $v^*$}
         \includegraphics[width=\textwidth]{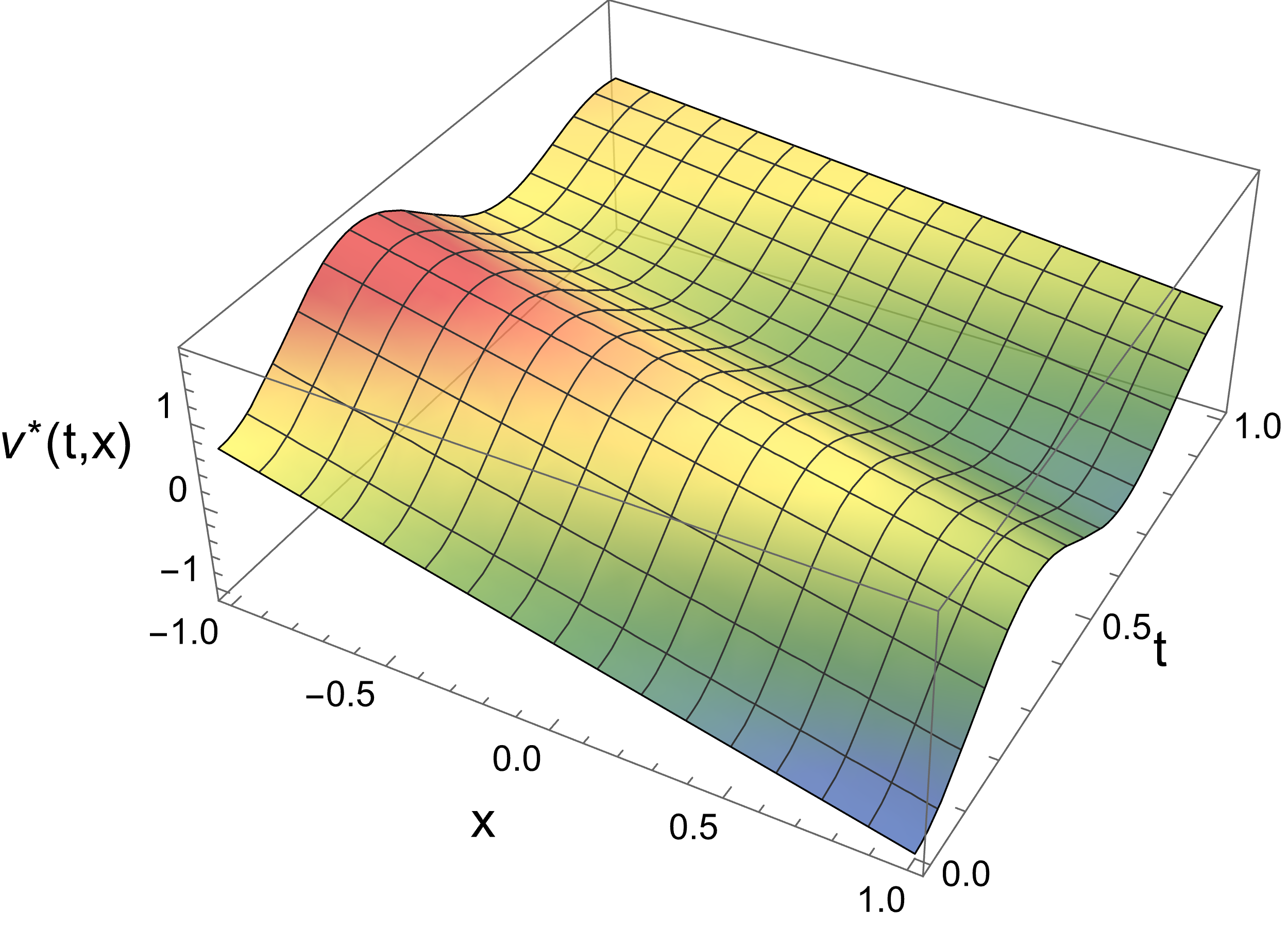}
     \end{subfigure}
     \hfill
     \begin{subfigure}[t]{0.32\textwidth}
         \centering
        \caption{$m$ and characteristics}
         \includegraphics[width=\textwidth]{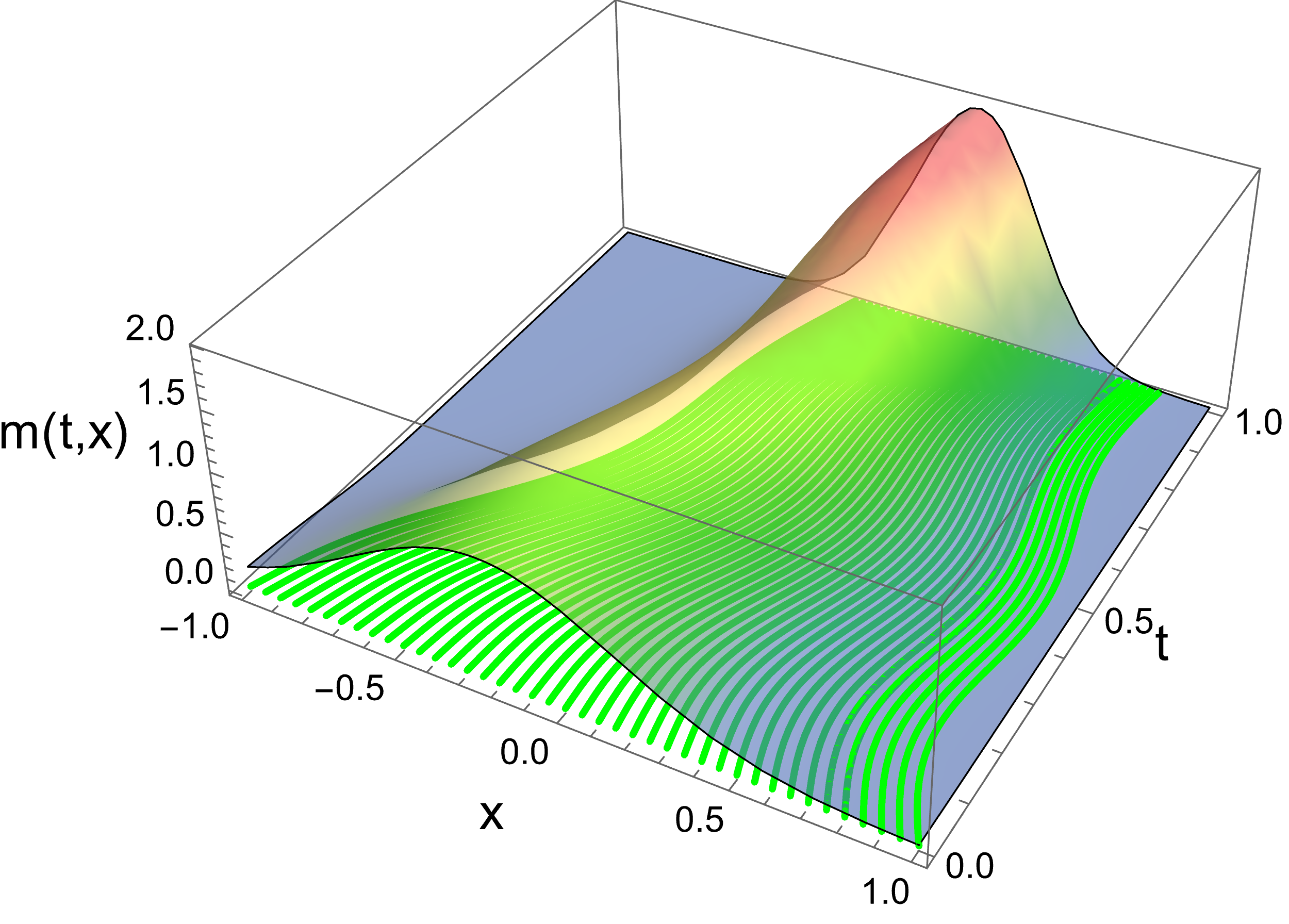}
     \end{subfigure}
        \caption{Analytical solutions for $\overline{Q}=7 e^{-t} \sin(3 \pi t)$ with quadratic cost.}
        \label{fig:Qbar osc analres}
\end{figure}

{\bf MLP with instantaneous feedback.} Figure \ref{fig:MLP Qbar osc results} shows the price approximation and error on the approximated values for $v^*$. Despite the oscillatory behavior of $Q$, the price and optimal control approximations reach an error of the same order ($10^{-2}$) as the constant-mean supply case.

\begin{figure}[htp]
     \centering
     \begin{subfigure}[b]{0.4\textwidth}
         \centering
        \caption{Analytic price vs. approx.}
         \includegraphics[width=\textwidth]{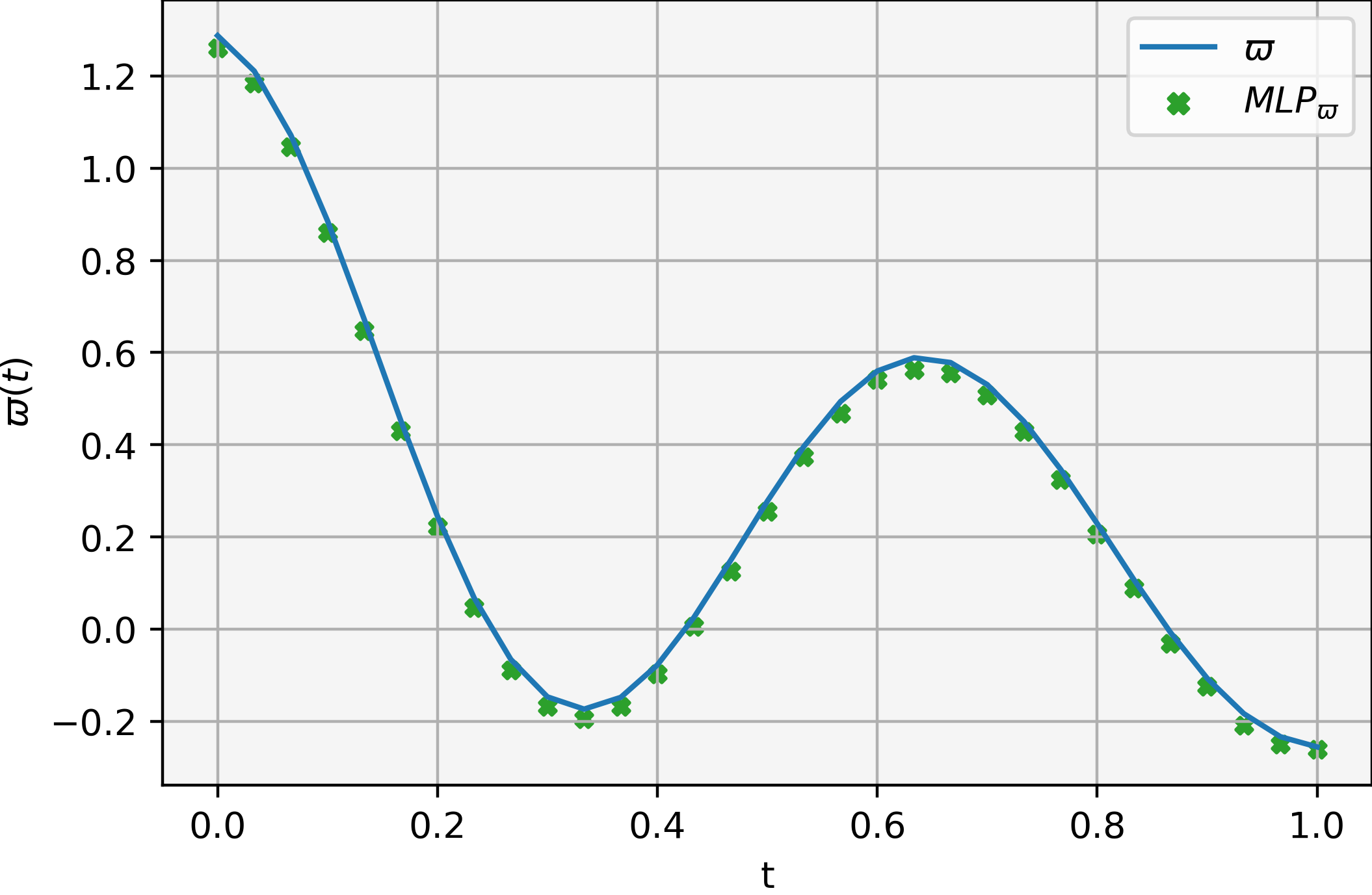}
     \end{subfigure}   
     \hskip0.3cm
     \begin{subfigure}[b]{0.4\textwidth}
         \centering
        \caption{Price approx. error}
         \includegraphics[width=\textwidth]{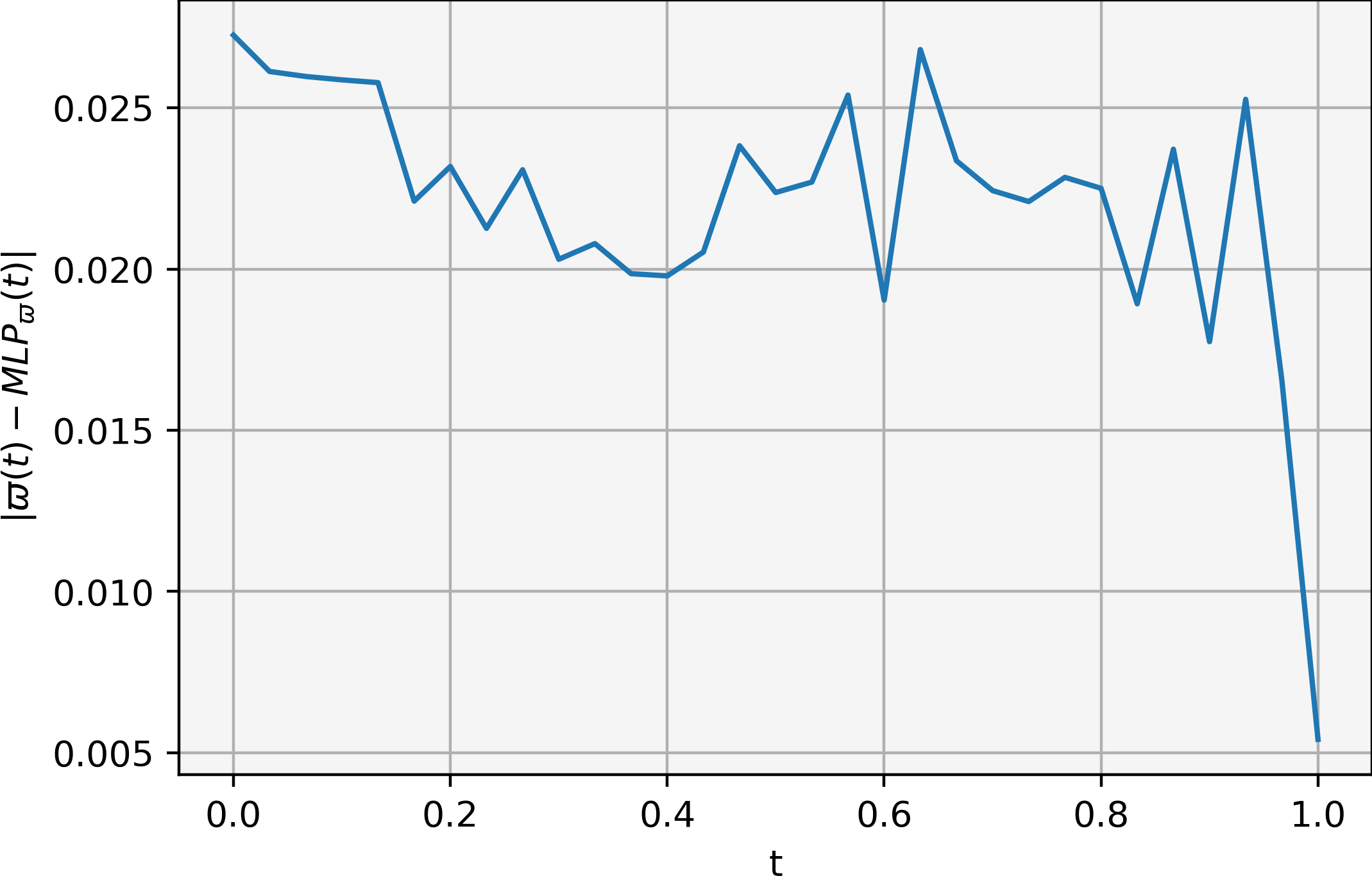}
     \end{subfigure}        
     
     \begin{subfigure}[b]{0.29\textwidth}
         \centering
         \caption{$v^*$ obtained by $\mathrm{MLP}_v$}
         \includegraphics[width=\textwidth]{Plots_New/DetPrice_2MLP_V2/OCNN.png}
     \end{subfigure}
     \hfill
     \begin{subfigure}[b]{0.29\textwidth}
         \centering
         \caption{Absolute error ($v^*$)}
         \includegraphics[width=\textwidth]{Plots_New/DetPrice_2MLP_V2/OCerror.png}
     \end{subfigure}
     \hfill
     \begin{subfigure}[b]{0.29\textwidth}
         \centering
         \caption{$m$-weighted error ($v^*$)}
         \includegraphics[width=\textwidth]{Plots_New/DetPrice_2MLP_V2/OCerrorWm.png}
     \end{subfigure}
        \caption{$\varpi$ (top) and $v^*$ (bottom) using $\mathrm{MLP}_\varpi$ and $\mathrm{MLP}_v$, respectively, for $\overline{Q}(t)=7 e^{-t} \sin(3 \pi t)$ with quadratic cost.}
        \label{fig:MLP Qbar osc results}
\end{figure}

Figure \ref{fig:MLP Qbar osc training} shows the a posteriori estimate \eqref{eq:EL discrete a posteriori}. Regardless of oscillatory behavior, we obtain a persistent decay as we iterate, stabilizing around an order of $10^{-1}$. The error in the price approximation decays while oscillating around an order of $10^{-2}$. The error $\bepsilon$ decays consistently. The terminal error stabilizes after 150 epochs around the order $10^{-2}$. The error in the balance condition oscillates the most.

\begin{figure}[htp]
     \centering
     \begin{subfigure}[b]{0.48\textwidth}
     \centering
        \caption{Residuals (training)}
         \includegraphics[width=\textwidth]{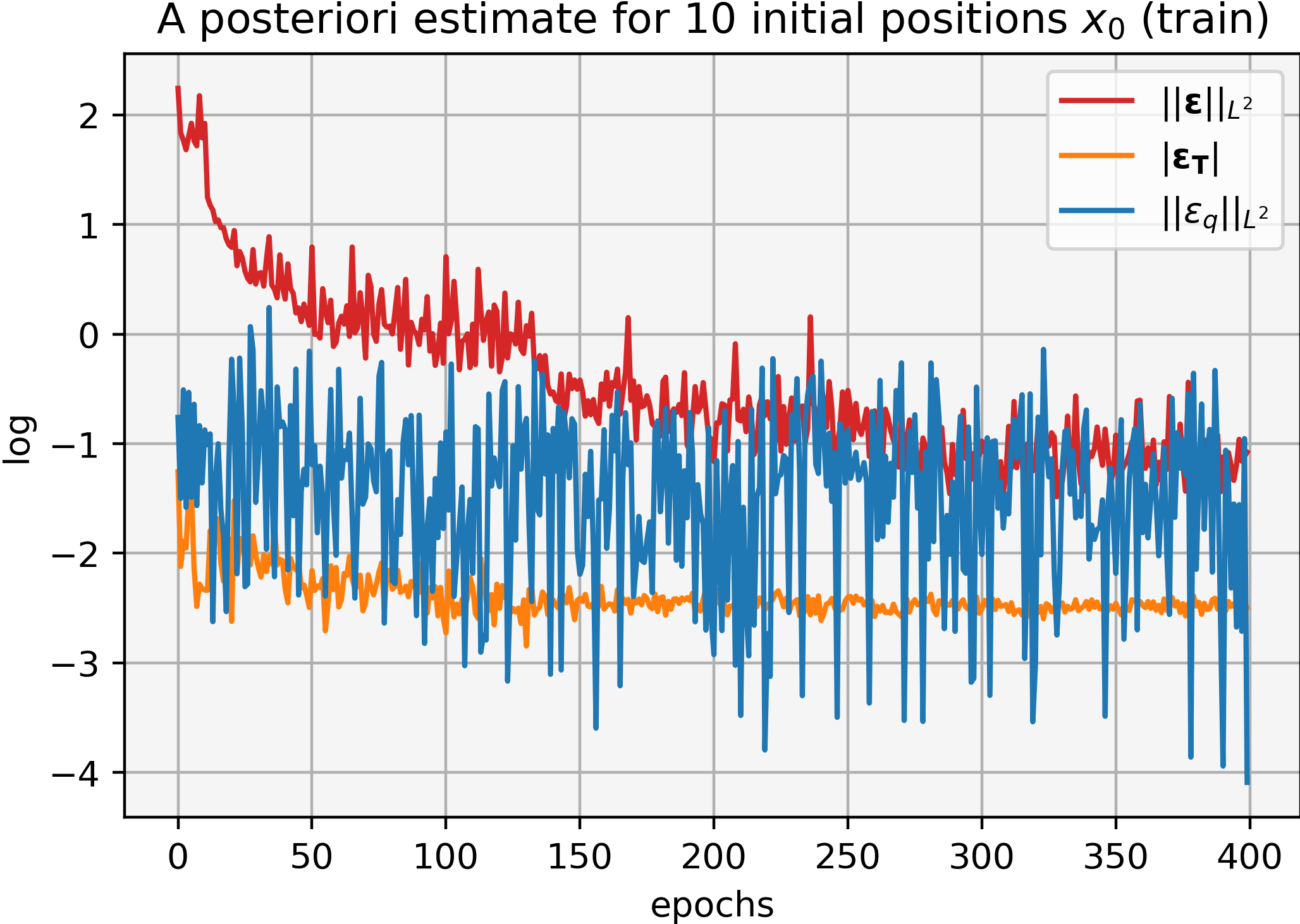}
          \end{subfigure}
     \hfill
     \begin{subfigure}[b]{0.48\textwidth}
     \centering
        \caption{Estimate vs. price error (training)}
         \includegraphics[width=\textwidth]{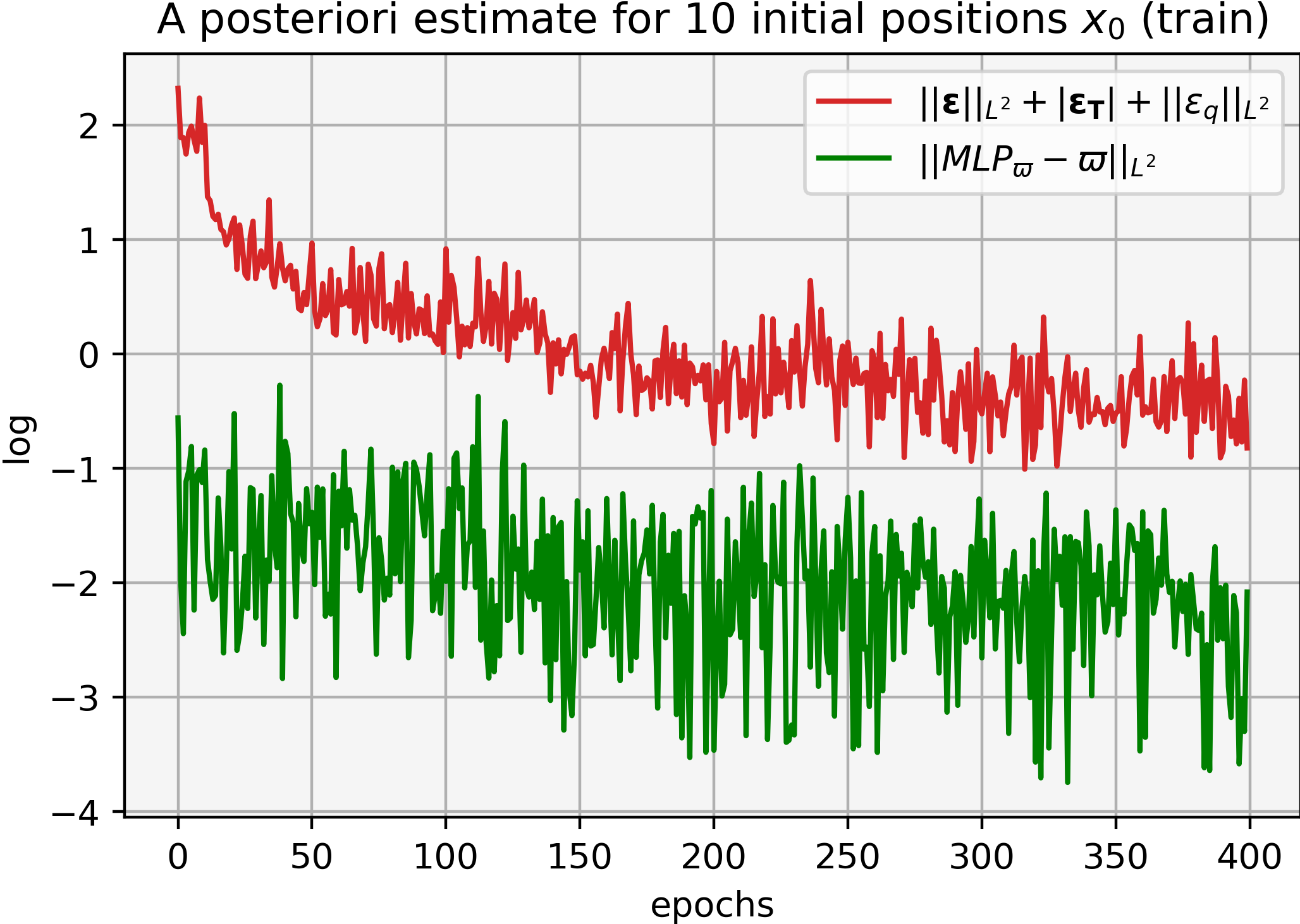}
          \end{subfigure}
         \caption{Residuals (left), a posteriori estimate \eqref{eq:EL discrete a posteriori} and price error (right) during training (log. scale) using $\mathrm{MLP}_\varpi$ and $\mathrm{MLP}_v$ for $\overline{Q}(t)=7 e^{-t} \sin(3 \pi t)$ with quadratic cost.}
        \label{fig:MLP Qbar osc training}
\end{figure}

{\bf RNN with supply history.} The approximated price and error values obtained by $\mathrm{RNN}_v$ and $\mathrm{RNN}_\varpi$ are shown in Figure \ref{fig:RNNMLP Qbar osc results}. Although the error's magnitude does not increase compared to the previous case, we observe areas where this error rises. The errors approximating the price and optimal control are of an order $10^{-2}$. 

\begin{figure}[htp]
     \centering
     \begin{subfigure}[b]{0.4\textwidth}
         \centering
        \caption{Analytic price vs. approx.}
         \includegraphics[width=\textwidth]{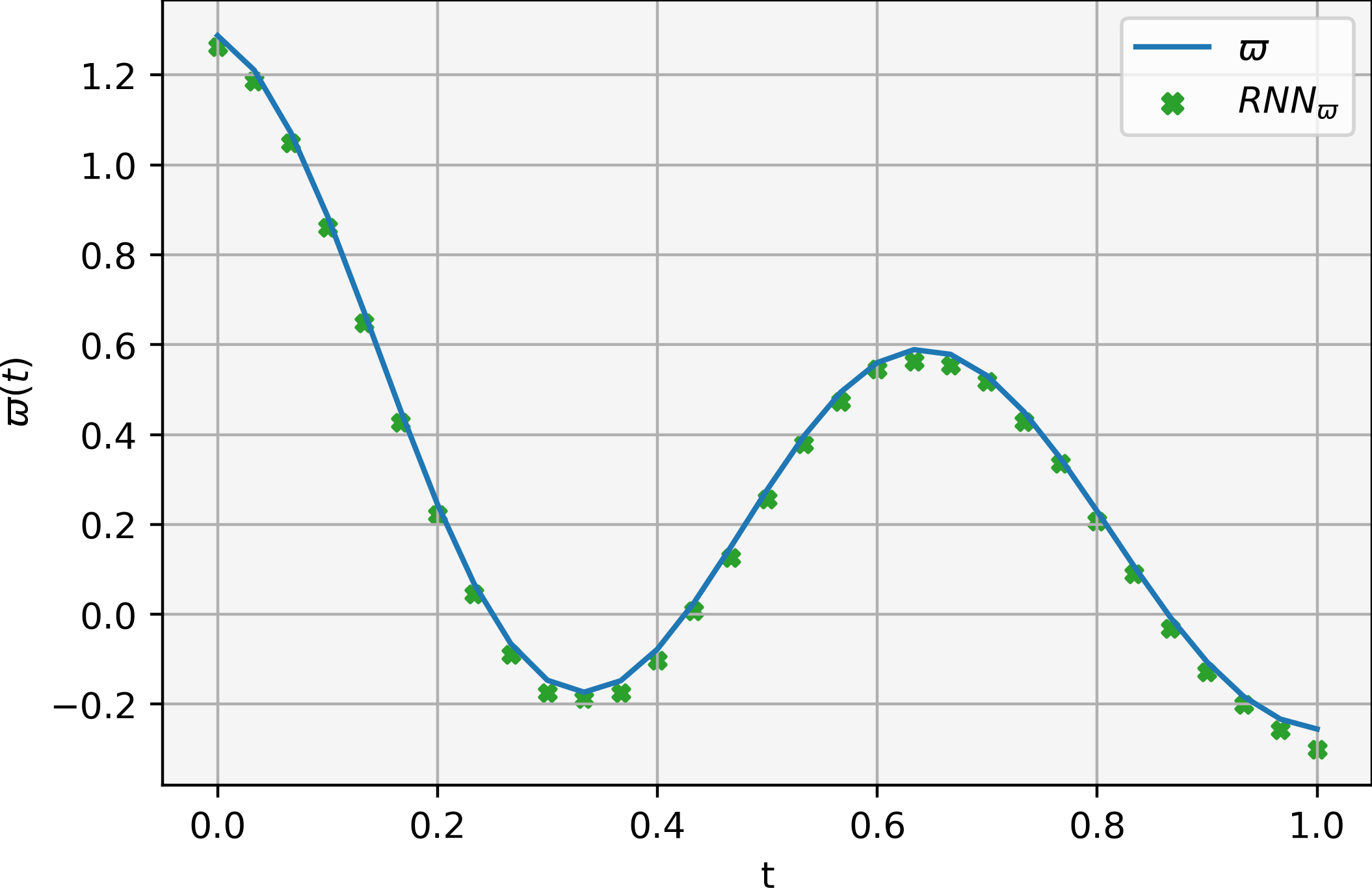}
     \end{subfigure}   
     \hskip0.3cm  
     \begin{subfigure}[b]{0.4\textwidth}
         \centering
        \caption{Price approx. error}
         \includegraphics[width=\textwidth]{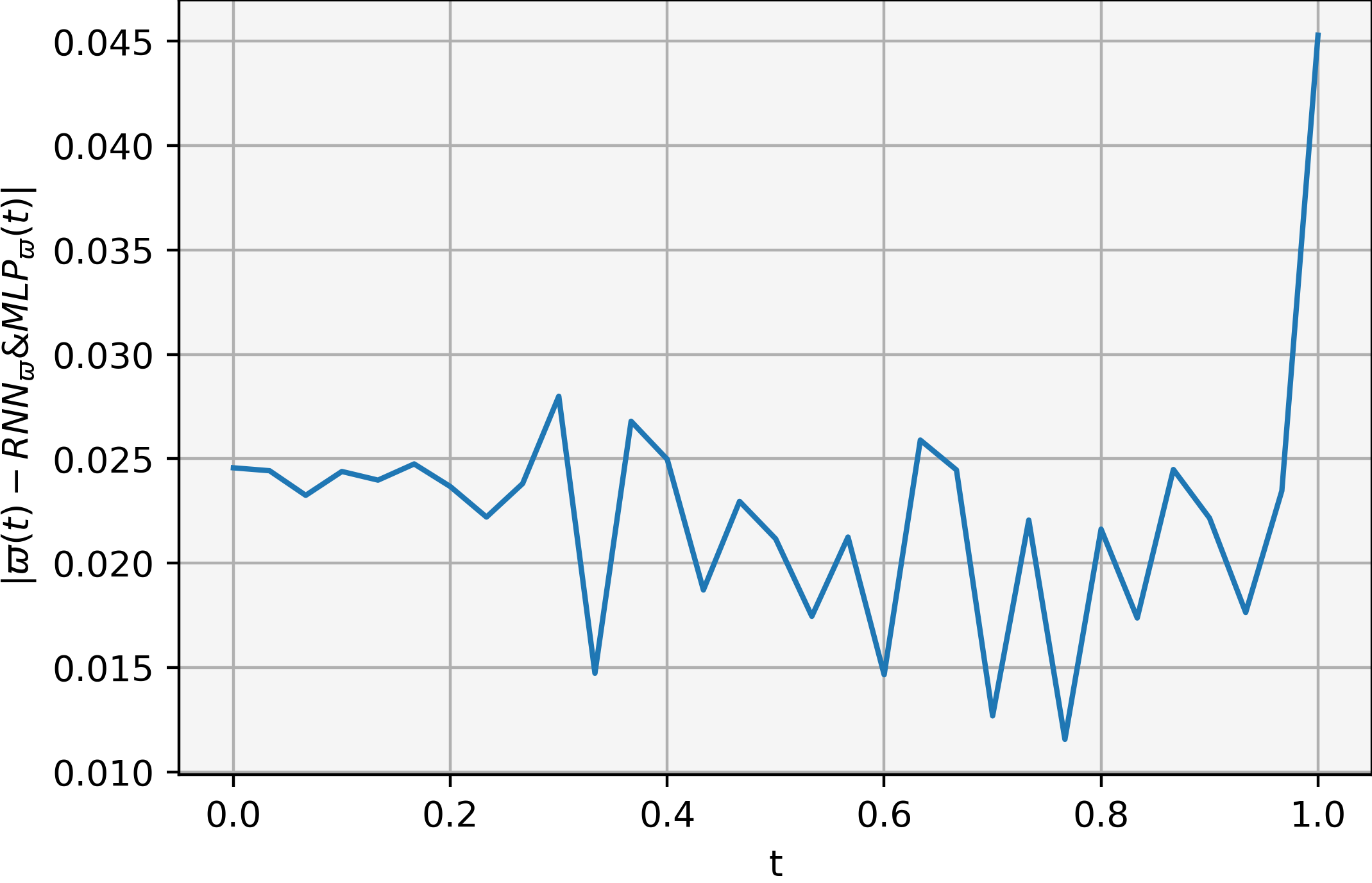}
     \end{subfigure}        
     
     \begin{subfigure}[b]{0.29\textwidth}
         \centering
         \caption{$v^*$ obtained by $\mathrm{RNN}_v$}
         \includegraphics[width=\textwidth]{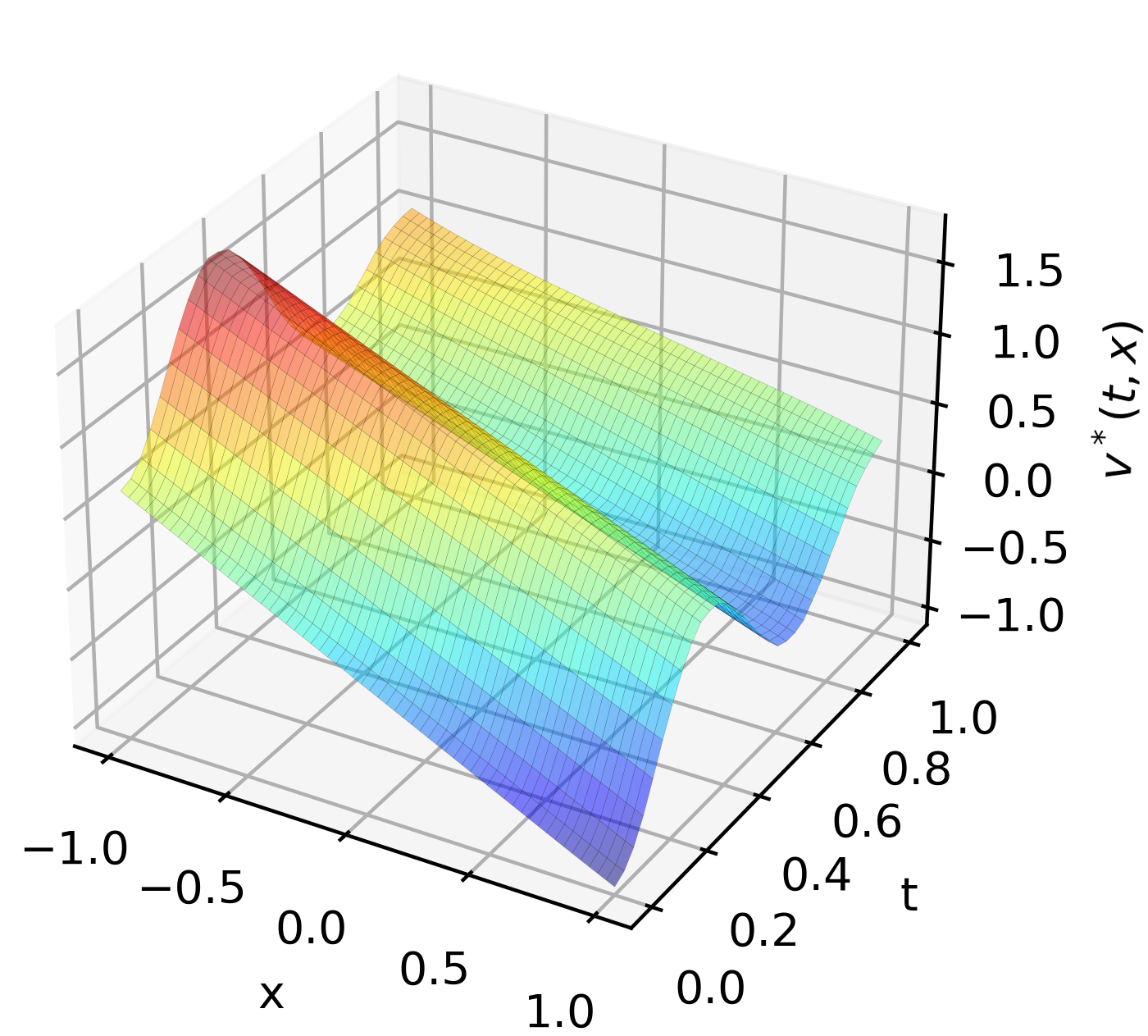}
     \end{subfigure}
     \hfill
     \begin{subfigure}[b]{0.29\textwidth}
         \centering
         \caption{Absolute error ($v^*$)}
         \includegraphics[width=\textwidth]{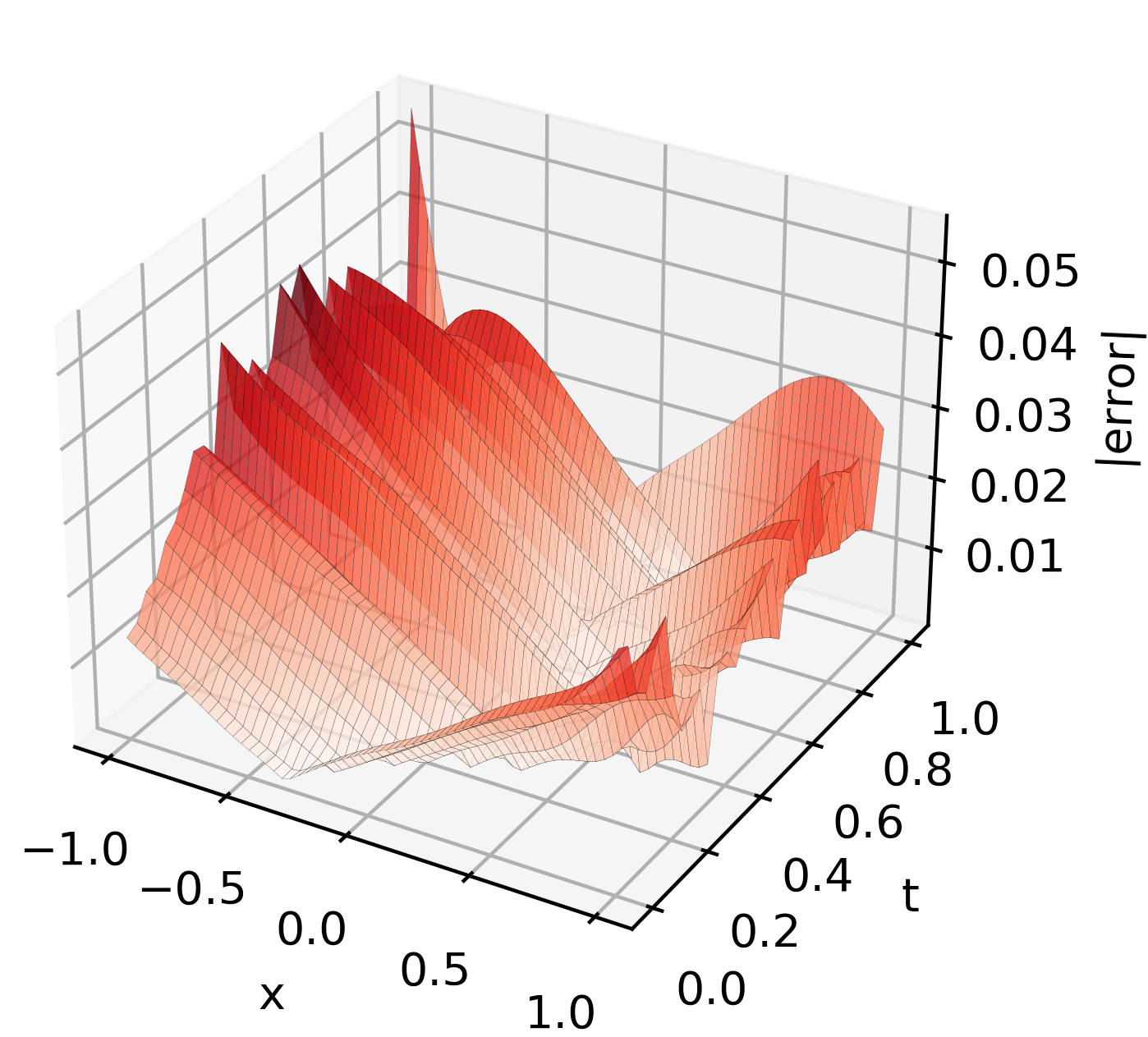}
     \end{subfigure}
     \hfill
     \begin{subfigure}[b]{0.29\textwidth}
         \centering
         \caption{$m$-weighted error ($v^*$)}
         \includegraphics[width=\textwidth]{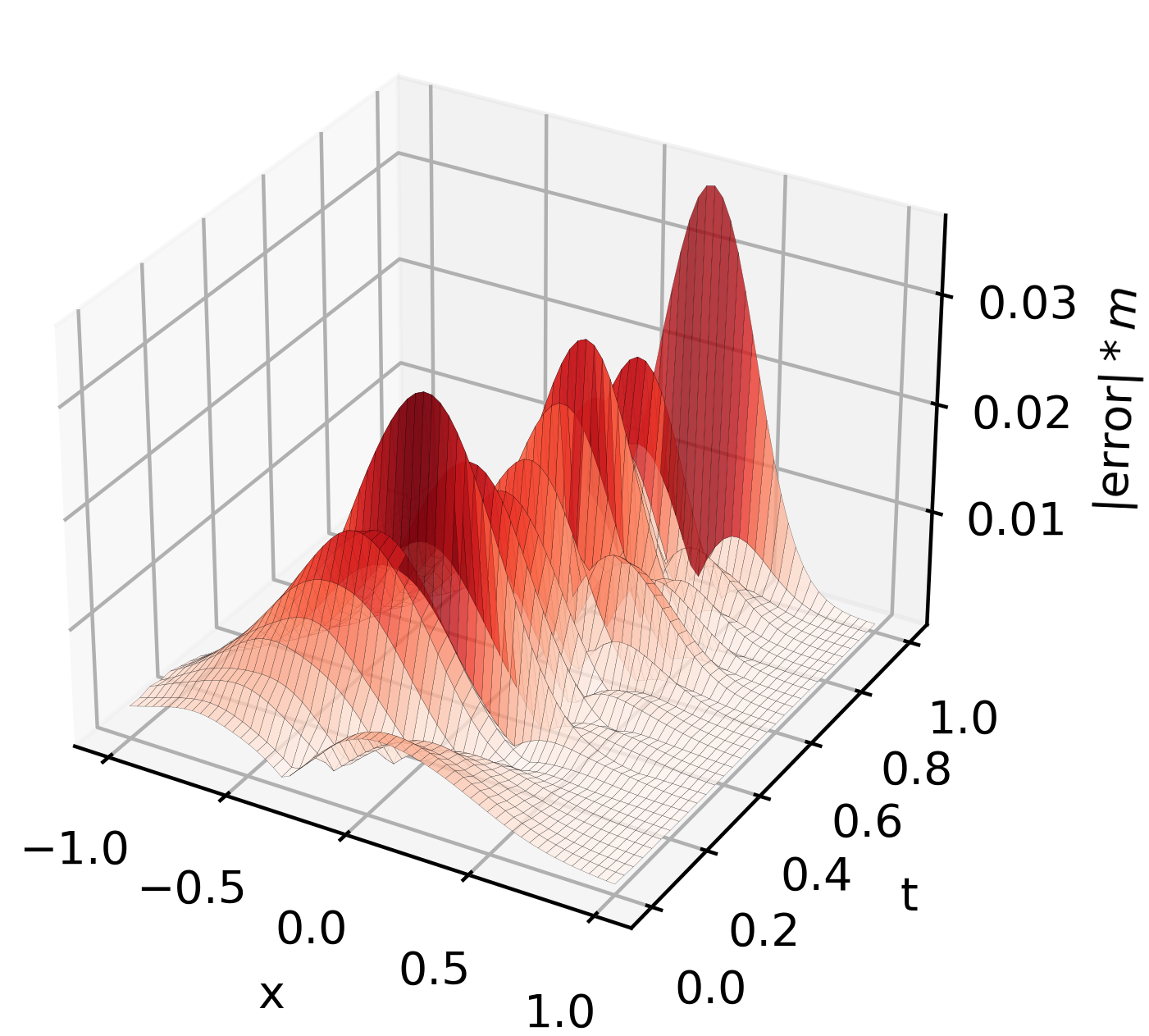}
     \end{subfigure}
        \caption{$\varpi$ (top) and $v^*$ (bottom) using $\mathrm{RNN}_\varpi$ and $\mathrm{RNN}_v$, respectively, for $\overline{Q}(t)=7 e^{-t} \sin(3 \pi t)$ with quadratic cost.}
        \label{fig:RNNMLP Qbar osc results}
\end{figure}

Figure \ref{fig:RNNMLP Qbar osc training} depicts the a posteriori estimate \eqref{eq:EL discrete a posteriori}. We observe the stabilization of the estimate around $300$ epochs, which suggests that more parameters and training steps would improve the approximating capabilities of this architecture. The error in the price approximation decays more consistently than the previous architecture. All error components in the a posteriori estimate behave similarly to the previous architecture, except for the oscillation of $\epsilon_q$ at the beginning of training.

\begin{figure}[htp]
     \centering
     \begin{subfigure}[b]{0.48\textwidth}
     \centering
        \caption{Residuals (training)}
         \includegraphics[width=\textwidth]{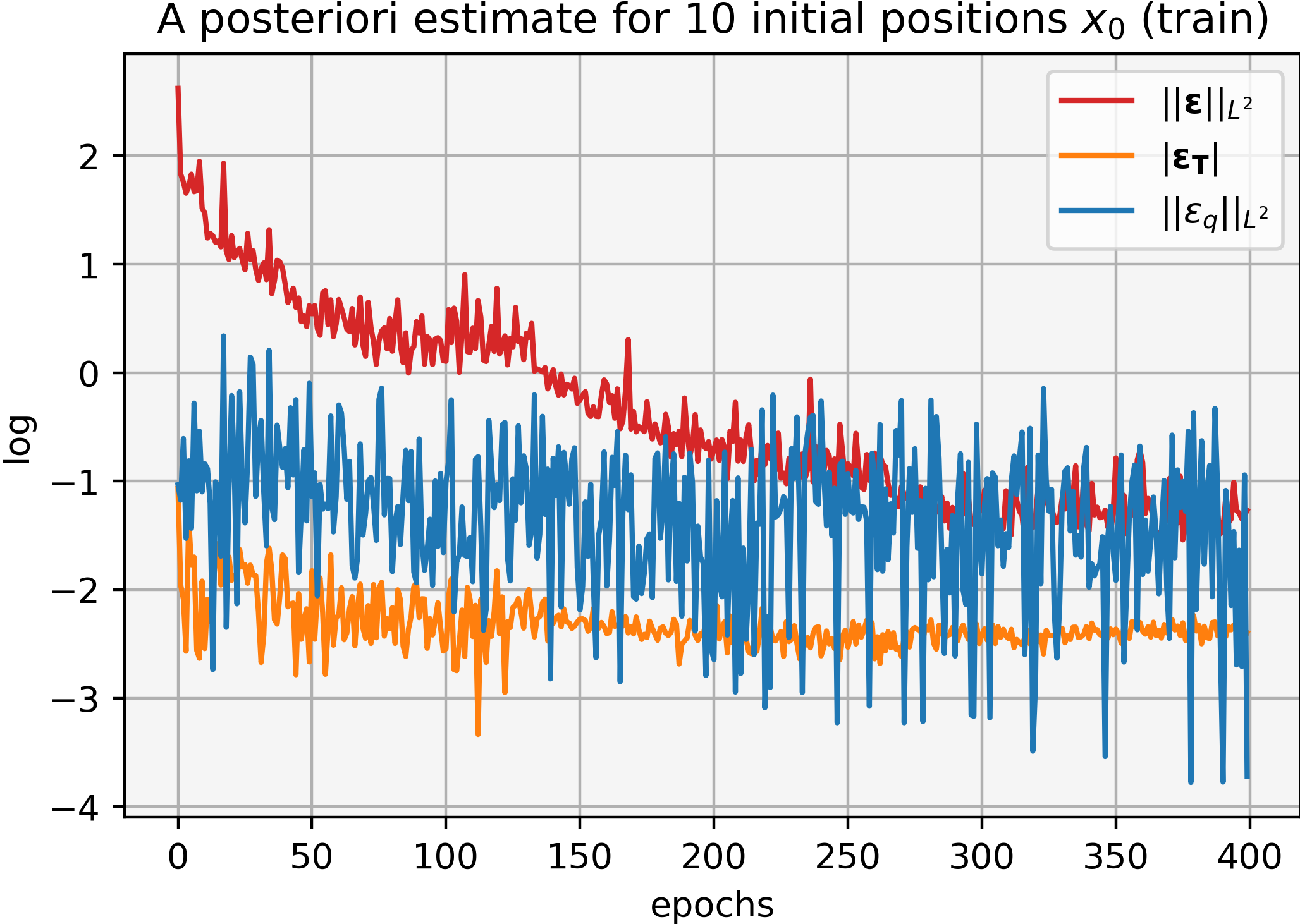}
          \end{subfigure}
     \hfill
     \begin{subfigure}[b]{0.48\textwidth}
     \centering
        \caption{Estimate vs. price error (training)}
         \includegraphics[width=\textwidth]{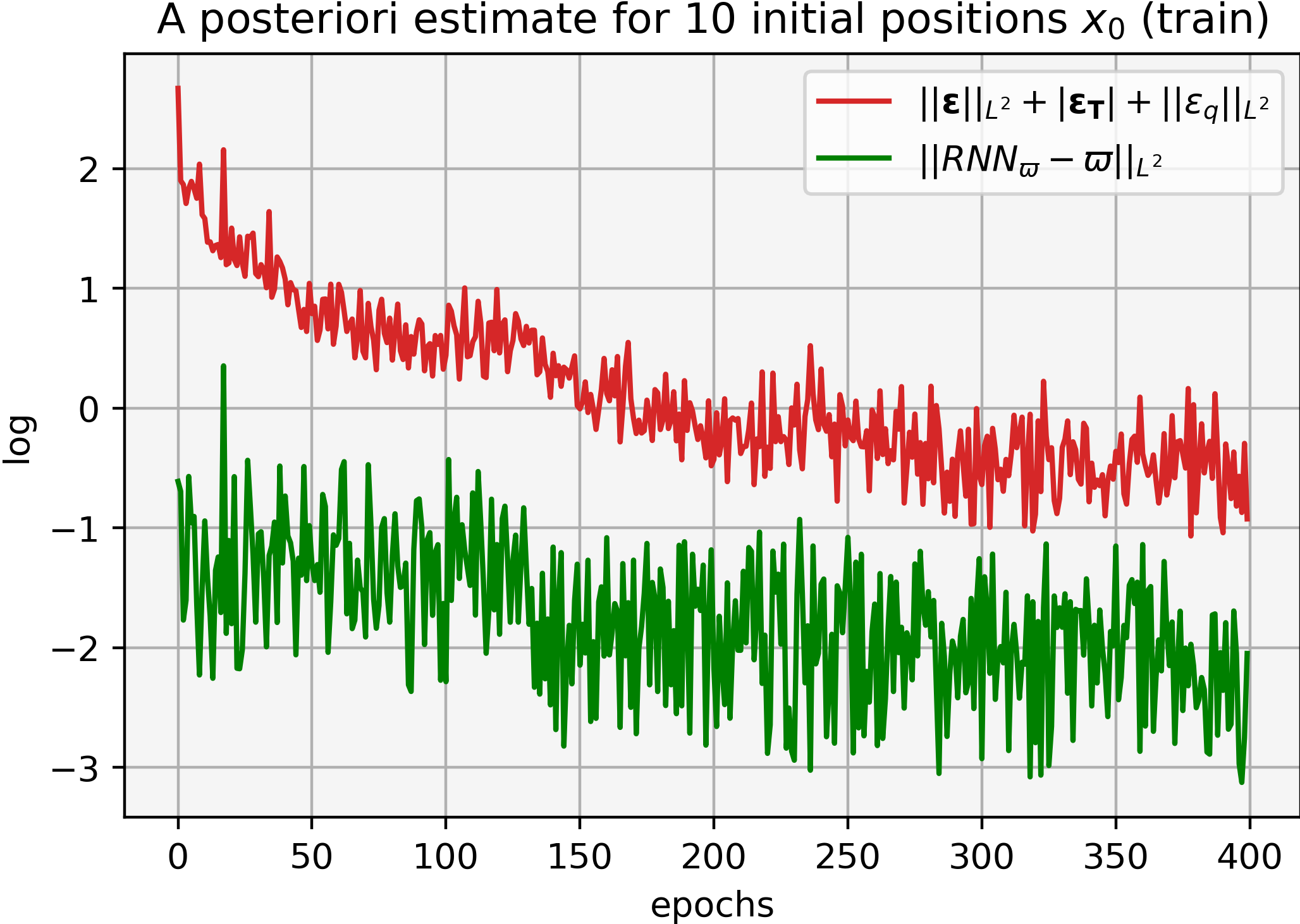}
          \end{subfigure}
         \caption{Residuals (left), a posteriori estimate \eqref{eq:EL discrete a posteriori} and price error (right) during training (log. scale) using $\mathrm{RNN}_\varpi$ and $\mathrm{RNN}_v$ for $\overline{Q}(t)=7 e^{-t} \sin(3 \pi t)$ with quadratic cost.}
        \label{fig:RNNMLP Qbar osc training}
\end{figure}

{\bf Comparison.} We observe that the $RNN$ architecture has a more consistent decay in the a posteriori estimate compared to the $MLP$. The $RNN$ has a peak at terminal time in the price approximation error, but otherwise, no significant difference in the approximating performance is observed. 

\subsection{Non-quadratic cost}
By adding a power-like term to the quadratic cost, we consider
\[
	L(x,v)=\frac{\eta}{2} \left( x - \kappa\right)^2 + \frac{c}{4} v^2 + \frac{c}{8}(v-1)^4, \quad \mbox{and} \quad u_T\left(x\right) = \frac{\gamma}{2}\left(x-\zeta\right)^2,
\]
where $\kappa,\zeta \in \Rr$, $\eta,\gamma\geq 0$, and $c>0$. The choice of $L$ gives agents the least cost w.r.t. $v$ when $v \approx 0.317672$. Notice that $L$ satisfies all assumptions in Section \ref{sec:Assumptions}. In this setting, we have no explicit solutions for the MFG system \eqref{eq:MFG system}. However, the binomial-tree approach used in \cite{gomes2021randomsupply} allows us to solve the $N$-player price formation problem for a relatively large value of $N$ and use the corresponding price and optimal trajectories as a benchmark. Because the evolution of our supply is deterministic, the binomial tree collapses to a single branch, and we obtain a high-dimensional convex optimization problem that is solvable with standard optimization tools. The price is obtained on the interval $[0,T-h]$. We use $N=500$ and consider only the case $\overline{Q}(t) = 7 e^{-t} \sin(3 \pi t)$ and $q_0=0$. Figure \ref{fig:Qw NonQ AnalSol} illustrates the price solving the $N$-player problem.

\begin{figure}[htp]
     \centering
      \includegraphics[width=0.4\textwidth]{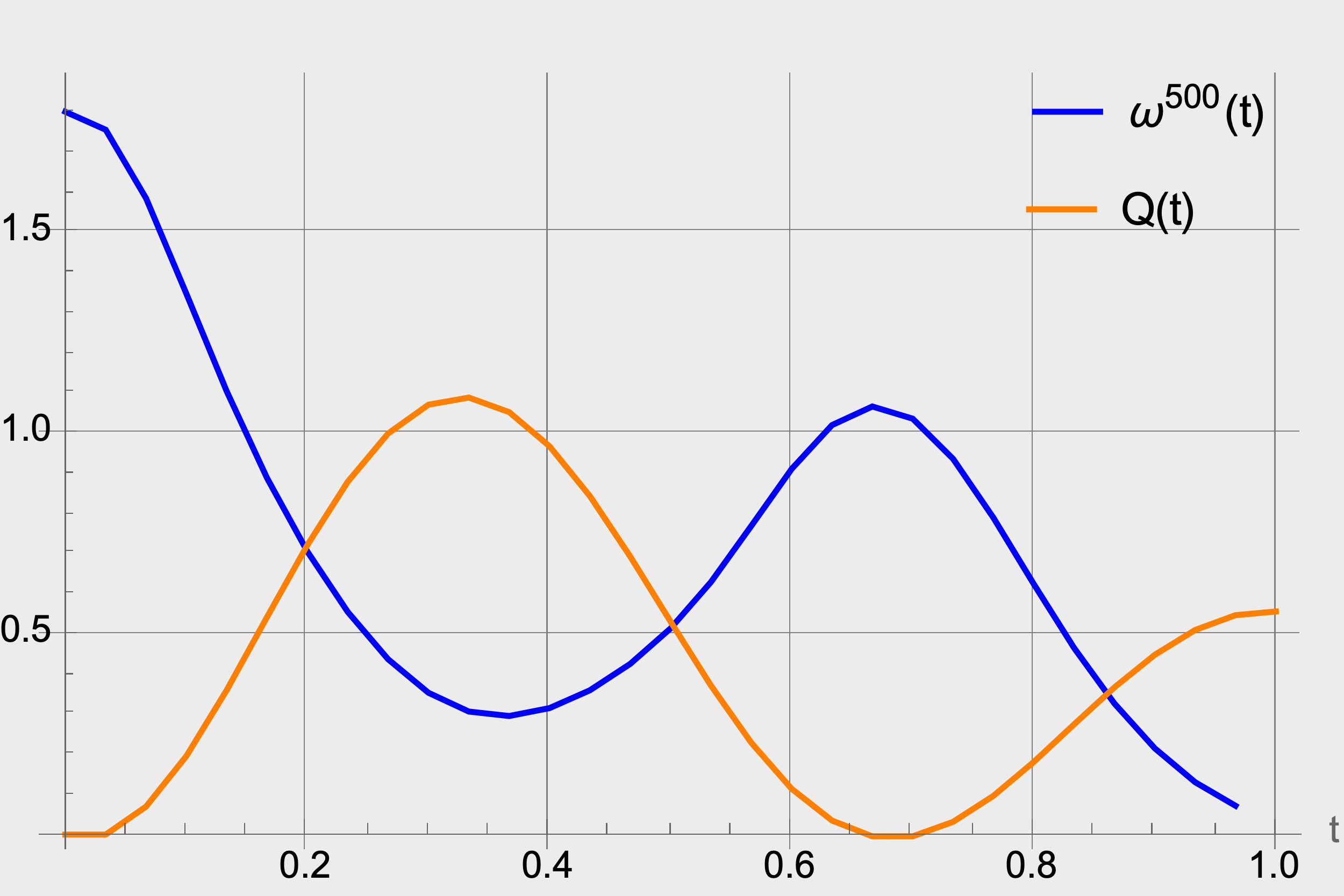}
        \caption{Supply and price ($N=500$-player game) for $\overline{Q}=7 e^{-t} \sin(3 \pi t)$ with non-quadratic cost.}
        \label{fig:Qw NonQ AnalSol}
\end{figure}

{\bf MLP with instantaneous feedback.} Figure \ref{fig:MLP NonQ Qbar osc results} shows the price and optimal control approximations, including the error on the approximated optimal trajectories for a sample of $10$ initial positions out of the $500$ used to compute the price. The error in the price approximation is of an order $10^{-1}$, while that of the optimal control is of an order $10^{-2}$ for the sample of trajectories.

\begin{figure}[htp]
     \centering
     \begin{subfigure}[b]{0.4\textwidth}
         \centering
        \caption{Numerical price vs. approx.}
         \includegraphics[width=\textwidth]{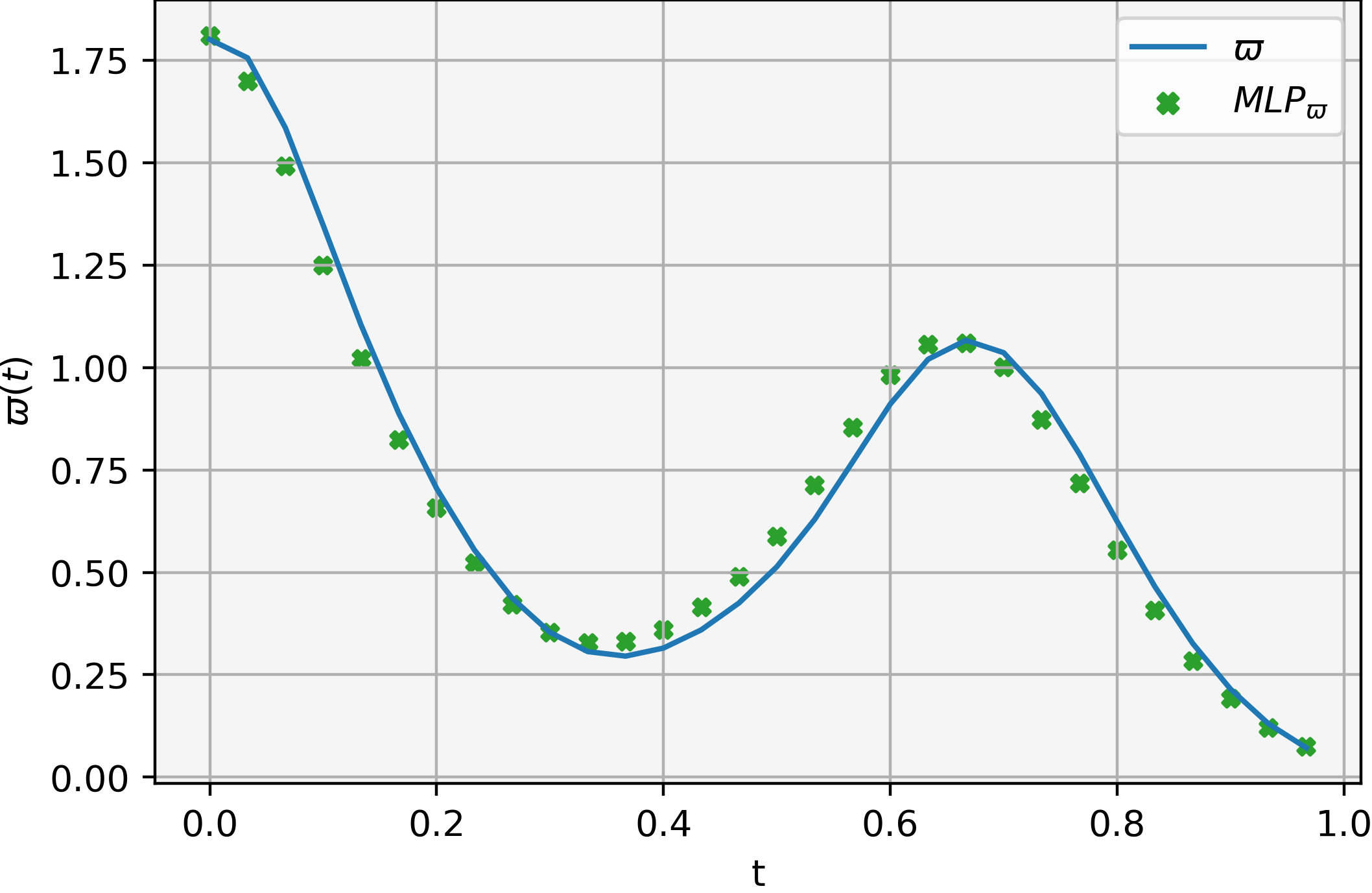}
     \end{subfigure}   
     \hskip0.3cm  
          \begin{subfigure}[b]{0.4\textwidth}
         \centering
        \caption{Numerical price vs. approx.}
         \includegraphics[width=\textwidth]{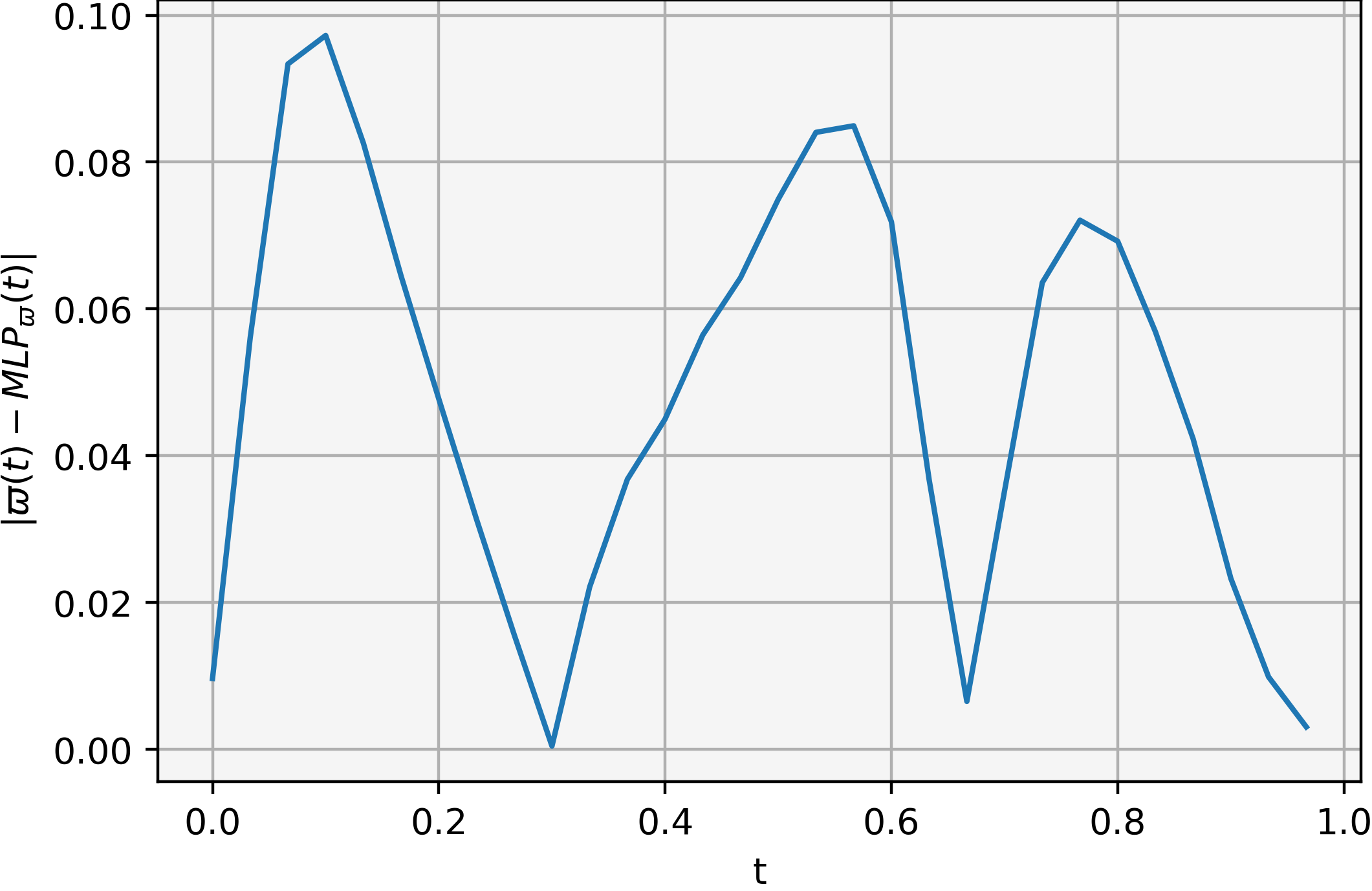}
     \end{subfigure}        
     
       \begin{subfigure}[b]{0.27\textwidth}
         \centering
         \caption{$v^*$ obtained by $\mathrm{MLP}_v$}
         \includegraphics[width=\textwidth]{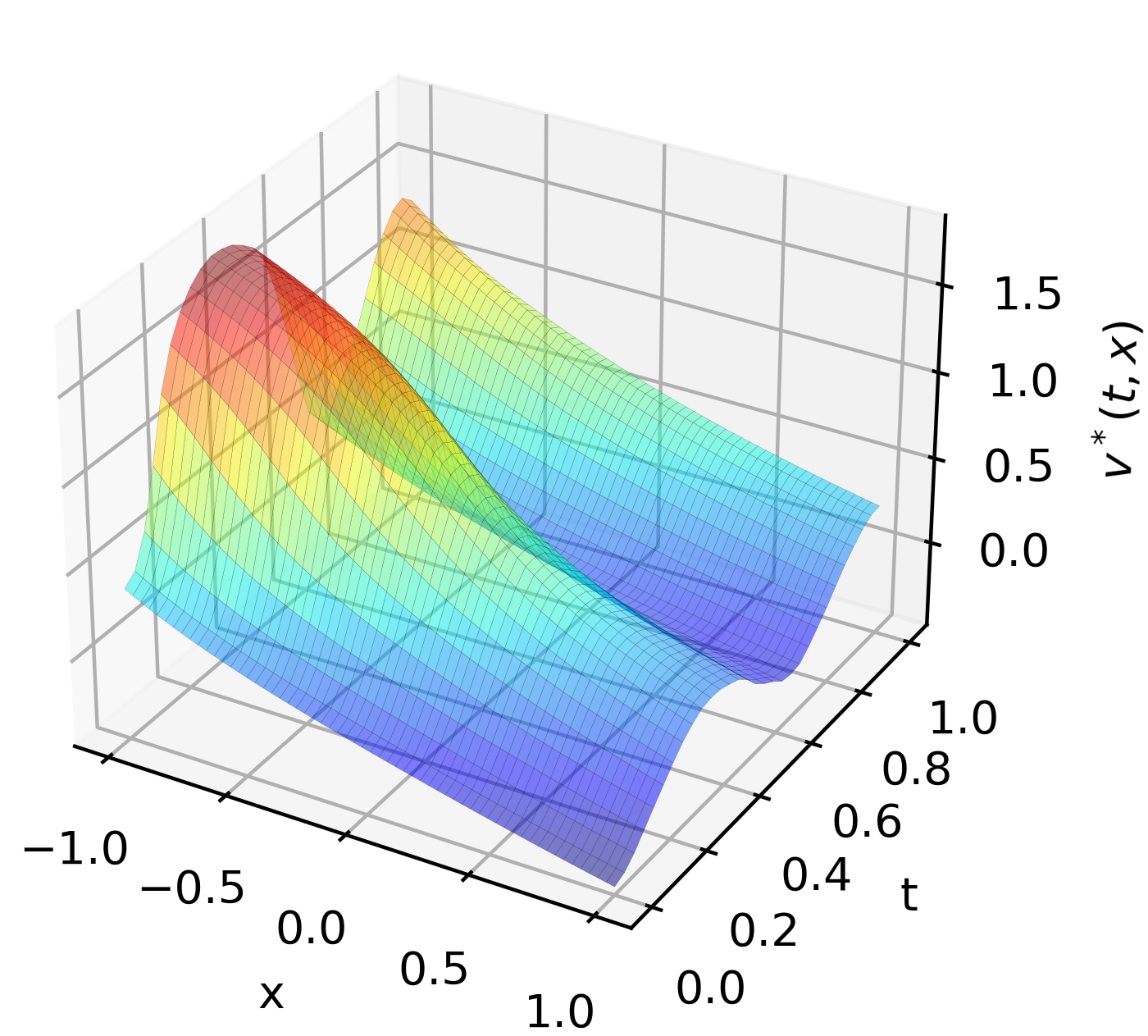}
     \end{subfigure}
     \hfill
     \begin{subfigure}[b]{0.34\textwidth}
         \centering
         \caption{Trajectories approx.}
         \includegraphics[width=\textwidth]{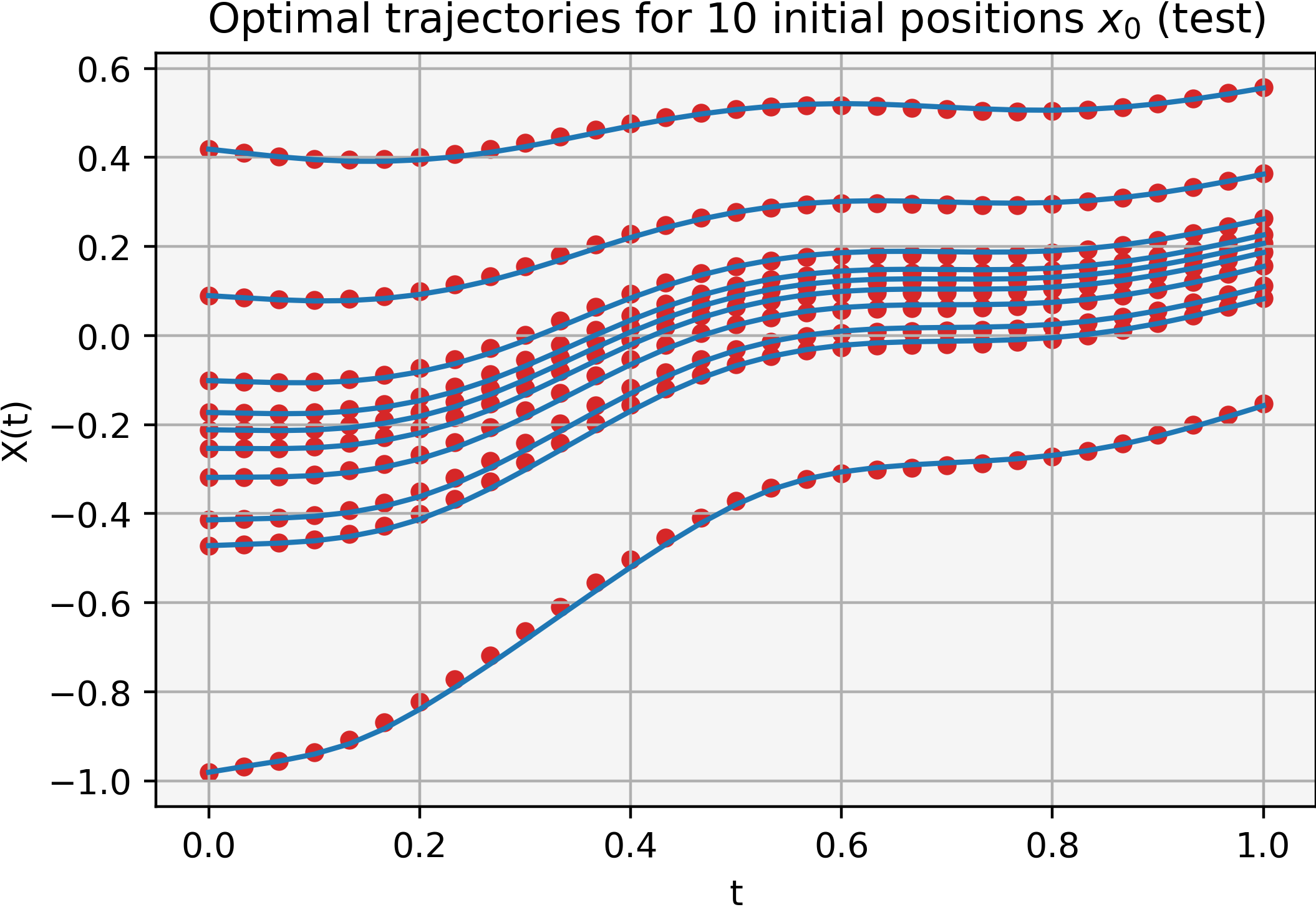}
     \end{subfigure}
     \hfill
     \begin{subfigure}[b]{0.34\textwidth}
         \centering
         \caption{Trajectories (error)}
         \includegraphics[width=\textwidth]{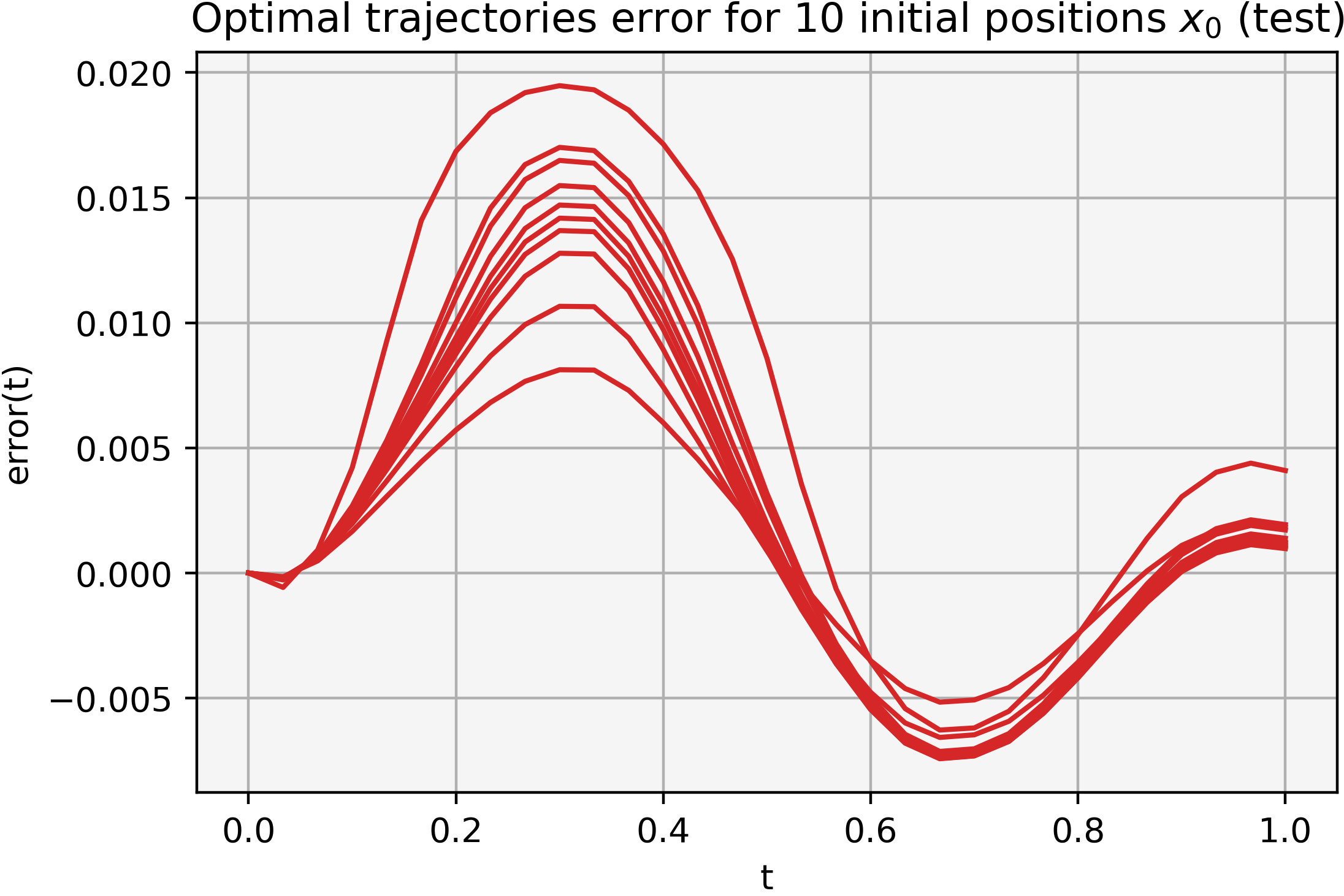}
     \end{subfigure}
        \caption{$\varpi$ (top) and $v^*$ (bottom) using $\mathrm{MLP}_\varpi$ and $\mathrm{MLP}_v$, respectively, for $\overline{Q}(t)=7 e^{-t} \sin(3 \pi t)$ with non-quadratic cost.}
        \label{fig:MLP NonQ Qbar osc results}
\end{figure}

Figure \ref{fig:MLP Qbar osc NonQ training} shows the a posteriori estimate decays to oscillate around an order $10^{-1}$, while the error in the price approximation stabilizes around the same order. All components of error in the estimate decay consistently, except for $\epsilon_q$, which oscillates drastically.  


\begin{figure}[htp]
     \centering
     \begin{subfigure}[b]{0.48\textwidth}
     \centering
        \caption{Residuals (training)}
         \includegraphics[width=\textwidth]{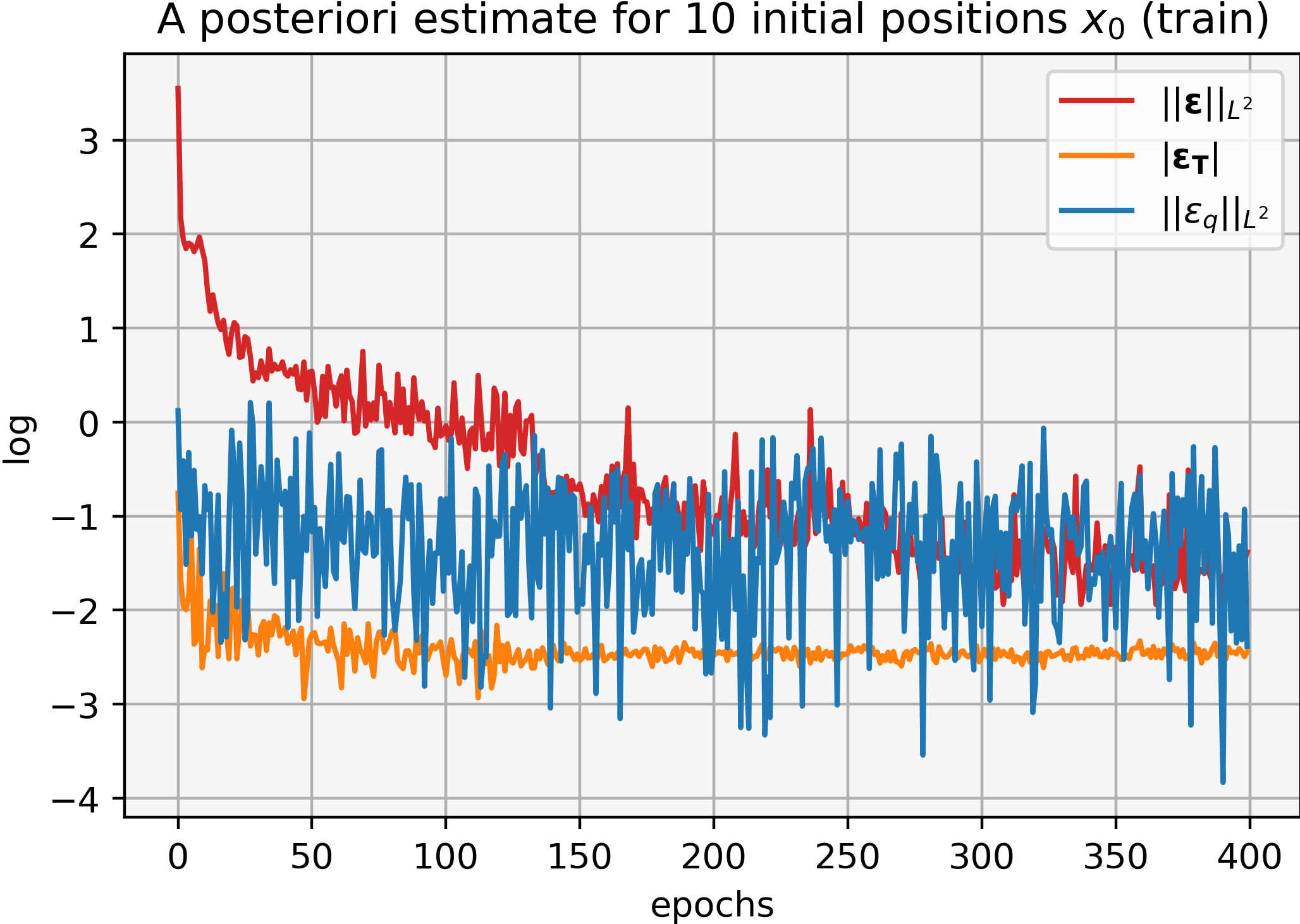}
          \end{subfigure}
     \hfill
     \begin{subfigure}[b]{0.48\textwidth}
     \centering
        \caption{Estimate vs. price error (training)}
         \includegraphics[width=\textwidth]{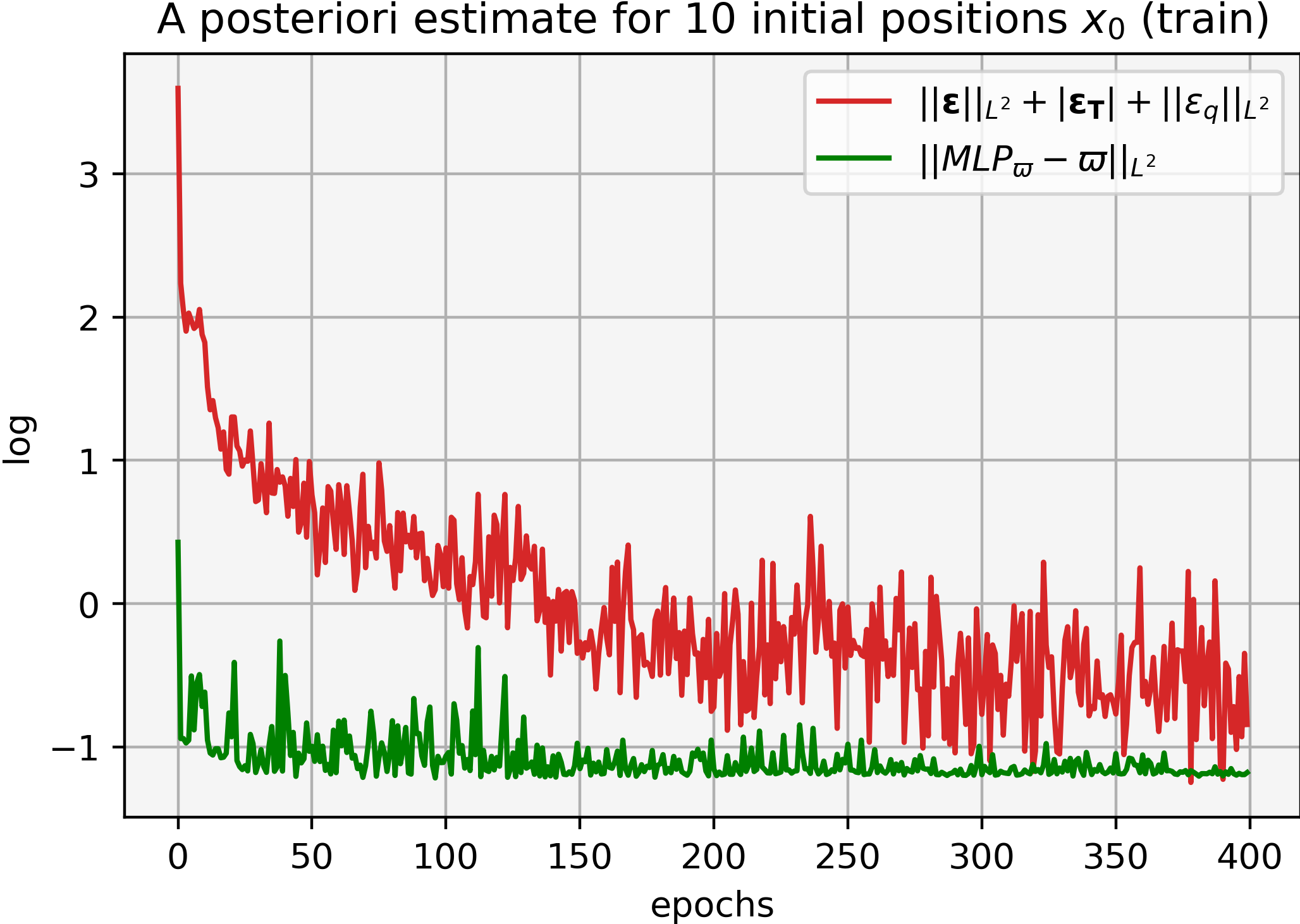}
          \end{subfigure}
         \caption{Residuals (left), a posteriori estimate \eqref{eq:EL discrete a posteriori} and price error (right) during training (log. scale) using $\mathrm{MLP}_\varpi$ and $\mathrm{MLP}_v$ for $\overline{Q}(t)=7 e^{-t} \sin(3 \pi t)$ with non-quadratic cost.}
        \label{fig:MLP Qbar osc NonQ training}
\end{figure}

{\bf RNN with supply history.} Figure \ref{fig:RNN NonQ Qbar osc results} shows the price and optimal control approximations, including the error on the approximated optimal trajectories (for the same sample as with the previous architecture). The price and optimal control errors exhibit the same order as the previous $NN$ configuration, $10^{-1}$, and $10^{-2}$, respectively. Moreover, the error in both architectures oscillates in the same time intervals, which may suggest that the approximation properties are strongly affected by the non-quadratic cost structure.

\begin{figure}[htp]
     \centering
     \begin{subfigure}[b]{0.4\textwidth}
         \centering
        \caption{Numerical price vs. approx.}
         \includegraphics[width=\textwidth]{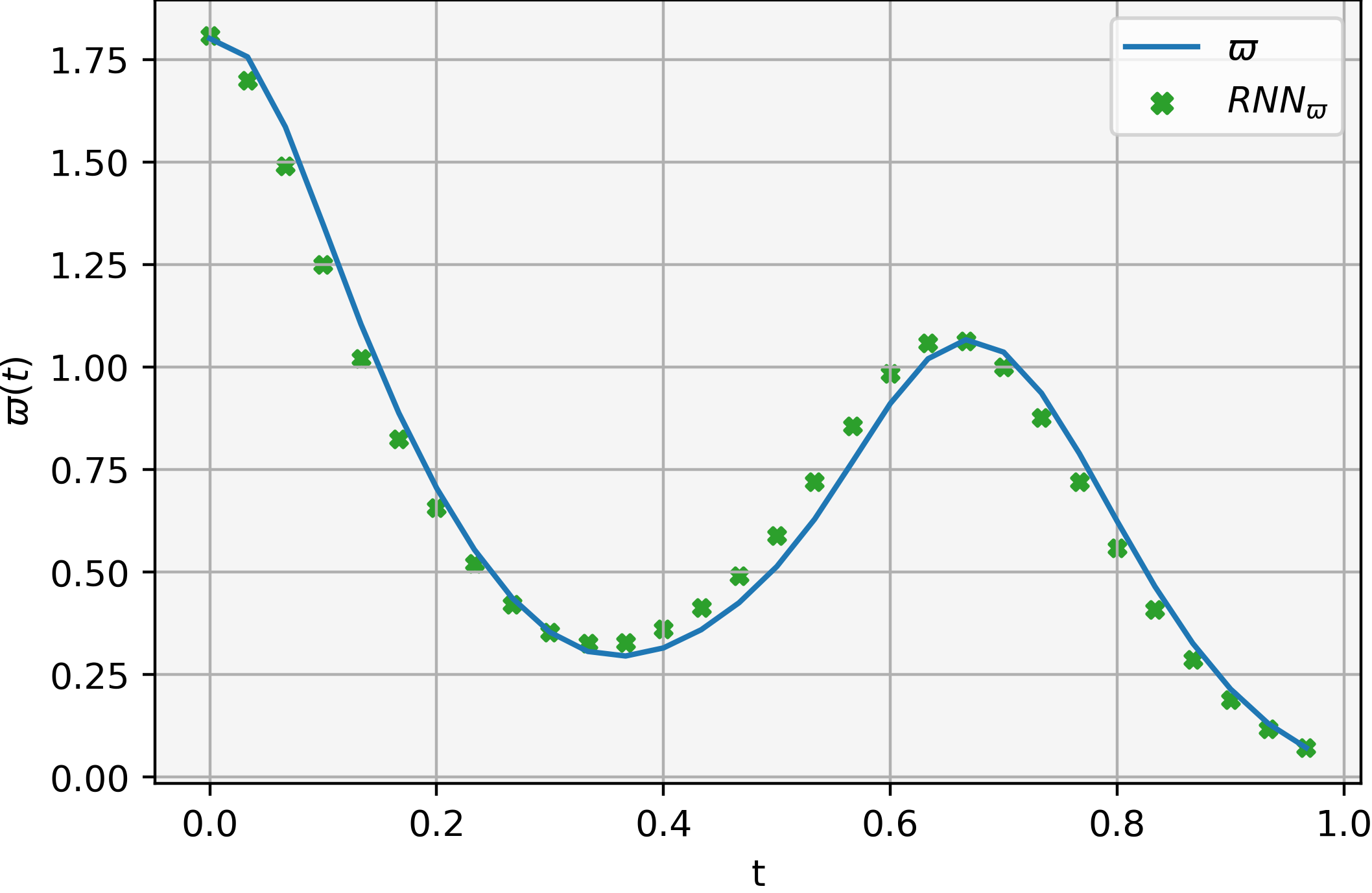}
     \end{subfigure}   
     \hskip0.3cm  
     \begin{subfigure}[b]{0.4\textwidth}
         \centering
        \caption{Price approx. error}
         \includegraphics[width=\textwidth]{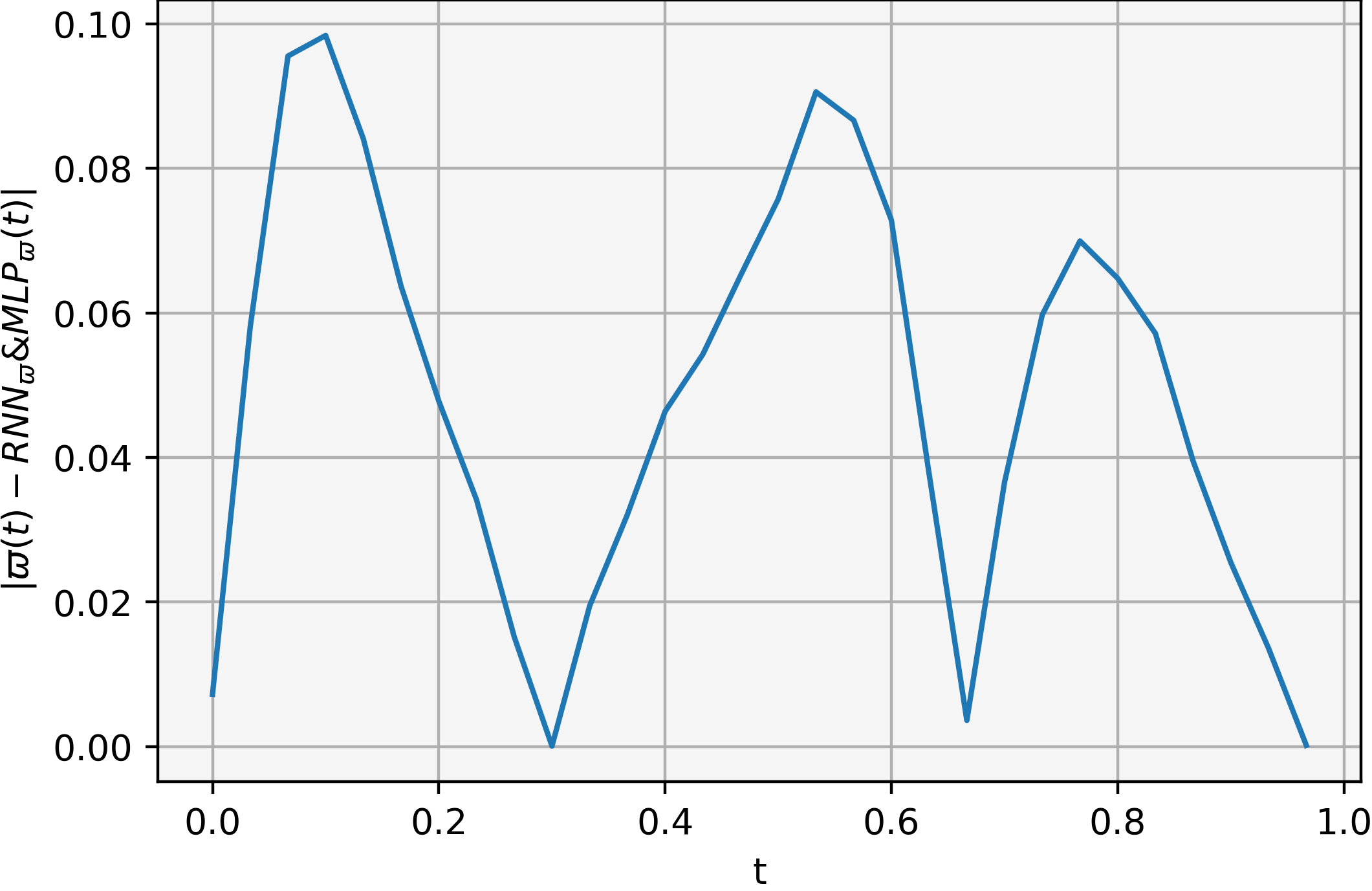}
     \end{subfigure}       
     
       \begin{subfigure}[b]{0.27\textwidth}
         \centering
         \caption{$v^*$ obtained by $\mathrm{MLP}_v$}
         \includegraphics[width=\textwidth]{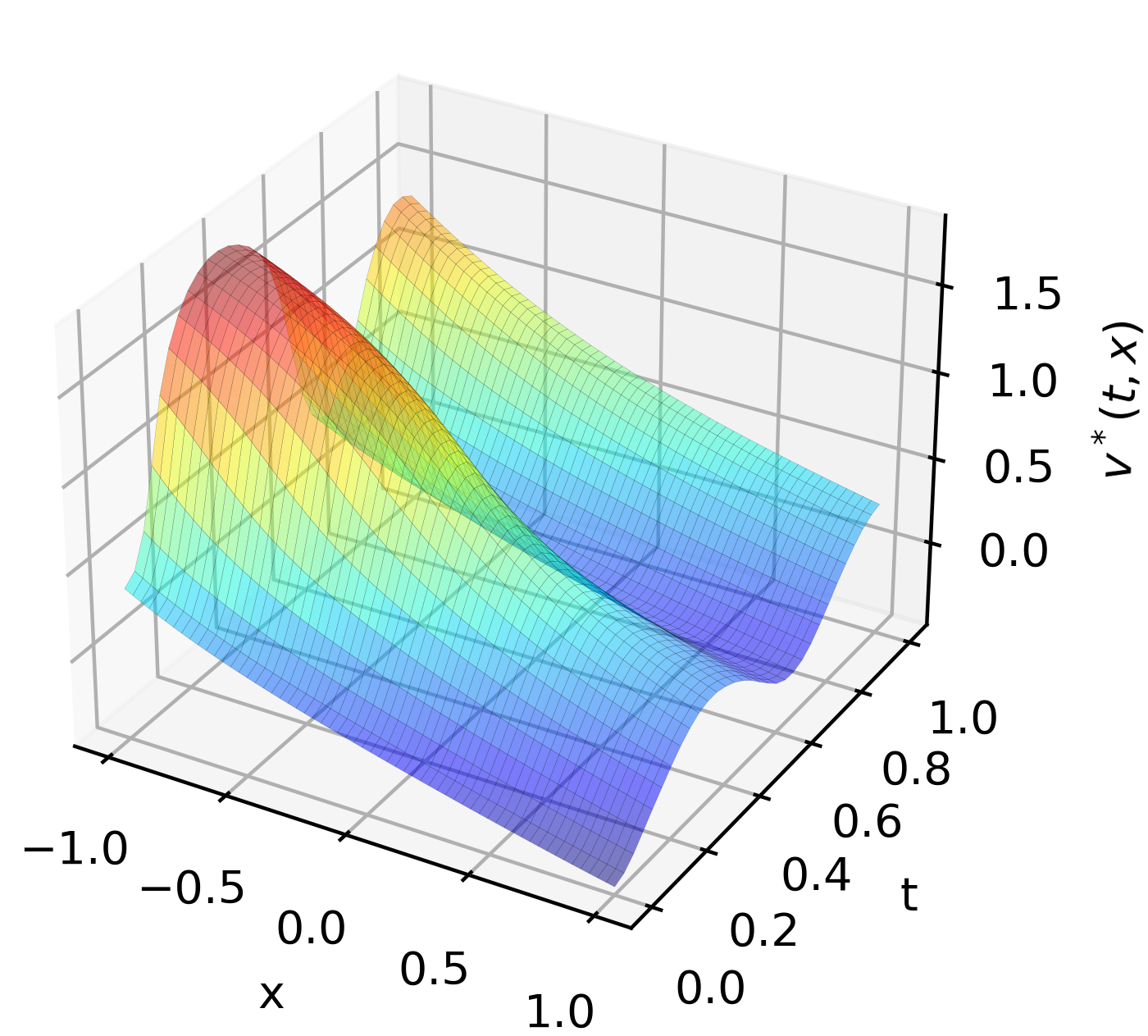}
     \end{subfigure}
	\hfill
     \begin{subfigure}[b]{0.34\textwidth}
         \centering
         \caption{Trajectories approx.}
         \includegraphics[width=\textwidth]{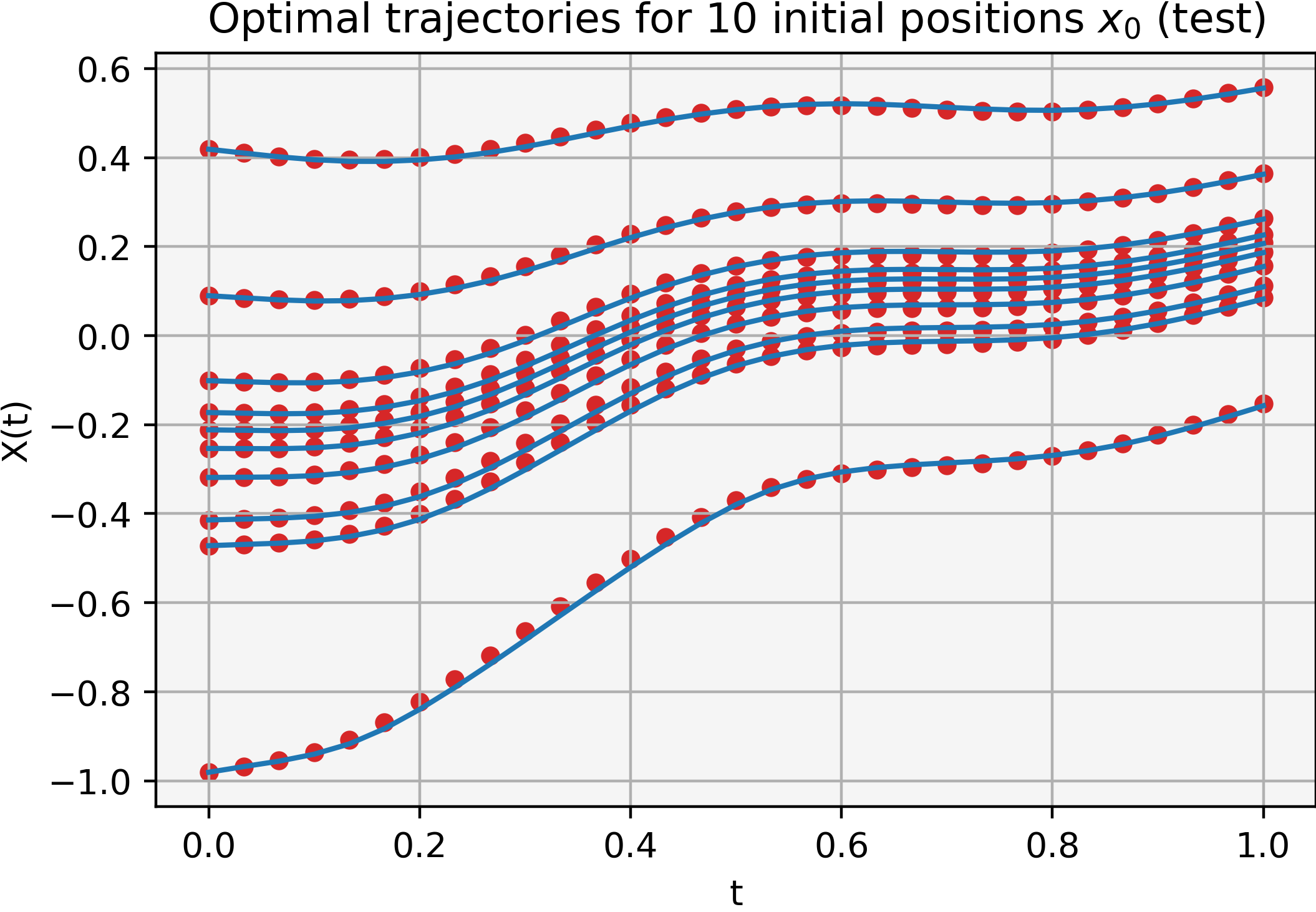}
     \end{subfigure}
	\hfill
     \begin{subfigure}[b]{0.34\textwidth}
         \centering
         \caption{Trajectories (error)}
         \includegraphics[width=\textwidth]{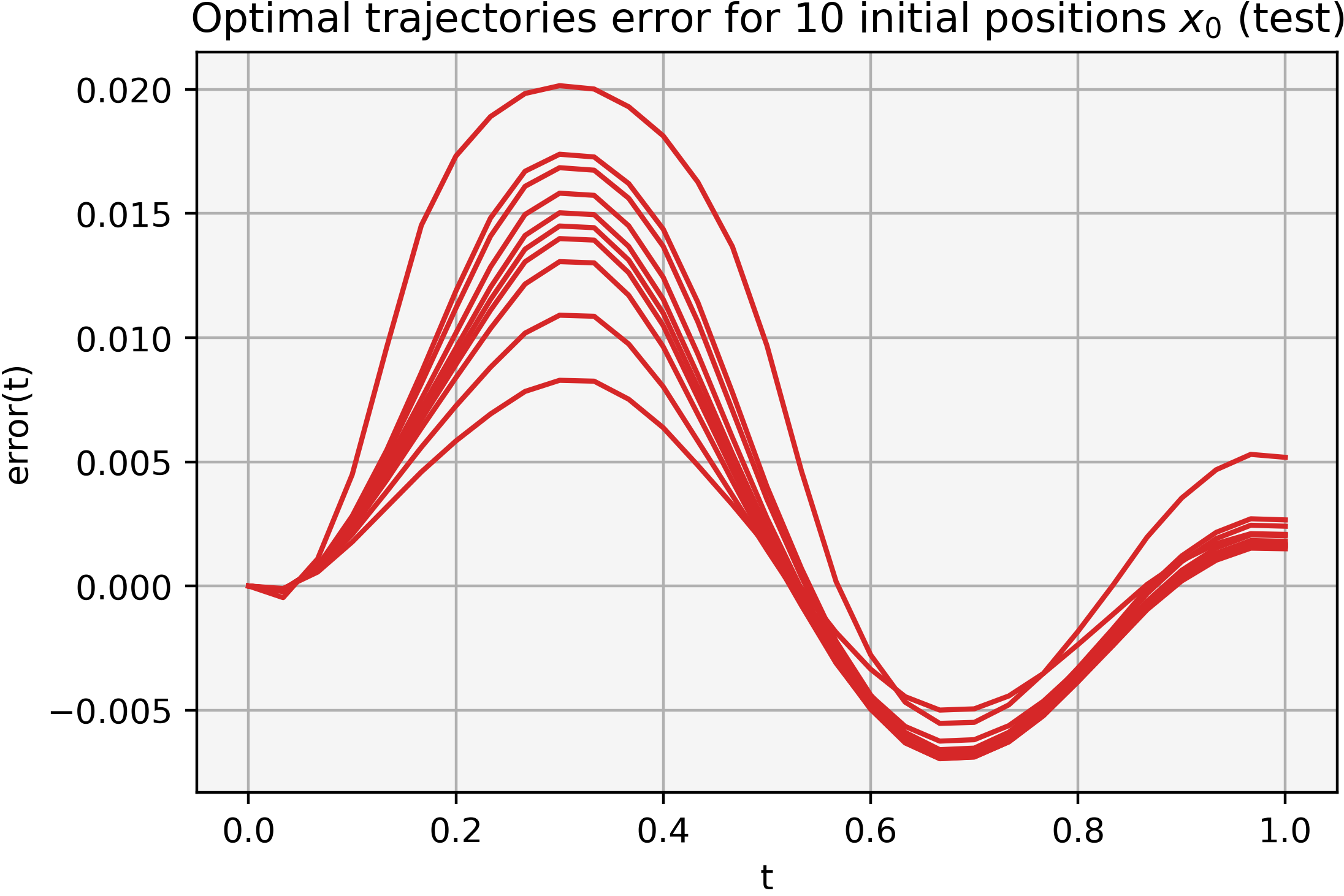}
     \end{subfigure}
        \caption{$\varpi$ (top) and $v^*$ (bottom) using $\mathrm{RNN}_\varpi$ and $\mathrm{RNN}_v$, respectively, for $\overline{Q}(t)=7 e^{-t} \sin(3 \pi t)$ with non-quadratic cost.}
        \label{fig:RNN NonQ Qbar osc results}
\end{figure}

Figure \ref{fig:RNN Qbar osc NonQ training} shows the a posteriori estimate with more consistent decay and less oscillatory behavior than the previous architecture. The error in the price approximation stabilizes around an order of $10^{-1}$, as it did in the previous configuration. As in the previous architecture, all components of error decay consistently. However, an earlier and less oscillatory decay is observed compared to the previous configuration.

\begin{figure}[htp]
     \centering
     \begin{subfigure}[b]{0.48\textwidth}
     \centering
        \caption{Residuals (training)}
         \includegraphics[width=\textwidth]{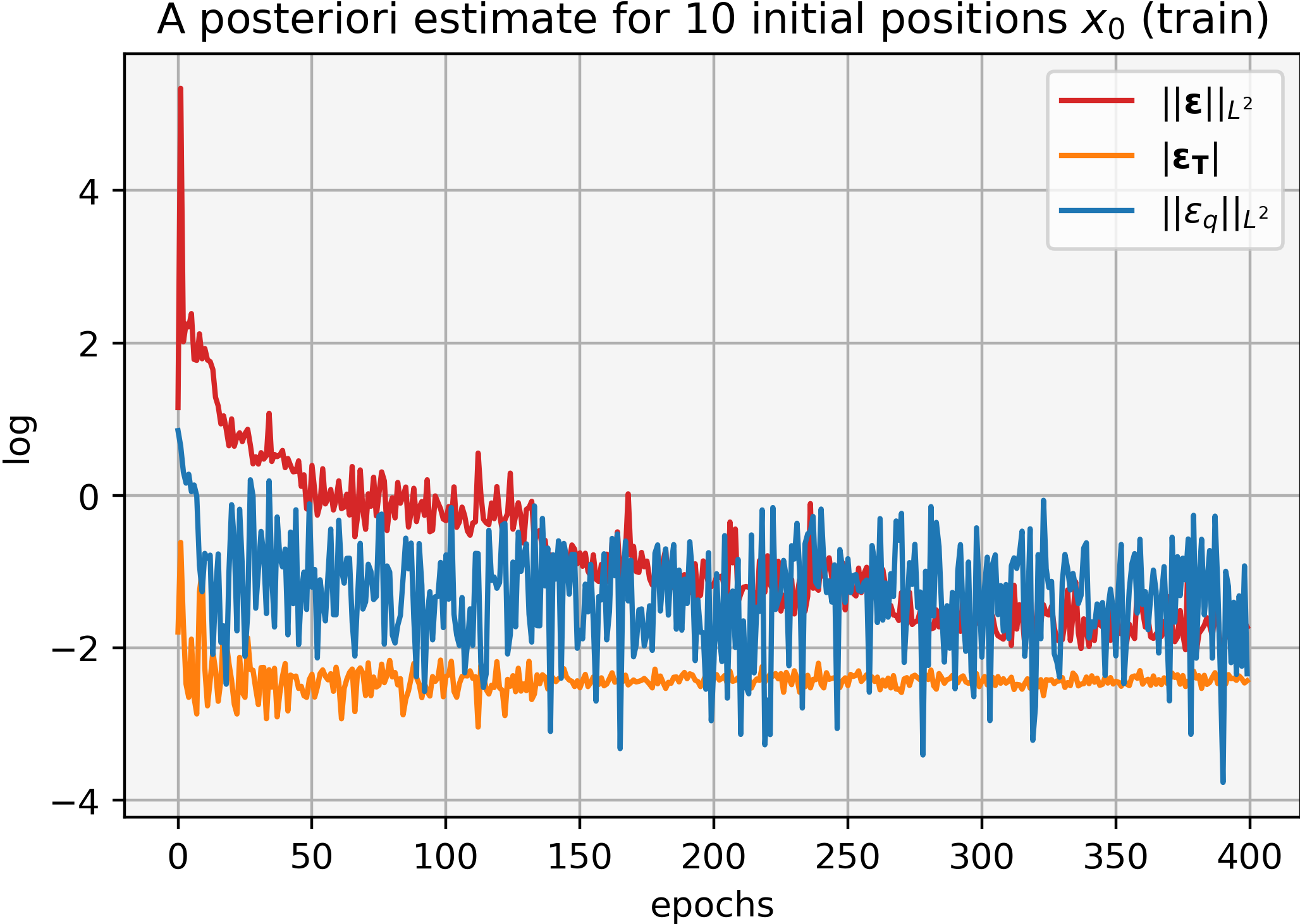}
          \end{subfigure}
     \hfill
     \begin{subfigure}[b]{0.48\textwidth}
     \centering
        \caption{Estimate vs. price error (training)}
         \includegraphics[width=\textwidth]{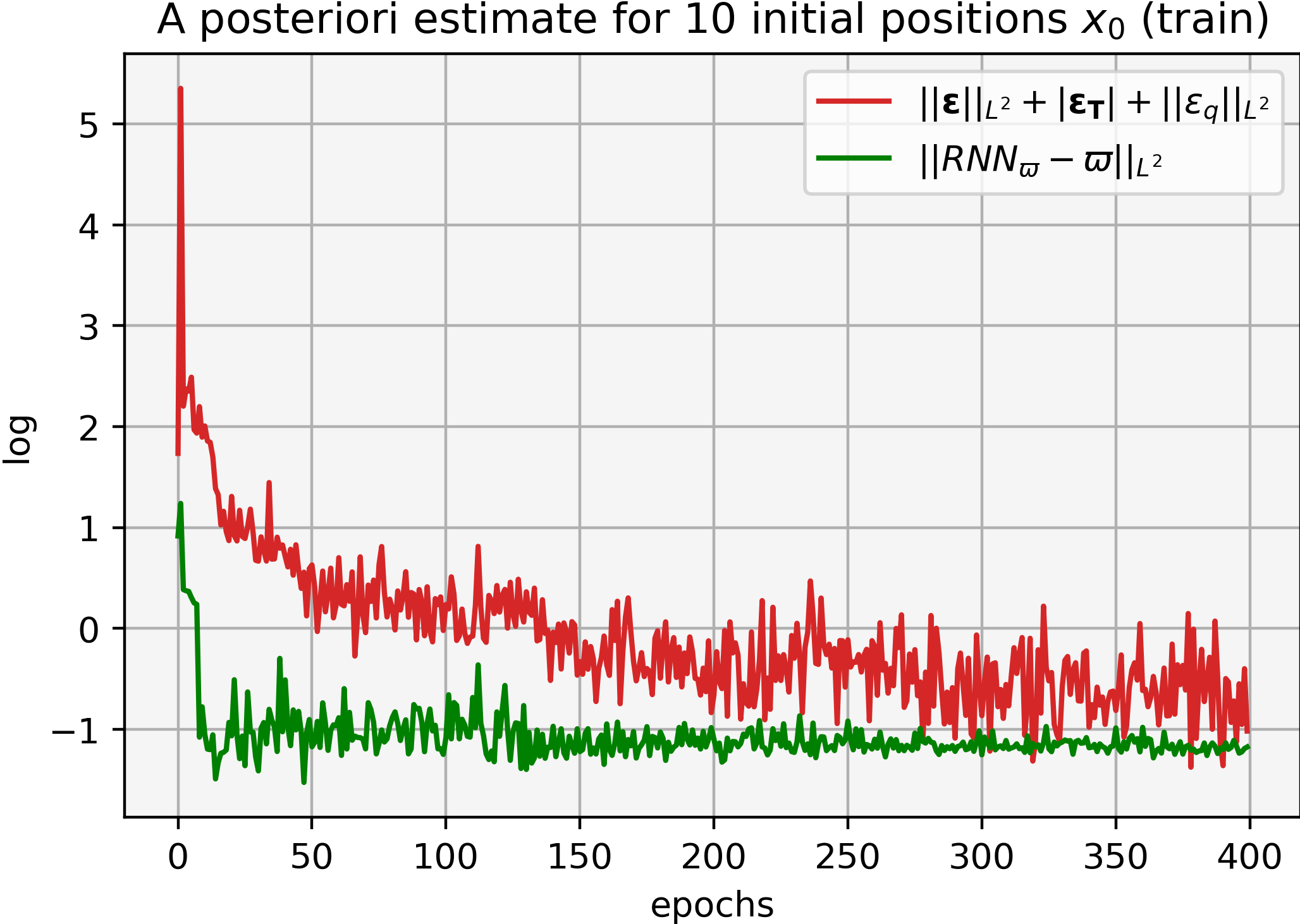}
          \end{subfigure}
         \caption{Residuals (left), a posteriori estimate \eqref{eq:EL discrete a posteriori} and price error (right) during training (log. scale) using $\mathrm{MLP}_\varpi$ and $\mathrm{MLP}_v$ for $\overline{Q}(t)=7 e^{-t} \sin(3 \pi t)$ with non-quadratic cost.}
        \label{fig:RNN Qbar osc NonQ training}
\end{figure}

{\bf Comparison.} The $RNN$ architecture achieves earlier and more consistent decays than the $MLP$ in the a posterior estimate. However, no significant difference is observed in the price approximation.

\section{Conclusions and further directions}

We examined the implementation of ML techniques to approximate the solution of a MFG price formation model. We formulated the training algorithm using a saddle point approach, which relies on the characterization of price existence as a Lagrange multiplier corresponding to a balance constraint. Instead of using a convergence proof for the ML algorithms, we adopt an alternative approach by using a posteriori estimates to assess the convergence of the training process without requiring the exact solution to be known. These estimates rely on the Euler-Lagrange equation characterizing optimal trajectories, which uniquely correspond to the solution of the price formation model being approximated. 

For the linear-quadratic case explored in Section \ref{sec: Numerical results}, the dependence on $m_0$ is through its mean, as \eqref{eq:LQ price formula} shows. This dependence can be obtained by sampling the initial positions according to $m_0$, as in Algorithm \ref{alg:our algorithm}. This feature of the linear-quadratic structure allows for efficiently approximating $\varpi$ through sampling with no requirement on taking $N\to \infty$, as long as the mean of $m_0$ is well approximated during training. Both architectures provide accurate price approximations for the constant and oscillatory mean supplies. In particular, the architecture with two RNN seems to offer more room for improvement as more parameters and/or training steps are added. 

Nonetheless, we used the same parameters for the RNN approximating the optimal control $v^*$ for both supply scenarios (constant and oscillating mean-reverting function for the supply), we observe that the NN can capture the oscillatory behavior with no further requirement in the number of parameters, such as the number of layers and neurons. Therefore, the ML framework provides a reliable approximation capability for real scenarios, for instance, in the case of energy supply, whose demand is characterized by peak loads. Finally, the characteristic curves are well-defined up to the terminal time in the linear-quadratic model. It would be interesting to study the case of shocks, for instance, when players aggregate in a single position creating a singular distribution. 

Adding the supply history dependence in the RNN did not improve the approximation results. However, this architecture offers a history dependence that is required in the case of stochastic supply. As shown in \cite{GoGuRi2021} and \cite{gomes2021randomsupply}, when the supply is random, this introduces a noise that all players perceive; that is, the price formation model with common noise. In that case, a strategy to obtain a price is to use a SDE characterization, which requires the process to be progressively measurable. The RNN architecture meets this condition in the discrete-time approximation. We plan to pursue this line of research in future works.

In our method to solve the price formation MFG,  the optimal control $v^*$ and the price $\varpi$ play the role of primal and dual variables, respectively, as explained in Section \ref{sec:training algorithm}. For variational MFG, the PDE system can be seen as the optimality condition of a variational problem with the continuity equation as a constraint. Variational MFGs admit a dual formulation that has been exploited to propose numerical schemes (see \cite{MFGNumerical}, \cite{BenamouBrenierVarMFG}, and \cite{LauriereEtAlSurvey}). For instance, the Alternating Direction Method of Multipliers and the primal-dual method, as introduced in Section \ref{sec:training algorithm}, are used to solve variational MFGs; the main advantage is that convergence proofs are well understood. Alternatively, the primal-dual method can be implemented using NN, as we proposed in Section \ref{sec:NN for price}, where one NN plays the role of the primal variable and another NN plays the role of the dual variable.


\bibliographystyle{acm}

\bibliography{mfg_filtered.bib}
\end{document}